\documentclass[11pt]{article}   
\usepackage{amsthm,amsfonts,amsmath,amscd,amssymb,latexsym}    
\usepackage{mathrsfs}
\usepackage{epsfig,graphicx,color}     
\usepackage{color}    
\usepackage[all]{xy}  
\usepackage[titles]{tocloft}  
\title{Hypercontact structures and Floer homology}   
 
\author{  
Sonja Hohloch \\ ETH Z\"urich
\and
Gregor Noetzel \\ Universit\"at Leipzig
\and
Dietmar~A. Salamon \\
ETH Z\"urich}  
  
\date{29 July 2008}   

\newtheorem{PARA}{}[section]    
\newtheorem{theorem}[PARA]{Theorem}    
\newtheorem{corollary}[PARA]{Corollary}    
\newtheorem{lemma}[PARA]{Lemma}    
\newtheorem{proposition}[PARA]{Proposition}    
\newtheorem{definition}[PARA]{Definition}    
\newtheorem{remark}[PARA]{Remark}    
\newtheorem{example}[PARA]{Example}   
\newcommand{\para}{\begin{PARA}\rm}    
\newcommand{\arap}{\end{PARA}\rm}    
\newcommand{\dfn}{\begin{definition}\rm}    
\newcommand{\nfd}{\end{definition}\rm}    
\newcommand{\rmk}{\begin{remark}\rm}    
\newcommand{\kmr}{\end{remark}\rm}    
\newcommand{\xmpl}{\begin{example}\rm}    
\newcommand{\lpmx}{\end{example}\rm}    
\newcommand{\sA}{\mathscr{A}}    
\newcommand{\sB}{\mathscr{B}}    
\newcommand{\sC}{\mathscr{C}}    
\newcommand{\sD}{\mathscr{D}}    
\newcommand{\sE}{\mathscr{E}}    
\newcommand{\sF}{\mathscr{F}}    
    
\newcommand{\sH}{\mathscr{H}}

\newcommand{\sL}{\mathscr{L}}    
\newcommand{\sM}{\mathscr{M}}

\newcommand{\sR}{\mathscr{R}}

\newcommand{\sU}{\mathscr{U}}

\newcommand{\sZ}{\mathscr{Z}}    
\newcommand{\tsF}{\widetilde{\mathscr{F}}}   

\newcommand{\cD}{\mathcal{D}}    
\newcommand{\cF}{\mathcal{F}}

\newcommand{\cH}{\mathcal{H}}

\newcommand{\cL}{\mathcal{L}}

\newcommand{\cP}{\mathcal{P}}

\newcommand{\C}{{\mathbb{C}}}    
\newcommand{\D}{{\mathbb{D}}}    
    
\renewcommand{\H}{{\mathbb{H}}}

\newcommand{\N}{{\mathbb{N}}}    
    
\newcommand{\R}{{\mathbb{R}}}    
\newcommand{\T}{{\mathbb{T}}}    
\newcommand{\W}{{\mathbb{W}}}    
\newcommand{\Z}{{\mathbb{Z}}}    
\newcommand{\CF}{{\mathrm{CF}}}    
\newcommand{\HF}{{\mathrm{HF}}}    
\newcommand{\SU}{{\mathrm{SU}}}    
\newcommand{\Sp}{{\mathrm{Sp}}}    
\newcommand{\RE}{{\mathrm{Re}}}    
\newcommand{\IM}{{\mathrm{Im}}}    
\newcommand{\Map}{{\mathrm{Map}}}    
\newcommand{\Vol}{{\mathrm{Vol}}}

\newcommand{\morse}{{\mathrm{morse}}}    

\newcommand{\tu}{{\widetilde u}}

\renewcommand{\i}{{\mathbf{i}}}    
\renewcommand{\j}{{\mathbf{j}}}    
\renewcommand{\k}{{\mathbf{k}}}    
\newcommand{\one}    
{{{\mathchoice \mathrm{ 1\mskip-4mu l} \mathrm{ 1\mskip-4mu l}    
\mathrm{ 1\mskip-4.5mu l} \mathrm{ 1\mskip-5mu l}}}}    
\def\slashi#1{\rlap{\sl/}#1}
\def\slashii#1{\setbox0=\hbox{$#1$}             
\dimen0=\wd0                                 
\setbox1=\hbox{\sl/} \dimen1=\wd1            
\ifdim\dimen0>\dimen1                        
\rlap{\hbox to \dimen0{\hfil\sl/\hfil}}   
#1                                        
\else                                        
\rlap{\hbox to \dimen1{\hfil$#1$\hfil}}   
\hbox{\sl/}                               
\fi}                                         %
\def\slashiii#1{\setbox0=\hbox{$#1$}#1\hskip-\wd0\hbox to\wd0{\hss\sl/\/\hss}}
\newcommand{\dd}{{\slashi{\partial}}}

\newcommand{\dcD}{{\slashiii{\cD}}}

\renewcommand{\div}{\mathrm{ div}}  
\newcommand{\dvol}{\mathrm{ dvol}}  
\newcommand{\trace}{\mathrm{ trace }}  
             
\newcommand{\g}{{\mathfrak{g}}}    
    
\newcommand{\G}{{\mathrm{G}}}     
\newcommand{\Lie}{{\mathrm{Lie}}}     
    
\newcommand{\grad}{\mathrm{ grad }}    
\newcommand{\id}{\mathrm{ id}}         
    
\newcommand{\im}{\mathrm{ im }}        
\newcommand{\INT}{\mathrm{ int}}       
     
\newcommand{\Hom}{\mathrm{Hom}}       %

\newcommand{\PSL}{{\mathrm{PSL}}}             
             
\newcommand{\rank}{\mathrm{ rank }}    
\newcommand{\SL}{{\mathrm{SL}}}             
\newcommand{\Spin}{\mathrm{ Spin}}       
\newcommand{\Diff}{\mathrm{ Diff}}        
\newcommand{\Vect}{\mathrm{ Vect}}        
\newcommand{\eps}{{\varepsilon}}    
\newcommand{\om}{{\omega}}    
\newcommand{\Om}{{\Omega}}    
\newcommand{\Cinf}{C^{\infty}}    
\newcommand{\reg}{\mathrm{ reg}}    

\newcommand{\inner}[2]{\bigl\langle #1, #2\bigr\rangle}

\def\NABLA#1{{\mathop{\nabla\kern-.5ex\lower1ex\hbox{$#1$}}}}    
\def\Nabla#1{\nabla\kern-.5ex{}_{#1}}    
\def\Tabla#1{\Tilde\nabla\kern-.5ex{}_{#1}}    
\def\abs#1{\mathopen|#1\mathclose|}    
\def\Abs#1{\left|#1\right|}    
\def\norm#1{\mathopen\|#1\mathclose\|}    
\def\Norm#1{\left\|#1\right\|}    
\renewcommand{\Tilde}{\widetilde}

\newcommand{\p}{{\partial}}

\begin{document}    
    
\maketitle    
  
    
\begin{abstract}   
We introduce a new Floer theory associated to a pair consisting 
of a Cartan hypercontact $3$-manifold~$M$ and a hyperk\"ahler manifold~$X$.
The theory is a based on the gradient flow of the hypersymplectic action 
functional on the space of maps from $M$ to $X$.  The gradient flow lines
satisfy a nonlinear analogue of the Dirac equation and can also be viewed
as hyperk\"ahler analogues of holomorphic curves. We work out the 
details of the analysis and compute the Floer homology groups 
in the case where $X$ is flat.  As a corollary we derive an existence 
theorem for the $3$-dimensional perturbed nonlinear Dirac equation 
which can be viewed as an analogue of the Arnold conjecture. 
\end{abstract}

    
\section{Introduction} \label{sec:intro}  

In this paper we examine a hyperk\"ahler analogue of
symplectic Floer homo\-logy. We assume throughout that 
$X$ is a hyperk\"ahler manifold with complex structures 
$I,J,K$ and symplectic forms $\om_1,\om_2,\om_3$.  We also assume  
that $M$ is a compact oriented $3$-manifold equipped with a volume form 
$\sigma\in\Om^3(M)$ and a positive frame $v_1,v_2,v_3\in\Vect(M)$
of the tangent bundle. Associated to these data is a natural 
$1$-form on the space $\sF:=\Cinf(M,X)$ of smooth functions
$f:M\to X$ defined by 
\begin{equation}\label{eq:action-form}
\hat f\mapsto \int_M
\Bigl(
\om_1(\p_{v_1}f,\hat f)
+ \om_2(\p_{v_2}f,\hat f)
+ \om_3(\p_{v_3}f,\hat f) 
\Bigr)\,\sigma
\end{equation}
for $\hat f\in T_f\sF=\Om^0(M,f^*TX)$.
This $1$-form is closed if and only if the vector fields $v_i$
are volume preserving, i.e. 
$$
\cL_{v_1}\sigma = \cL_{v_2}\sigma = \cL_{v_3}\sigma = 0.
$$
Our two main examples are the $3$-torus with the coordinate vector fields 
and the $3$-sphere with the standard hypercontact structure. 

\subsection*{Hypercontact structures}

A {\it hypercontact structure} on a $3$-manifold $M$
is a triple of contact forms 
$
\alpha=(\alpha_1,\alpha_2,\alpha_3)\in\Om^1(M,\R^3)
$ 
such that
$$
\alpha_1\wedge d\alpha_1
= \alpha_2\wedge d\alpha_2
= \alpha_3\wedge d\alpha_3
=: \sigma 
$$
and $\alpha_i\wedge d\alpha_j+\alpha_j\wedge d\alpha_i=0$ 
for $i\ne j$. The Reeb vector fields 
$v_1,v_2,v_3$ are pointwise linearly independent
and preserve the volume form $\sigma$. 
The hypercontact structure is called {\it positive} 
if they form a positive frame of the tangent bundle. 
In this setting the $1$-form~\eqref{eq:action-form}
is the differential of the action functional
$\sA:\sF\to\R$ defined by 
\begin{equation}\label{eq:Action}
\sA(f) := - \int_M \Bigl(\alpha_1\wedge f^*\om_1 
+ \alpha_2\wedge f^*\om_2 + \alpha_3\wedge f^*\om_3  \Bigr).
\end{equation}
A positive hypercontact structure is called a 
{\it Cartan structure} if the $\alpha_i$ form a dual 
frame of the cotangent bundle, i.e.\ $\alpha_i(v_j)=\delta_{ij}$. 
In the Cartan case 
$\kappa:=d\alpha_1(v_2,v_3)=d\alpha_2(v_3,v_1)=d\alpha_3(v_1,v_2)$
is constant and $d\alpha_i=\kappa \alpha_j\wedge\alpha_k$ and 
$[v_i,v_j]=\kappa v_k$ for every cyclic permutation $i,j,k$ of $1,2,3$. 
(We use the sign convention of~\cite{MS} for the Lie bracket.)

The archetypal example is the $3$-sphere $M=S^3$, understood 
as the unit quaternions, with $v_1(y)=\i y$, $v_2(y)=\j y$, $v_3(y)=\k y$. 
Hypercontact structures were introduced by Geiges--Gonzalo~\cite{GG,GG1}.  
They use the term taut contact sphere for what we call 
a hypercontact structure.  They proved that every Cartan hypercontact 
$3$-manifold is diffeomorphic to a quotient of the 
$3$-sphere by the right action of a finite subgroup of $\Sp(1)$. 

\subsection*{Tori}

Let $M=\T^3=\R^3/\Z^3$ be the standard $3$-torus equipped with the 
standard volume form $\sigma=dt_1\wedge dt_2\wedge dt_3$ and
$v_i=\sum_{j=1}^3a_{ij}\p_j$ where $A=(a_{ij})_{i,j=1}^3$ is 
a nonsingular real $3\times3$ matrix.  In this case 
the lift of the $1$-form~\eqref{eq:action-form} to the universal 
cover $\tsF$ of $\sF$ is the differential of the function
\begin{equation}\label{eq:ActionT}
\sA = \sum_{i,j=1}^3a_{ij}\sA_{ij}:\tsF\to\R
\end{equation}
where $\sA_{ij}(f)$ denotes the $\om_i$-symplectic action 
of the loop $t_j\mapsto f(t)$, averaged over the remaining 
two variables $t_k,t_\ell$ with $k,\ell\ne j$. If $X$ is flat 
and $\sF_0\subset\sF$ denotes the space of contractible maps 
${f:\T^3\to X}$ then $\sA$  descends to $\sF_0$.  Explicitly, 
we have 
$
\sA_{ij}(f) := -\int_0^1\int_0^1\int_\D u_{t_k,t_\ell}^*\om_i\,dt_k\,dt_\ell
$
for $f\in\sF_0$, where $u_{t_k,t_\ell}:\D\to X$ is a smooth family of maps
satisfying $u_{t_k,t_\ell}(e^{2\pi i t_j})=f(t_1,t_2,t_3)$.

\subsection*{Hyperbolic spaces}

A third class of examples arises from unit tangent bundles of higher
genus surfaces or equivalently from quotients of the group 
$\G:=\PSL(2;\R)$.  Let $\cH\subset\C$ denote the upper half plane
and 
$
\cP:=\left\{(z,\zeta)\in\C^2\,|\,\IM(z)=\Abs{\zeta}\right\}
$
the unit tangent bundle of~$\cH$.  The group $\G$
acts freely and transitively on $\cP$ by 
$$
g_*(z,\zeta) := \left(\frac{az+b}{cz+d},\frac{\zeta}{(cz+d)^2}\right),\qquad
g=:\left(\begin{array}{cc} a & b \\ c & d\end{array}\right)\in\SL(2;\R)..
$$
Now let $\Gamma\subset\PSL(2;\R)$ be a discrete subgroup acting 
freely on $\cH$ such that the quotient 
$\Sigma:=\Gamma\backslash\cH$ is a closed 
Riemann surface.  Then the $3$-manifold
$$
M := \Gamma\backslash\G
$$
is diffeomorphic to the unit tangent bundle 
$T_1\Sigma=\Gamma\backslash\cP$ via $[g]\mapsto[g_*(i,1)]$.
The group $\G$ carries a natural bi-invariant volume form 
$\sigma\in\Om^3(\G)$ given by
$$
\sigma(g\xi,g\eta,g\zeta) := \frac12\trace([\xi,\eta]\zeta)
$$
for $\xi,\eta,\zeta\in\g:=\Lie(\G)=\mathfrak{sl}(2;\R)$.
This volume form descends to $M$ and is invariant under 
the right action of $\G$. Now consider the traceless matrices 
$$
\xi_1 := \left(\begin{array}{rr} 1 & 0 \\ 0 & -1\end{array}\right),\qquad
\xi_2 := \left(\begin{array}{rr} 0 & 1 \\ 1 & 0\end{array}\right),\qquad
\xi_3 := \left(\begin{array}{rr} 0 & -1 \\ 1 & 0\end{array}\right).
$$
The resulting vector fields $v_i(g):=g\xi_i$ on $\G$ 
are $\Gamma$-equivarient and preserve the volume form $\sigma$.
Hence they descend to volume preserving vector fields
on $M$ (still denoted by $v_i$) and so the 
$1$-form~\eqref{eq:action-form} is closed in this setting.

Note that
$
\sigma(v_1,v_2,v_3) = 2
$
and $d\pi(v_3)=0$, $d\pi(v_1)=id\pi(v_2)$. 
The Lie brackets  of  the vector fields $v_i$ are given by
$$
[v_2,v_3] = -2v_1,\qquad [v_3,v_1]=-2v_2,\qquad [v_1,v_2]=2v_3
$$
(because the $\xi_i$ act on $\G$ on the right). Hence, if
$\alpha_i\in\Om^1(M)$ denote the $1$-forms dual to 
the vector fields $v_i$, we have
$$
d\alpha_1 = -2\alpha_2\wedge\alpha_3,\qquad
d\alpha_2 = -2\alpha_3\wedge\alpha_1,\qquad
d\alpha_3 = 2\alpha_1\wedge\alpha_2.
$$
This implies that the $1$-form~\eqref{eq:action-form} is the differential 
of the action functional 
$$
\sA(f) := \int_M\left(\alpha_1\wedge f^*\om_1+\alpha_2\wedge f^*\om_2
-\alpha_3\wedge f^*\om_3\right).
$$
However, in this setting the energy identity~\eqref{eq:energy-crit} 
discussed below does not help in the compactness proof.  
This is the reason why we do not include the higher genus case 
in our discussion in the main part of this paper. 

\subsection*{Floer theory}

The zeros of the $1$-form~\eqref{eq:action-form} are the solutions 
$f:M\to X$ of the nonlinear elliptic first order partial differential equation 
\begin{equation}\label{eq:crit}
\dd(f) := I\p_{v_1}f + J\p_{v_2}f+K\p_{v_3}f = 0.
\end{equation}
This is a nonlinear analogue of the Dirac equation
that was first introduced by Taubes~\cite{T}.
Obviously, the constant functions are solutions of~\eqref{eq:crit}.
When $M=S^3$ other solutions arise from the composition of 
rational curves with suitable Hopf fibrations (see below). 
When $M=\T^3$ solutions can be obtained 
from elliptic curves. In the case $M=\Gamma\backslash\G$ 
solutions arise from the composition of $K$-holomorphic curves 
$\Sigma\to X$ with $\pi:M\to\Sigma$.   

In this paper we prove an existence result for the 
solutions of the perturbed nonlinear Dirac equation
\begin{equation}\label{eq:crit-H}
\dd_H(f) := I\p_{v_1}f + J\p_{v_2}f+K\p_{v_3}f - \nabla H(f) = 0.
\end{equation}
Here $H:X\times M\to\R$ is a smooth function and we denote by 
$\nabla H(f)$ the gradient with respect to the first argument.
The linearized operator for this equation is self adjoint and we call 
a solution $f:M\to X$ of~\eqref{eq:crit-H} {\it nondegenerate}
if the linearized operator is bijective. In the nondegenerate case, 
and when $X$ is flat, one can count the solutions with signs, 
however, it turns out that this count gives zero. Nevertheless we shall 
prove the following hyperk\"ahler analogue of the Conley-Zehnder theorem 
confirming the Arnold conjecture for the torus~\cite{CZ}. 
In fact, in the torus case with $v_1=\p/\p t_1$ the solutions 
of~\eqref{eq:crit} can be interpreted as the periodic orbits 
of a suitable infinite dimensional Hamiltonian system. 

\medskip\noindent{\bf Theorem~A.}
{\it Let $M$ be either a compact Cartan hypercontact 
$3$-manifold (with Reeb vector fields $v_i$) or 
the $3$-torus (with a constant frame~$v_i$). 
Let $X$ be a compact flat hyperk\"ahler manifold. 
Then the space of solutions of~\eqref{eq:crit-H} 
is compact.  Moreover, if the contractible 
solutions are all nondegenerate, then their 
number is bounded below by the sum of the 
$\Z_2$-Betti numbers of $X$.  In particular, 
equation~\eqref{eq:crit-H} has a contractible 
solution for every $H$.}

\medskip\noindent
The proof of Theorem~A is based on the observation that the 
solutions of~\eqref{eq:crit-H} are the critical points of the 
perturbed {\it hypersymplectic action functional}
$\sA_H(f) := \sA(f) - \int_MH(f) \sigma$.
As in symplectic Floer theory, this functional is 
unbounded above and below, and the Hessian has infinitely many positive 
and negative eigenvalues.  Thus the standard techniques of Morse 
theory are not available for the study of the critical points. 
However, with appropriate modifications, the familiar techniques 
of Floer homology carry over to the present case,
at least when~$X$ is flat, and thus give rise to 
natural Floer homology groups for a pair~$(M,X)$.

The Floer homology groups are determined by a chain 
complex that is generated by the solutions of~\eqref{eq:crit-H}.
The boundary operator is determined by the finite energy 
solutions ${u:\R\times M\to X}$ of the negative gradient flow equation
\begin{equation}\label{eq:floer-H}
\p_su + I\p_{v_1}u + J\p_{v_2}u+K\p_{v_3}u = \nabla H(u).
\end{equation}
One of the key ingredients in the compactness proof is the energy
identity
\begin{equation}\label{eq:energy-crit}
\frac12\int_M\Abs{df}^2 
= \frac12\int\Abs{I\p_{v_1}f + J\p_{v_2}f+K\p_{v_3}f}^2
- \int_M\sum_{i=1}^3\eps_i\wedge f^*\om_i
\end{equation}
for $f:M\to X$, where the $\eps_i\in\Om^1(M)$ 
are dual to the vector fields $v_i$.  In the torus case these forms 
are closed and thus the last term in~\eqref{eq:energy-crit} is a 
topological invariant.  In the Cartan hypercontact case this 
term is the hypersymplectic action $\sA(f)$. 

To compute the Floer homology groups we choose a Morse--Smale 
function ${H:X\to\R}$ and study the equation
\begin{equation}\label{eq:floer-Heps}
\p_su + \eps^{-1}\left(I\p_{v_1}u + J\p_{v_2}u+K\p_{v_3}u\right) 
= \nabla H(u)
\end{equation}
for small values of $\eps$.  The gradient lines of $H$ are solutions of 
this equation and we shall prove that, for $\eps>0$ sufficiently small, there 
are no other contractible solutions.   This implies that our Floer 
homology groups $\HF_*(M,X)$ are isomorphic to the singular 
homology $H_*(X;\Z_2)$.

\medskip\noindent{\bf Theorem~B.}
{\it Let $M$ be either a compact Cartan hypercontact 
$3$-manifold (with Reeb vector fields $v_i$)
or the $3$-torus (with a constant frame~$v_i$).  
Let $X$ be a compact flat hyperk\"ahler manifold 
and fix a class ${\tau\in\pi_0(\sF)}$. 
Then, for a generic perturbation $H:X\times M\to\R$, 
there is a natural Floer homology group $\HF_*(M,X,\tau;H)$
associated to a chain complex generated by 
the solutions of~\eqref{eq:crit-H} where the 
boundary operator is defined by counting the 
solutions of~\eqref{eq:floer-H}.
The Floer homology groups associated to 
different choices of $H$ are naturally 
isomorphic. Moreover, for the component $\tau_0$ of the
constant maps there is a natural isomorphism
$\HF_*(M,X,\tau_0;H)\cong H_*(X;\Z_2)$.}

\medskip\noindent{\bf Remark.}
The precise condition we need for extending the standard techniques
of Floer theory to our setting is that $X$ has nonpositive sectional
curvature. As every hyperk\"ahler manifold has vanishing 
Ricci tensor, nonpositive sectional curvature implies that
$X$ is flat and hence is a quotient of a hyperk\"ahler torus 
by a finite group.  An example is the quotient of the standard 
$12$-torus $\H^3/\Z^{12}$ by the $\Z_2$-action 
determined by $(x,y,z)\mapsto(y,x,z+1/2)$. 

\subsection*{A more general setting}

There is conjecturally a much richer theory which 
provides Floer homological invariants for all triples $(M,X,\tau)$,
consisting of a Cartan hypercontact $3$-manifold~$M$, a compact
hyperk\"ahler manifold~$X$, and a homotopy class $\tau$ 
of maps from $M$ to $X$.  One basic observation is 
that every holomorphic sphere in a hyperk\"ahler manifold gives 
rise to a solution of~\eqref{eq:crit} on $M=S^3$.  
Another point is that $\pi_3(X)$ can be a very rich group.  For example, 
the third homotopy group of the K3-surface has 253 generators
(see~\cite[Appendix]{CH}).

\medskip\noindent{\bf Example.}
Think of the $3$-sphere as the unit sphere in the quaternions 
$\H\cong\R^4$ and of the $2$-sphere as the unit sphere 
in the imaginary quaternions $\IM(\H)\cong\R^3$. 
For $\lambda=\lambda_1\i+\lambda_2\j+\lambda_3\k\in S^2$ 
denote 
$
J_\lambda:=\lambda_1I+\lambda_2J+\lambda_3K
$
and
$
\om_\lambda=\lambda_1\om_1+\lambda_2\om_2+\lambda_3\om_3.
$
Define $h_\lambda:S^3\to S^2$ by 
$h_\lambda(y) := -\bar y\lambda y$.
If $u:S^2\to X$ is a $J_\lambda$-ho\-lo\-mor\-phic sphere then 
$$
f:=u\circ h_\lambda:S^3\to X
$$ 
is a critical point of $\sA$ and 
$$
E(u) = \frac12\int_{S^2}\Abs{du}^2
= \int_{S^2}u^*\om_\lambda
= \frac{1}{2\pi}\sA(u\circ h_\lambda).
$$ 
To see this, assume ${\lambda=\i}$ and write ${h_1(y):=-\bar y\i y}$,  
${h_2(y):=-\bar y\j y}$, and ${h_3(y):=-\bar y\k y}$.  
These functions satisfy $\p_{v_i}h_i=0$ and 
$\p_{v_j}h_i=-\p_{v_i}h_j=2h_k$
for every cyclic permutation $i,j,k$ of $1,2,3$. 
Hence $h_1\wedge\p_{v_3}h_1 = \p_{v_2}h_1$. 
If $u:S^2\to X$ is an $I$-holomorphic sphere it follows that 
the function $f := u\circ h_1$ satisfies 
$\p_{v_1}f=0$ and $I\p_{v_3}f=\p_{v_2}f$
and hence is a solution of~\eqref{eq:crit}.  
Moreover, $2\pi\int_{S^2}\sigma 
= -\int_{S^3}\alpha_1\wedge h_1^*\sigma$
for $\sigma\in\Om^2(S^2)$. 
(When $\sigma$ is exact both sides are zero. Since
$-\alpha_1\wedge h_1^*\dvol_{S^2}=4\dvol_{S^3}$
the value of the factor follows from 
$\Vol(S^2)=4\pi$ and $\Vol(S^3)=2\pi^2$.)
With $\sigma=u^*\om_1$ this implies 
$2\pi\int_{S^2}u^*\om_1 
= - \int_{S^2}\alpha_1\wedge h_1^*u^*\om_1
= \sA(u\circ h_1)$. 
Here the last equation follows from the fact that $u^*\om_2=u^*\om_3=0$
for every $I$-holomorphic curve~$u$. 

\medskip\noindent
The main technical difficulty in setting up the Floer theory for general 
hyperk\"ahler manifolds is to establish a suitable compactness theorem.  
In contrast to the familiar theory the derivatives for a sequence 
of solutions of~\eqref{eq:crit-H} or~\eqref{eq:floer-H}
will not just blow up at isolated points but along codimension-$2$ 
subsets.  For example, if $u_\nu:S^2\to X$ is a sequence of $I$-holomorphic
curves and $h:S^3\to S^2$ is a suitable Hopf fibration, then 
$f_\nu:=u_\nu\circ h$ is a sequence of solutions of~\eqref{eq:crit} 
and its derivatives blow up along the Hopf circle ${h^{-1}(z_0)}$ 
whenever the derivatives of $u_\nu$ blow up near $z_0$.  
This phenomenon is analogous to the codimension 
$4$ bubbling in Donaldson--Thomas theory~\cite{DT}. 

\subsection*{Floer--Donaldson theory}

Let $\Sigma$ be a hyperk\"ahler $4$-manifold with complex structures
$\i,\j,\k$ and symplectic forms $\sigma_1,\sigma_2,\sigma_3$. 
Consider the elliptic partial differential equation
\begin{equation}\label{eq:Hhol}
du - I du\i - J du\j - Kdu\k = 0
\end{equation}
for smooth maps $u:\Sigma\to X$.  This is sometimes called 
the Cauchy--Riemann--Fueter equation and it has been widely 
studied (see~\cite{T}, \cite[Chapter~3]{H} and references).
For $\Sigma=\R\times M$ with its standard hyperk\"ahler 
structure (see below) equation~\eqref{eq:Hhol}
is equivalent to~\eqref{eq:floer-H} with $H=0$.
The solutions of~\eqref{eq:Hhol} satisfy 
the energy identity
\begin{equation}\label{eq:Henergy}
E(u)
= \frac18\int_\Sigma\Abs{du - I du\i - J du\j - Kdu\k}^2\,\dvol_\Sigma
- \int_\Sigma\sum_{i=1}^3\sigma_i\wedge u^*\om_i,
\end{equation}
where $E(u):=\frac12\int_\Sigma\Abs{du}^2\,\dvol_\Sigma$. 
The linearized operator 
$$
\sD_u:\Om^0(\Sigma,u^*TX)\to \Om^1_\H(\Sigma,u^*TX)
$$
takes values in the space of $1$-forms on $\Sigma$ with values in $u^*TX$
that are complex linear with respect to $I$, $J$, and $K$.  When $\Sigma$
is closed this operator is Fredholm between appropriate Sobolev completions 
and its index is
\begin{equation}\label{eq:index}
\mathrm{ind}(\sD_u) 
= - \inner{c_2(TX)}{u_*[\Sigma]} 
+ \frac{\chi(\Sigma)}{24}\dim^\R X,
\end{equation}
where $\chi(\Sigma)$ is the Euler characteristic. 
Equation~\eqref{eq:index} continues to hold in the case 
$\Sigma=S^1\times M$ with its natural quaternionic structure. 
We sketch a proof below. Conjecturally, there should be 
Gromov--Witten type invariants obtained from intersection 
theory on the moduli space of solutions of~\eqref{eq:Hhol}. 

One can also consider hyper\-k\"ahler $4$-manifolds 
$\Sigma$ with cylindrical ends 
$\iota^\pm:\R^\pm\times M^\pm\to\Sigma$. Here we assume that 
$M^\pm$ is either a Cartan hypercontact 3-manifold or a 3-torus.
Then $\R^\pm\times M^\pm$ has a natural flat hyperk\"ahler 
structure~\cite{D,GG1}. In the hypercontact case the symplectic 
forms are 
$\om_i=\kappa^{-1}d(e^{-\kappa s}\alpha_i)
= e^{-\kappa s}
\bigl(-ds\wedge\alpha_i+\alpha_j\wedge\alpha_k\bigr)$
and in the torus case they are 
$\om_i=-ds\wedge\alpha_i+\alpha_j\wedge\alpha_k$
for every cyclic permutation $i,j,k$ of $1,2,3$. 
In both cases the complex structure $\i$ is given by 
$\p_s\mapsto -v_1$, $v_1\mapsto \p_s$, $v_2\mapsto v_3$,
$v_3\mapsto -v_2$ and similarly for $\j$ and $\k$.
We assume that the embeddings $\iota^\pm$ are hyperk\"ahler 
isomorphisms onto their images and that the complement 
$\Sigma\setminus(\im\,\iota^+\cup\,\im\,\iota^-)$ 
has a compact closure. 
Alternatively, it might also be interesting to consider hyperk\"ahler 
$4$-manifolds with asymptotically cylindrical ends as in~\cite{K1,K2}.
One can then (conjecturally) use the solutions 
of equation~\eqref{eq:Hhol} with Hamiltonian perturbations on the cylindrical 
ends to obtain a homomorphism $\HF_*(M^-,X)\to \HF_*(M^+,X)$
respectively $\HF^*(M^+,X)\to\HF^*(M^-,X)$. 

\medskip\noindent{\bf Proof of the index formula.}
We relate $\sD_u$ to a Dirac operator on $\Sigma$ 
associated to a spin$^c$ structure.
On $\Sigma$ we have a Hermitian vector bundle 
$W=W^+ \oplus W^-$ where 
$$
W^+:= u^*TX\oplus u^*TX,\quad
W^-:= \mathrm{Hom}_\H(T\Sigma,u^*TX)\oplus \mathrm{Hom}_I(T\Sigma,u^*TX).
$$ 
Here $\mathrm{Hom}_\H(T\Sigma,u^*TX)$ denotes
the bundle of quaternionic homomorphisms and 
$\mathrm{Hom}_I(T\Sigma,u^*TX)$ denotes the bundle 
of homomorphisms that are complex linear with respect 
to $I$ and complex anti-linear with respect to $J$ and $K$.  
The complex structures on $W^+$ and $W^-$ are given by 
$(\xi_1,\xi_2)\mapsto(I\xi_2,I\xi_1)$. The spin$^c$ 
structure $\Gamma: T\Sigma\to\mathrm{End}(W)$
has the form
$$
\Gamma(v) := \left(\begin{array}{cc}
  0 & -\gamma(v)^* \\
  \gamma(v) & 0
\end{array}\right)\
$$
for $v\in T_z\Sigma$ where $\gamma(v):W_z^+\to\W_z^-$
is given by
$$
\gamma(v)(\xi_1,\xi_2) 
:= (\pi_\H(\inner{v}{\cdot}\xi_1),\pi_I(\inner{v}{\cdot}\xi_2)).
$$
Here $\pi_\H,\pi_I:\Hom_\R(T\Sigma,u^*TX)\to\Hom_\R(T\Sigma.u^*TX)$
denote the projections
\begin{equation*}
\pi_\H\big(A\big) := A  - IA\i - JA\j - KA\k,\quad
\pi_I\big(A\big) := A  - IA\i + JA\j + KA\k.
\end{equation*}
The Dirac operator 
$D:\Om^0(\Sigma,W^+)\to\Om^0(\Sigma,W^-)$ 
is the direct sum of~$\sD_u$ and 
$\widetilde{\sD}_{u}:\Om^0(\Sigma,u^*TX)\to \Om^1_I(\Sigma,u^*TX)$ 
given by~${\widetilde{\sD}_{u}\xi := \pi_I(\nabla\xi)}$.
These operators have the same index and hence 
$$
2\mathrm{ind}^\R(\sD_{u})
= \mathrm{ind}^{\R}(D)
= \frac{\mathrm{rank}^{\R}(W^+)}{24}\chi(\Sigma)
+ \frac12\inner{c_1(W^+)^2-2c_2(W^+)}{[\Sigma]}.
$$
The last equation follows from the Atiyah--Singer 
index theorem (see~\cite{M}). Alternatively, one can 
identify $\Om^0(\Sigma,W^+)$ with 
$\Om^{0,0}(\Sigma,u^*TX)\oplus\Om^{2,0}(u^*TX)$
via 
$
(\xi_1,\xi_2)\mapsto(\xi_1+\xi_2,J(\xi_2-\xi_1)\om_\j+K(\xi_2-\xi_1)\om_\k)
$
and the space $\Om^0(\Sigma,W^-)$ with $\Om^{1,0}(\Sigma,u^*TX)$
via $(\alpha_1,\alpha_2)\to\alpha_1+\alpha_2$.  Under these
identifications the Dirac operator $D$ corresponds to the 
twisted Cauchy--Riemann operator 
$
{\p+\p^*:\Om^{\mathrm{ev},0}(\Sigma,u^*TX)
\to \Om^{\mathrm{odd},0}(\Sigma,u^*TX)}.
$
Since $I$ is homotopic to $-I$, 
the complex Fredholm index of $D$ is the holomorphic 
Euler characteristic of the bundle $u^*TX\to\Sigma$
and, by the Hirzebruch--Riemann--Roch formula, 
$$
\mathrm{ind}^\R(\sD_{u})
= \mathrm{index}^{\C}(D)
= \int_\Sigma\mathrm{ch}(u^*TX)\mathrm{td}(T\Sigma).
$$
With 
$
\mathrm{ch}=\rank^\C+c_1+\frac12(c_1^2-2c_2)
$
and 
$
\mathrm{td}=1+\frac12c_1+\frac1{12}(c_1^2+c_2)
$ 
this gives again the above formula, and~\eqref{eq:index} 
follows because $c_1(TX)=c_1(T\Sigma)=0$.

\subsection*{Ring structure}

As an example of this construction we obtain (conjecturally) a ring 
structure on $\HF^*(S^3,X)$. Take $\Sigma:=\H\setminus\{-\frac12,\frac12\}$
and define $\iota^-:(-\infty,0]\times S^3\to\H$ by 
$$
\iota^-(s,y):=e^{-s}y.
$$
The image of this map is the complement of the open unit ball in $\H$. 
The embedding $\iota^+:[0,\infty)\times(S^3\sqcup S^3)\to\H$
is the disjoint union of the embeddings $(s,y)\mapsto e^{-1-s}y\pm\frac12$. 
The resulting {\it quaternionic pair of pants product} 
$$
\HF^*(S^3,X)\otimes \HF^*(S^3,X)\to\HF^*(S^3,X)
$$
should be independent of the choice of the embeddings and the Hamiltonian
perturbations used to define it.   Moreover, counting the solutions 
of~\eqref{eq:Hhol} on the punctured cylinder 
$\R\times M\setminus\{\mathrm{pt}\}$, will lead to a module structure
of $\HF^*(M,X)$ over $\HF^*(S^3,X)$ for every $M$. 

The compactness and transversality results in the present paper show 
that this construction is perfectly rigorous and gives rise to an associative 
product on $\HF^*(S^3,X)$ whenever $X$ is flat.  Moreover, in this case 
it agrees with the usual cup product under our isomorphism
$$
\HF^*(S^3,X)\cong H^*(X;\Z_2).
$$ 

\subsection*{Relations with Donaldson--Thomas theory}

In~\cite{DT} Donaldson and Thomas outline the construction of
Donaldson type invariants of $8$-dimensional $\Spin(7)$-manifolds $Z$
and Floer homological invariants of $7$-dimensional $\G_2$-manifolds $Y$. 
In the case $Z=\Sigma\times S$, where $\Sigma$ and $S$ are
hyperk\"ahler surfaces, they explain that solutions of their equation
on $\Sigma\times S$ correspond, in the adiabatic limit where the metric
on $S$ degenerates to zero, to solutions $u:\Sigma\to\sM(S)$ of~\eqref{eq:Hhol}
with values in a suitable moduli space $X=\sM(S)$ of bundles over $S$.
In a similar vein there is a conjectural correspondence between the 
Donaldson-Thomas-Floer theory of 
$$
Y=M\times S
$$ 
with the Floer homology groups $\HF_*(M,\sM(S))$ discussed above whenever 
$M$ is either a Cartan hypercontact $3$-manifold or a flat $3$-torus. 
Namely, the solutions of the Floer equation in Donaldson--Thomas theory
on $\R\times Y$ with $Y=M\times S$ correspond, in the adiabatic
limit, formally to the solutions of~\eqref{eq:floer-H} 
on $\R\times M$ with values in $\sM(S)$. 

\subsection*{Boundary value problems}

If $M$ is Cartan hypercontact $3$-manifold with boundary $\p M$
and Reeb vector fields $v_1,v_2,v_3$ then there is a unique map 
$\lambda:\p M\to S^2$ such that 
$$
\nu:=\sum_i\lambda_iv_i:\p M\to TM
$$
is the outward pointing unit normal vector field.  In this case the
$1$-form~\eqref{eq:action-form} is not closed.  Its differential 
is given by the formula
$$
T_f\sF\times T_f\sF\to\R:
(\hat f_1,\hat f_2)\mapsto\int_{\p M}\om_\lambda(\hat f_1,\hat f_2)
\dvol_{\p M}.
$$
This is a symplectic form on the space of maps $\p M\to X$.  Thus it 
seems natural to impose the Lagrangian boundary condition
$$
f(y) \in L_y,\qquad y\in\p M,
$$
where $\bigsqcup_{y\in\p M}L_y$ is a smooth submanifold of $\p M\times X$
such that $L_y$ is Lagrangian with respect to $\om_{\lambda(y)}$ for every
$y\in\p M$.  In this paper we do not carry out the analysis for this boundary 
value  problem. 

In the technical parts of this paper we shall restrict the discussion
to the case where $M$ is a (Cartan) hypercontact $3$-manifold. 
The analysis for the case $M=\T^3$ is almost verbatim the same and in
some places easier because the metric is flat. 
In Section~\ref{sec:hyac} 
we introduce the hypersymplectic action functional and its critical points, 
discuss the Floer equation, and restate Theorem~A.   
In Section~\ref{sec:compact} we prove the main compactness 
and exponential decay theorems for the solutions 
of~\eqref{eq:crit-H} and~\eqref{eq:floer-H}.
These results are only valid for flat target manifolds $X$.  
The details of the transversality theory are worked out 
in Section~\ref{sec:trans} (for general target manifolds $X$).
With compactness and transversality established, the construction
of Floer homology is completely standard and we restrict ourselves 
to restating the result in Section~\ref{sec:floer}.   However, the 
computation of Floer homology still requires some serious analysis 
which is carried out in Section~\ref{sec:floer}. Three appendices 
discuss basic properties of hypercontact $3$-manifolds, 
the relevant a priori estimates, and a removable singularity theorem.

\medskip\noindent{\bf Acknowledgement.}
Thanks to Ron Stern for pointing out to us the 
discussion of $\pi_3(X)$ for a simply connected 
$4$-manifold $X$ in Cochran--Habegger~\cite{CH}. 
Thanks to Oliver Baues, Kenji Fukaya, Hansjoerg Geiges, and Katrin Wehr\-heim 
for helpful comments. Sonja Hohloch and Gregor Noetzel 
would like to thank the Forschungsinstitut f\"ur Mathematik 
at ETH Z\"urich for its hospitality. 

    
\section{The hypersymplectic action functional} \label{sec:hyac}

Let $X$ be a hyperk\"ahler manifold with complex structures 
$I,J,K$ and associated symplectic forms $\om_1,\om_2,\om_3$. 
Let $(M,\alpha_1,\alpha_2,\alpha_3)$ be a positive hypercontact 
$3$-manifold with Reeb vector fields $v_1,v_2,v_3$ 
(see Appendix~\ref{app:hyco}). 
Then the space $\sF:=\Map(M,X)$ 
of smooth maps $f:M\to X$ carries a natural 
{\bf hypersymplectic action functional}
$\sA:\sF\to\R$ given by
\begin{equation}\label{eq:action}
\sA(f) := - \int_M \Bigl(\alpha_1\wedge f^*\om_1 
+ \alpha_2\wedge f^*\om_2 + \alpha_3\wedge f^*\om_3  \Bigr).
\end{equation}
The next lemma shows that the critical points of $\sA$ are the 
solutions of the partial differential equation
\begin{equation}\label{eq:ijk}
\dd(f) := Idf(v_1)+Jdf(v_2)+Kdf(v_3)=0.
\end{equation}
This is a Dirac type elliptic equation because the vector fields $v_i$ 
are everywhere linearly independent (see Lemma~\ref{le:reeb}) 
and the complex structures $I,J,K$ satisfy the quaternionic relations. 
(The square of $\dd$ in local coordinates is a standard second 
order elliptic operator.) 

\begin{lemma}\label{le:crit}
The differential of $\sA$ along a path $\R\to\sF:t\mapsto f_t$
is 
$$
\frac{d}{dt} \sA(f_t) = \int_M\inner{\p_tf_t}{\dd(f_t)}\kappa\;\dvol_M,
$$
where $\kappa$ and the metric on $M$ 
are as in Remark~\ref{rmk:metric}. 
\end{lemma}

\begin{proof} 
By Cartan's formula, we have
$$
\frac{d}{dt} f_t^*\om_i = d\beta_i,\qquad \beta_i:= \om_i(\p_tf,df_t\cdot),
$$
for $i=1,2,3$.  Hence
\begin{eqnarray*}
\frac{d}{dt} \sA(f_t)
&=&
-\int_M\sum_i \alpha_i\wedge d\beta_i 
=
-\int_M\sum_i d\alpha_i\wedge\beta_i \\
&= &
-\int_M\left(\sum_i d\alpha_i\wedge\beta_i\right)(v_1,v_2,v_3)\,\dvol_M \\
&= &
-\int_M\kappa\sum_i\beta_i(v_i) \; \dvol_M \\
&= &
\int_M\kappa\inner{\p_tf_t}{Idf_t(v_1)+Jdf_t(v_2)+Kdf_t(v_3)} \,\dvol_M.
\end{eqnarray*}
Here the second equation follows from integration by parts,
the third equation follows from the fact that $v_1,v_2,v_3$ 
form an orthonormal basis, the fourth follows from the 
definition of $\kappa$ in Remark~\ref{rmk:metric},
and the last equation uses the definition of $\beta_i$ and the 
hyperk\"ahler structure of $X$.  This proves the lemma. 
\end{proof}

\subsection*{The energy identity}

The {\bf energy} of a smooth function $f:M\to X$ is defined by 
\begin{equation}\label{eq:energy}
\sE(f) :=  \frac12\int_M\Abs{df}^2\;\dvol_M
= \frac12\int_M\sum_{i=1}^3\Abs{df(v_i)}^2\;\dvol_M.
\end{equation}

\begin{lemma}\label{le:acten}
The energy of a smooth function $f:M\to X$ is related to the 
hypersymplectic action via
\begin{equation}\label{eq:acten}
\sE(f) = \sA(f) + \frac12\int_M\Abs{\dd(f)}^2\,\dvol_M 
- \int_M\inner{\dd(f)}{df(v_0)}\,\dvol_M,
\end{equation}
where 
\begin{equation}\label{eq:R}
v_0:=\alpha_2(v_3)v_1 + \alpha_3(v_1)v_2 + \alpha_1(v_2)v_3. 
\end{equation}
In particular $\sE(f)=\sA(f)$ for every solution of~\eqref{eq:ijk}.
\end{lemma}

\begin{remark}\label{rmk:R}\rm
The vector field $v_0$ vanishes if and only if 
$\alpha_i(v_j)=\delta_{ij}$. If this holds then, 
for every $f\in\sF$, we have
$$
\sE(f) = \sA(f) + \frac12\int_M\Abs{\dd(f)}^2\,\dvol_M. 
$$
Hence the energy of $f$ is controlled by the $L^2$ norm 
of $\dd(f)=\grad\sA(f)$ and the action. 
\end{remark}

\begin{proof}[Proof of Lemma~\ref{le:acten}]
By direct calculation (dropping the term $\dvol_M$) we obtain
\begin{eqnarray*}
&&\frac12\int_M\Abs{\dd(f)}^2 - \sE(f)  \\
&& =
\frac12\int_M\biggl(\Abs{Idf(v_1)+Jdf(v_2)+Kdf(v_3)}^2
-\sum_i\Abs{df(v_i)}^2\biggr) \\
&& =
\int_M\biggl(\inner{Kdf(v_1)}{df(v_2)}
+\inner{Idf(v_2)}{df(v_3)}
+\inner{Jdf(v_3)}{df(v_1)}
\biggr) \\
&& =
\int_M\biggl(f^*\om_1(v_2,v_3)
+f^*\om_2(v_3,v_1)+f^*\om_3(v_1,v_2)
\biggr).
\end{eqnarray*}
On the other hand
\begin{eqnarray*}
&& \int_M\inner{\dd(f)}{df(v_0)} - \sA(f)  \\
&& = 
\int_M\inner{\dd(f)}
{\alpha_2(v_3)df(v_1) + \alpha_3(v_1)df(v_2) + \alpha_1(v_2)df(v_3)} 
- \sA(f) \\
&& = 
\int_M\alpha_2(v_3)\biggl(f^*\om_2(v_2,v_1)+f^*\om_3(v_3,v_1)\biggr) \\
&&\quad +\,
\int_M\alpha_3(v_1)\biggl(f^*\om_1(v_1,v_2)+f^*\om_3(v_3,v_2)\biggr) \\
&&\quad +\,
\int_M\alpha_1(v_2)\biggl(f^*\om_1(v_1,v_3)+f^*\om_2(v_2,v_3)\biggr) \\
&&\quad
+\,  \int_M\biggl(\alpha_1\wedge f^*\om_1+
\alpha_2\wedge f^*\om_2+\alpha_3\wedge f^*\om_3\biggr)\\
&& =
\int_M\biggl(f^*\om_1(v_2,v_3)
+f^*\om_2(v_3,v_1)+f^*\om_3(v_1,v_2)
\biggr).
\end{eqnarray*}
The last equation follows by inserting the vector fields 
$v_1,v_2,v_3$ into the $3$-forms $\alpha_i\wedge f^*\om_i$. 
This proves the lemma.
\end{proof}
   
\subsection*{The Hessian}

The tangent space of $\sF$ at $f$ is the space
of vector fields along~$f$:
$$
T_f\sF = \Vect(f) = \Om^0(M,f^*TX).
$$
It is convenient to use the inner product
\begin{equation}\label{eq:L2}
\inner{\xi}{\eta}_{L^2} 
:= \int_M\inner{\xi}{\eta}\,\kappa\,\dvol_M.
\end{equation}
on this space.  One reason for this choice is the formula in
Lemma~\ref{le:crit}.  Another is the following observation.

\begin{lemma}\label{le:div}
For every smooth function $f:M\to\R$ we have
\begin{equation}\label{eq:dfint}
\int_M df(v_i)\kappa\,\dvol_M = 0.
\end{equation}
Thus the covariant divergence of the vector field $v_i$ is given by
$$
\div(v_i) = -\kappa^{-1}d\kappa(v_i)
$$
and the operator $\Nabla{v_i}:\Om^0(M,E)\to\Om^0(M,E)$
is skew adjoint with respect to the $L^2$ inner product~\eqref{eq:L2}
(for every Riemannian vector bundle $E\to M$ with any 
Riemannian connection).
\end{lemma}

\begin{proof}
The covariant divergence of a vector field $v\in\Vect(M)$
is the function $\div(v):M\to\R$ defined by 
$
\div(v):=\sum_j\inner{\Nabla{e_j}v}{e_j}
$
for any orthonormal frame $e_j$ of $TM$.
It is characterized by the property 
$$
\int_M df(v)\dvol_M + \int_Mf\div(v)\dvol_M = 0
$$
for every function $f:M\to\R$.  Now, for every $1$-form $\beta\in\Om^1(M)$,
we have $(\beta\wedge d\alpha_i)(v_1,v_2,v_3) = \beta(v_i)\kappa$
and hence $\beta\wedge d\alpha_i = \beta(v_i)\kappa\dvol_M$.
With $\beta=df$ this gives~\eqref{eq:dfint}. The formula for the 
covariant divergence of $v_i$ follows by replacing $f$ with $\kappa^{-1}f$.
This proves the lemma. 
\end{proof}

\begin{lemma}\label{le:hessian}
The covariant Hessian of $\sA$ at $f\in\sF$ is the 
operator 
$$
\dcD=\dcD_f:\Om^0(M,f^*TX)\to\Om^0(M,f^*TX)
$$ 
given by
\begin{equation}\label{eq:hessian}
\dcD\xi := I\Nabla{v_1}\xi 
+ J\Nabla{v_2}\xi + K\Nabla{v_3}\xi
\end{equation}
for $\xi\in\Om^0(M,f^*TX)$.
Here $\nabla$ is the Levi--Civita connection of the 
hyperk\"ahler metric on~$X$.  The operator 
$\dcD:W^{1,2}(M,f^*TX)\to L^2(M,f^*TX)$ is 
self-adjoint with respect to the $L^2$ inner product~\eqref{eq:L2}.
\end{lemma}

\begin{proof}
The covariant Hessian of $\sA$ at $f\in\sF$ is defined by 
the formula 
$
{\xi\mapsto \Nabla{t}(\dd f_t)|_{t=0}}
$
where $t\mapsto f_t$ is a smooth curve in $\sF$ with $f_0=f$
and $\p_tf_t|_{t=0}=\xi$. Hence~\eqref{eq:hessian} 
follows from the fact that the complex structures $I,J,K$ are 
covariant constant and $\nabla$ is torsion free. 
That $\dcD$ is symmetric with respect to the
$L^2$ inner product~\eqref{eq:L2} follows from 
Lemma~\ref{le:div}. To prove that $\dcD$ is self-adjoint 
we observe that its square is given by
\begin{equation}\label{eq:D2}
\begin{split}
\dcD\dcD\xi = 
&-\Nabla{v_1}\Nabla{v_1}\xi 
-\Nabla{v_2}\Nabla{v_2}\xi 
-\Nabla{v_3}\Nabla{v_3}\xi \\
&+ I\bigl(R(df(v_2),df(v_3))\xi - \Nabla{[v_2,v_3]}\xi\bigr) \\
&+ J\bigl(R(df(v_3),df(v_1))\xi - \Nabla{[v_3,v_1]}\xi\bigr) \\
&+ K\bigl(R(df(v_1),df(v_2))\xi - \Nabla{[v_1,v_2]}\xi\bigr). 
\end{split}
\end{equation}
Here $R$ denotes the Riemann curvature tensor on $X$.
Since $v_1,v_2,v_3$ are linearly independent
$\dcD^2$ is a standard second order elliptic operator 
in local coordinates (with leading term in diagonal form) 
and hence has the usual elliptic regularity properties.
In particular, if $\xi\in L^2$ and $\dcD\xi\in L^2$,
then $\dcD^2\xi\in W^{-1,2}$ and elliptic regularity 
gives $\xi\in W^{1,2}$. This implies that $\dcD$ 
is self-adjoint as an operator on $L^2$ with 
domain $W^{1,2}$, as claimed.
\end{proof}

As in symplectic and instanton Floer theories it is a fundamental 
observation that the action functional is unbounded above and below
and that the operator $\dcD$ has infinitely positive and negative 
eigenvalues.  

\begin{remark}\label{rmk:unbounded}\rm
If the symplectic forms $\om_i=d\lambda_i$ on $X$ are exact
then the hypersymplectic action functional can be written 
in the form
$$
\sA(f) = \int_M\sum_{i=1}^3\lambda_i(\p_{v_i}f)\,\kappa\,\dvol_M.
$$
The archetypal example is the space $X=\H$ of quaternions
with the standard hyperk\"ahler structure.  In this case
the operator $f\mapsto\dd(f)=\dcD f$ is linear 
and the hypersymplectic action is the associated 
quadratic form  
$$
\sA(f) = \frac12\int_M\inner{f}{\dcD f}\kappa\dvol_M.
$$
Since $\sA(f)=0$ for every real valued function 
$f:M\to\R\subset\H$ it follows that the negative and 
positive eigenspaces of $\dcD$ are both infinite dimensional.
In the case $M=S^3$ with the standard 
hypercontact structure, specific eigenfunctions are 
$f(y)=y$ with eigenvalue $-3$, $f(y)=y+2\bar y$ with 
eigenvalue~$1$, and $f(y)=\iota\circ h(y)$
where $h:S^3\to S^2$ is a suitable Hopf fibration and 
$\iota:S^2\to\H$ is the inclusion of the $2$-sphere
into the imaginary quaternions; in the last example the 
eigenvalue is $-4$.  
\end{remark}

\subsection*{Perturbations} 

Let $H:X\times M\to\R$ be a smooth function and define 
the perturbed hypersymplectic action functional 
$\sA_H:\sF\to\R$ by 
$$
\sA_H(f) := - \int_M \sum_{i=1}^3\alpha_i\wedge f^*\om_i
- \int_M H(f) \kappa\,\dvol_M.
$$
Here we write $H(f)$ for the function $M\to\R:y\mapsto H(f(y),y)$.
For ${y\in M}$ let $H_y:=H(\cdot,y)$ and denote by 
$\nabla H(\cdot,y):=\nabla H_y$ the gradient of $H$ with respect 
to the first argument. Then, by Lemma~\ref{le:crit}, the critical 
points of $\sA_H$ are the solutions of the perturbed equation
\begin{equation}\label{eq:critH}
Idf(v_1)+Jdf(v_2)+Kdf(v_3) = \nabla H(f).
\end{equation}
Here we denote by $\nabla H(f)$ the vector field
$y\mapsto \nabla H_y(f(y))$ along $f$.  By Lemma~\ref{le:acten},
every solution of~\eqref{eq:critH} satisfies the inequality
$$
\sA_H(f)  
\ge - \int_M\left(\kappa H(f) 
+ \frac12\Abs{\nabla H(f)}^2\right)\,\dvol_M.
$$

\subsection*{Gradient flow lines}

By Lemma~\ref{le:crit}, the gradient 
of $\sA_H$ with respect to the $L^2$ inner 
product~\eqref{eq:L2} is given by
$$
\grad\sA_H(f) = Idf(v_1)+Jdf(v_2)+Kdf(v_3)-\nabla H(f)
=: \dd_H(f).
$$
Hence the negative gradient flow lines of $\sA_H$ 
are the solutions $u:\R\times M\to X$ 
of the partial differential equation
\begin{equation}\label{eq:floerH}
\p_s u + I\p_{v_1}u 
+ J\p_{v_2}u + K\p_{v_3}u = \nabla H(u).
\end{equation}
The {\bf energy} of a smooth map $u:\R\times M\to X$ is defined by
$$
\sE_H(u) := \frac12\int_{\R\times M}\left(
\Abs{\p_su}^2 + \Abs{\dd_H(u)}^2
\right)\,\kappa\,\dvol_M\,ds.
$$
As in finite dimensional Morse theory and Floer 
homology, the finite energy solutions of~\eqref{eq:floerH} are the ones
that converge to critical points of the perturbed hypersymplectic 
action functional as $s$ tends to $\pm\infty$ 
(see Theorem~\ref{thm:asymptotic} below).
Thus, in the case $\sE_H(u)<\infty$, there are 
solutions $f^\pm:M\to X$ of equation~\eqref{eq:critH}
such that $\lim_{s\to\pm\infty} \p_su(s,y) = 0$,
uniformly in $y$, and
\begin{equation}\label{eq:limitH}
\lim_{s\to\pm\infty} u(s,y) = f^\pm(y),\qquad
\lim_{s\to\pm\infty} \sA_H(u(s,\cdot)) = \sA_H(f^\pm).
\end{equation}
Moreover the solutions of~\eqref{eq:floerH}
minimize the energy $\sE_H(u)$ subject to~\eqref{eq:limitH}
and their energy is 
$
\sE_H(u) 
= \sA_H(f^-) - \sA_H(f^+).
$

\subsection*{Moduli spaces}

A solution $f$ of $\dd_H(f)=0$ is called 
{\bf nondegenerate} if the perturbed Hessian
\begin{equation}\label{eq:hessH}
\dcD_{f,H}\xi := I\Nabla{v_1}\xi + J\Nabla{v_2}\xi + K\Nabla{v_3}\xi 
- \Nabla{\xi}\nabla H_y(f)
\end{equation}
is bijective.  We shall prove that nondegeneracy can be achieved
by a generic choice of the Hamiltonian $H:X\times M\to\R$
(see Theorem~\ref{thm:morse} below). Assuming this 
we fix two critical points $f^\pm$ of the perturbed 
hypersymplectic action functional $\sA_H$ and 
denote the space of Floer trajectories by 
$$
\sM(f^-,f^+;H)
:= \left\{u:\R\times M\to X\,|\,u\mbox{ satisfies }\eqref{eq:floerH},\,
\eqref{eq:limitH},\,\sup_{\R\times M}\Abs{du}<\infty\right\}.
$$
We shall prove, again for a generic choice of the perturbation,
that these spaces are smooth finite dimensional manifolds.
The proof will involve the linearized operator 
\begin{equation}\label{eq:Du}
\sD_{u,H}\xi := \Nabla{s}\xi + I\Nabla{v_1}\xi 
+ J\Nabla{v_2}\xi + K\Nabla{v_3}\xi - \Nabla{\xi}\nabla H(u).
\end{equation}
As in all other versions of Floer homology the Fredholm index 
of this operator is the spectral flow of the Hessians along $u$.
We shall prove that, when $M$ is a Cartan hypercontact 
$3$-manifold and $X$ is flat all the known analysis 
of symplectic Floer theory carries over to the present setting 
and gives rise to Floer homology groups that are isomorphic 
to the singular homology of $X$.  
This leads to the following 
existence theorem for solutions of $\dd_H(f)=0$.  
We emphasize that the algebraic count of the solutions 
gives zero and thus does not provide an existence result.

\begin{theorem}\label{thm:CZ}
Let $M$ be a compact Cartan hypercontact $3$-manifold 
and $X$ be a compact flat hyperk\"ahler manifold.  
If every solution $f$ of $\dd_H(f)=0$ is nondegenerate 
then their number is bounded below by the sum of the 
Betti numbers of $X$ (with coefficient ring $\Z_2$).
In particular, $\dd_H(f)=0$ has a solution 
for every smooth function ${H:X\times M\to\R}$.
\end{theorem}

\begin{proof}
See Section~\ref{sec:floer}.
\end{proof}

    
\section{Regularity and compactness} \label{sec:compact}

We assume throughout that $X$ is a compact hyperk\"ahler manifold
and $M$ is a compact $3$-manifold equipped with a positive
hypercontact structure $\alpha$. Then the Reeb vector fields 
$v_1,v_2,v_3$ form a (positive) orthonormal frame of $TM$ 
and hence determine a second order elliptic operator
\begin{equation}\label{eq:L}
L := \sum_{i=1}^3 \cL_{v_i}\cL_{v_i}
= - d^*d - \sum_{i=1}^3\div(v_i)\cL_{v_i}.
\end{equation}
If $\alpha$ is a Cartan structure then
$\div(v_i)=d^*\alpha_i=0$ (by Lemma~\ref{le:div}) 
and so~$L$ is the Laplace--Beltrami operator on~$M$. 
In local coordinates $y^1,y^2,y^3$ on $M$
the operator $L$ has the form
\begin{equation}\label{eq:L-loc}
L = \sum_{\mu,\nu=1}^3a^{\mu\nu}\frac{\p^2}{\p y^\mu\p y^\nu}
+ \sum_{\nu=1}^3b^\nu\frac{\p}{\p y^\nu},
\end{equation}
where
$$
a^{\mu\nu} := \sum_{i=1}^3v^\mu_iv^\nu_i,\qquad
b^\nu := \sum_{i,\mu=1}^3 \frac{\p v^\nu_i}{\p y^\mu}v^\mu_i.
$$
Since the vector fields $v_i$ form an orthonormal frame of $TM$,
the coefficients~$a^{\mu\nu}$ define the Riemannian metric 
on the cotangent bundle in our local coordinates.
The operators $L$ on $M$ and 
\begin{equation}\label{eq:sL}
\sL := \p_s\p_s + L
\end{equation}
on $\R\times M$ will play a central role in our study of 
the solutions of equations~\eqref{eq:critH} and~\eqref{eq:floerH}.

\begin{theorem}[\bf Regularity]\label{thm:reg}
If $p>3$ then every $W^{1,p}$ solution $f$ of $\dd_H(f)=0$
is smooth. If $p>4$ then every $W^{1,p}$ solution $u$ 
of~\eqref{eq:floerH} is smooth.
\end{theorem}

\begin{proof}
For every vector field $v\in\Vect(M)$ we write 
$
\p_vf=\cL_vf=df(v).
$
Then, for every smooth map $f:M\to X$ and any two vector
fields~$v,w$ on~$M$, we have
$
\Nabla{v}\p_wf - \Nabla{w}\p_vf = -\p_{[v,w]}f.
$
Hence
\begin{equation}\label{eq:DDf}
\dcD\dd(f) = - \sum_{i=1}^3\Nabla{v_i}\p_{v_i}f 
 - I\p_{[v_2,v_3]}f - J\p_{[v_3,v_1]}f - K\p_{[v_1,v_2]}f.
\end{equation}
In local coordinates $(x^1,\dots,x^m)$ on $X$ 
and $(y^1,y^2,y^3)$ on $M$ we have 
\begin{equation*}
\begin{split}
(\Nabla{v}\p_wf)^k = 
\sum_{\nu,\mu=1}^3
\left(
\frac{\p^2f^k}{\p y^\nu\p y^\mu}v^\mu w^\nu
+ \frac{\p f^k}{\p y^\nu}\frac{\p w^\nu}{\p y^\mu}v^\mu
+ \sum_{i,j=1}^m\Gamma^k_{ij}\frac{\p f^i}{\p y^\mu}
  \frac{\p f^j}{\p y^\nu}v^\mu w^\nu
\right).
\end{split}
\end{equation*}
With $L$ as in~\eqref{eq:L-loc} this gives
$$
(\dcD\dd(f))^k
= 
- Lf^k 
- \sum_{\mu,\nu=1}^3\sum_{i,j=1}^m
\Gamma^k_{ij}a^{\mu\nu}\frac{\p f^i}{\p y^\mu}
\frac{\p f^j}{\p y^\nu} 
- g^k_f,
$$
where
$$
g^k_f := \sum_{\ell=1}^m\sum_{\nu=1}^3
\left(
I^k_\ell [v_2,v_3]^\nu 
+ J^k_\ell [v_3,v_1]^\nu 
+ K^k_\ell [v_1,v_2]^\nu 
\right)\frac{\p f^\ell}{\p y^\nu}.
$$
Moreover, the function $h_f^k:=(\dcD\dd(f))^k=(\dcD\nabla H(f))^k$ 
is given by 
\begin{eqnarray*}
h_f^k 
&=& 
\sum_{j,\ell=1}^m\sum_{\nu=1}^3
c^{k\nu}_\ell
\left(
\frac{\p(\nabla H)^\ell}{\p x^j}\frac{\p f^j}{\p y^\nu}
+ \frac{\p(\nabla H)^\ell}{\p y^\nu}
+ \sum_{i=1}^m\Gamma^\ell_{ij}(\nabla H)^i\frac{\p f^j}{\p y^\nu}
\right),
\end{eqnarray*}
where 
$$
c^{k\nu}_\ell:= I^k_\ell v^\nu_1 + J^k_\ell v^\nu_2 + K^k_\ell v^\nu_3.
$$
Hence every solution of~\eqref{eq:critH} (of class $W^{1,p}$)
satisfies the elliptic pde
\begin{equation}\label{eq:critH-local}
L f^k =
- \sum_{\mu,\nu=1}^3\sum_{i,j=1}^m
\Gamma^k_{ij}a^{\mu\nu}\frac{\p f^i}{\p y^\mu}
\frac{\p f^j}{\p y^\nu} - g_f^k-h_f^k.
\end{equation}
If $f\in W^{1,p}$ for some $p>3$ then the right hand side is 
in $L^{p/2}$ and hence, by elliptic regularity, $f$ is of 
class $W^{2,p/2}$. By the Sobolev embedding theorem, we then obtain 
$f\in W^{1,q}$ where $q:=3p/(6-p)>p$. Continuing by induction, 
we obtain eventually that $f\in W^{1,q}$ for some $q>6$,
hence $f\in W^{2,p}$ and, again by induction, $f\in W^{k,p}$ 
for every integer $k$. 

To prove regularity for the solutions 
of~\eqref{eq:floerH} we introduce the operators
$$
\sD := \Nabla{s}+I\Nabla{v_1}+J\Nabla{v_2}+K\Nabla{v_3},\quad
\sD^* := -\Nabla{s}+I\Nabla{v_1}+J\Nabla{v_2}+K\Nabla{v_3}
$$
Then 
\begin{equation}\label{eq:DDu}
\begin{split}
\sD^*(\p_su+\dd(u)) 
= &-\Nabla{s}\p_su - \sum_{i=1}^3\Nabla{v_i}\p_{v_i}u  \\ 
& - I\p_{[v_2,v_3]}u - J\p_{[v_3,v_1]}u - K\p_{[v_1,v_2]}u.
\end{split}
\end{equation}
Here we have used~\eqref{eq:DDf} and the fact that 
$\Nabla{s}\p_{v_i}u=\Nabla{v_i}\p_su$ for $i=1,2,3$. 
If $u$ is a solution of~\eqref{eq:floerH}
then $\p_su+\dd(u)=\nabla H(u)$.  Hence in this case we 
obtain the equation
\begin{equation}\label{eq:floerH-local}
\sL u^k =
- \sum_{\mu,\nu=1}^3\sum_{i,j=1}^m
\Gamma^k_{ij}(u)\left(\frac{\p u^i}{\p s}\frac{\p u^j}{\p s}
+ a^{\mu\nu}\frac{\p u^i}{\p y^\mu}\frac{\p u^j}{\p y^\nu}
\right) - g_u^k-h_u^k,
\end{equation}
where
\begin{equation*}
\sL := \frac{\p^2}{\p s^2} 
+ \sum_{\mu,\nu=1}^3a^{\mu\nu}\frac{\p^2}{\p y^\mu\p y^\nu}
+ \sum_{\nu=1}^3b^\nu\frac{\p}{\p y^\nu},
\end{equation*}
and
\begin{equation}\label{eq:gh}
\begin{split}
g^k_u 
&:= \sum_{\ell=1}^m\sum_{\nu=1}^3
\left(
I^k_\ell [v_2,v_3]^\nu 
+ J^k_\ell [v_3,v_1]^\nu 
+ K^k_\ell [v_1,v_2]^\nu 
\right)\frac{\p u^\ell}{\p y^\nu},  \\
h_u^k 
&:=
\sum_{j,\ell=1}^m\sum_{\nu=1}^3
c^{k\nu}_\ell
\left(
\frac{\p(\nabla H)^\ell}{\p x^j}\frac{\p u^j}{\p y^\nu}
+ \frac{\p(\nabla H)^\ell}{\p y^\nu}
+ \sum_{i=1}^m\Gamma^\ell_{ij}(\nabla H)^i\frac{\p u^j}{\p y^\nu}
\right) \\
&\quad
- \sum_{j=1}^m \frac{\p(\nabla H)^k}{\p x^j}\frac{\p u^j}{\p s}
- \sum_{i,j=1}^m\Gamma^k_{ij}(u)(\nabla H)^i\frac{\p u^j}{\p s}.
\end{split}
\end{equation}
If $u\in W^{1,p}$ with $p>4$ then the right hand side 
in~\eqref{eq:DDu} is in $L^{p/2}$ and so $u\in W^{2,p/2}$. 
Thus the Sobolev embedding theorem gives $u\in W^{1,q}$ 
with $q:=4p/(8-p)>p$. Continuing by induction we obtain 
that $u$ is smooth. This proves the theorem.
\end{proof}

The bootstrapping argument in the proof of Theorem~\ref{thm:reg} 
gives rise to uniform estimates for sequences that are bounded 
in $W^{1,p}$.  Hence the Arz\'ela--Ascoli theorem gives the 
following compactness result.

\begin{theorem}\label{thm:cpct}
Assume $X$ is compact.

\smallskip\noindent{\bf (i)} 
Let $p>3$ and $\Om\subset M$ be an open set.  
Then every sequence of solutions 
$f^\nu:\Om\to X$ of equation~\eqref{eq:critH} 
that satisfies
$
\sup_\nu\Norm{df^\nu}_{L^p(C)} < \infty
$
for every compact set $C\subset\Om$ has a subsequence 
that converges in the $\Cinf$ topology on every compact 
subset of $\Om$.

\smallskip\noindent{\bf (ii)} 
Let $p>4$ and $\Om\subset\R\times M$ be an open set.
Then every sequence of solutions $u^\nu:\Om\to X$
of equation~\eqref{eq:floerH} that satisfies
$
\sup_\nu\Norm{du^\nu}_{L^p(C)} < \infty
$
for every compact set $C\subset\Om$ has a subsequence 
that converges in the $\Cinf$ topology on every compact 
subset of $\Om$.
\end{theorem}

\subsection*{A priori estimates}

To remove the bounded derivative assumption in Theorem~\ref{thm:cpct},
at least in the case where $X$ is flat,  we establish mean value inequalities 
for the energy density of the solutions of~\eqref{eq:floerH}.  
The solutions of~\eqref{eq:critH} then correspond to the special case 
$\p_su\equiv0$.  The mean value inequalities will be based on 
Theorem~\ref{thm:heinz} in Appendix~\ref{app:heinz}. 

Throughout we denote by $L$ and $\sL$ the operators~\eqref{eq:L} 
and~\eqref{eq:sL} on $M$ and $\R\times M$, respectively, 
and by $R$ the Riemann curvature tensor on $X$.
For a map $u:\R\times M\to X$ we define the 
{\bf energy density} $e_u:\R\times M\to \R$ by
$$
e_u := \frac12\Abs{\p_su}^2 
+ \frac12\sum_{i=1}^3\Abs{\p_{v_i}u}^2,
$$
and the {\bf scalar curvature} 
$r_u:\R\times M\to\R$ {\bf along $u$} by
$$
r_u := 
2\sum_{j=1}^3\inner{R(\p_su,\p_{v_j}u)\p_{v_j}u}{\p_su}
+ \sum_{i,j=1}^3\inner{R(\p_{v_i}u,\p_{v_j}u)\p_{v_j}u}{\p_{v_i}u}.
$$
Throughout we fix a Hamiltonian perturbation $H:X\times M\to\R$.
We explicitly do not assume that the hypercontact structure on $M$ 
is a Cartan structure (unless otherwise mentioned).

\begin{lemma}\label{le:Ler}
There are positive constants $A$ and $B$, depending only on the 
vector fields $v_i$, the metric on $X$, and the Hamiltonian
perturbation $H$, such that every solution 
$u:\R\times M\to X$ of~\eqref{eq:floerH} 
satisfies the estimate
\begin{equation}\label{eq:Ler}
\sL e_u + r_u \ge - A - B(e_u)^{3/2}.
\end{equation}
If $H=0$ we obtain an estimate of the form
$
\sL e_u + r_u \ge  - Ce_u.
$
\end{lemma}

\bigbreak

\begin{proof}
It is convenient to denote the vector field $\p_s$
on $\R\times M$ by $v_0$.  Then the Lie brackets 
$[v_0,v_j]$ vanish for all $j$, but we shall not 
use this fact.  Abbreviate 
$$
w_1:=[v_2,v_3],\qquad 
w_2:=[v_3,v_1],\qquad 
w_3:=[v_1,v_2]
$$ 
and define the operators 
$$
\sL^X := \sum_{i=0}^3\Nabla{v_i}\p_{v_i},\qquad
\sL^\nabla := \sum_{i=0}^3\Nabla{v_i}\Nabla{v_i}.
$$
Thus $\sL^X$ acts on maps $u:\R\times M\to X$
and $\sL^\nabla$ acts on vector fields along such maps. 
With this notation every solution $u$ 
of~\eqref{eq:floerH} satisfies the equation
\begin{equation}\label{eq:DDfloer}
\sL^Xu = - \sD^*\nabla H(u)
- I\p_{w_1}u - J\p_{w_2}u - K\p_{w_3}u,
\end{equation}
where $\sD^*=-\Nabla{s}+I\Nabla{v_1}+J\Nabla{v_2}+K\Nabla{v_3}$.
Moreover,
\begin{equation}\label{eq:Le}
\sL e_u = 
\sum_{j=0}^3\inner{\sL^\nabla\p_{v_j}u}{\p_{v_j}u}
+ p_u,\qquad
p_u := \sum_{i,j=0}^3\Abs{\Nabla{v_i}\p_{v_j}u}^2.
\end{equation}
We compute
\begin{equation*}
\begin{split}
\sL^\nabla\p_{v_j}u
&=
\sum_i\Nabla{v_i}\Nabla{v_j}\p_{v_i}u
- \sum_i\Nabla{v_i}\p_{[v_i,v_j]}u \\
&=
\sum_i R(\p_{v_i}u,\p_{v_j}u)\p_{v_i}u 
+ \Nabla{v_j}\sL^Xu \\
&\quad
- \sum_i\left(\Nabla{v_i}\p_{[v_i,v_j]}u 
+\Nabla{[v_i,v_j]}\p_{v_i}u\right) \\
&=
\sum_i R(\p_{v_i}u,\p_{v_j}u)\p_{v_i}u 
+ h_j(u) + \xi_j(u),
\end{split}
\end{equation*}
where the sums are over $i=0,1,2,3$ and
\begin{equation*}
\begin{split}
h_j(u) &:= 
\Bigl(
\Nabla{v_j}\Nabla{s}
- I\Nabla{v_j}\Nabla{v_1}
- J\Nabla{v_j}\Nabla{v_2} 
- K\Nabla{v_j}\Nabla{v_3}
\Bigr)\nabla H(u), \\
\xi_j(u) &:=
- I\Nabla{v_j}\p_{w_1}u 
- J\Nabla{v_j}\p_{w_2}u
- K\Nabla{v_j}\p_{w_3}u \\
&\quad\;\,
- \sum_i\left(\Nabla{v_i}\p_{[v_i,v_j]}u 
+ \Nabla{[v_i,v_j]}\p_{v_i}u\right).
\end{split}
\end{equation*} 
Since the vector fields $v_1,v_2,v_3$ form an orthonormal 
frame of $TM$ there is a constant $c\ge1$ such that, 
for every smooth map $u:\R\times M\to X$ and every
smooth perturbation $H:X\times M\to\R$, we have
\begin{equation*}
\sum_{j=0}^3\Abs{\xi_j(u)}^2 
\le c 
\left(e_u+p_u\right), \qquad
\sqrt{\sum_{j=0}^3\Abs{h_j(u)}^2} 
\le c\Norm{H}_{C^3}\left(1+e_u + 
\sqrt{p_u}\right).
\end{equation*}
Here $p_u$ is as in~\eqref{eq:Le}. 
This gives 
\begin{equation*}
\begin{split}
\Abs{\sum_{j=0}^3\inner{\xi_j(u)}{\p_{v_j}u}}
&\le \frac{1}{2c}\sum_{j=0}^3\abs{\xi_j(u)}^2 
+ \frac{c}{2}\sum_{j=0}^3\abs{\p_{v_j}u}^2 \\
&\le \frac{p_u}{2} + \left(c+\frac{1}{2}\right)e_u, \\
\Abs{\sum_{j=0}^3\inner{h_j(u)}{\p_{v_j}u}}
&\le c\Norm{H}_{C^3}\sqrt{2e_u}\left(1+e_u + \sqrt{p_u}\right) \\
&\le \frac{p_u}{2} + c^2\Norm{H}_{C^3}^2e_u
+ \sqrt{2}c\Norm{H}_{C^3}\sqrt{e_u}(1+e_u).
\end{split}
\end{equation*}
Hence it follows from~\eqref{eq:Le} that
\begin{equation*}
\begin{split}
\sL e_u + r_u 
&=
p_u + \sum_{j=0}^3\inner{h_j(u)+\xi_j(u)}{\p_{v_j}u} \\
&\ge 
- \left(\frac12+c+c^2\Norm{H}_{C^3}^2\right)e_u
- \sqrt{2}c\Norm{H}_{C^3}(e_u)^{1/2}(1+e_u) \\
&\ge 
- c^2\Norm{H}_{C^3}^2
- \left(1+c+c^2\Norm{H}_{C^3}^2\right)e_u
- \sqrt{2}c\Norm{H}_{C^3}(e_u)^{3/2}
\end{split}
\end{equation*} 
and thus
\begin{equation}\label{eq:LerH}
\sL e_u + r_u 
\ge 
- C\left(\Norm{H}_{C^3}^2
+ \left(1+\Norm{H}_{C^3}^2\right)e_u
+ \Norm{H}_{C^3}(e_u)^{3/2}\right),
\end{equation} 
where $C:=1+c^2$.
Using the inequality 
$
ab\le \frac{1}{3}a^3+\frac{2}{3}b^{3/2}
$
for $a,b\ge0$ we obtain~\eqref{eq:Ler} with
$$
A:= C\left(\Norm{H}_{C^3}^2 
+ \frac{1}{3}\left(1+\Norm{H}_{C^3}^2\right)^3\right),\qquad
B:=C\left(\frac{2}{3}+\Norm{H}_{C^3}\right).
$$
This proves the lemma.
\end{proof}

\begin{remark}\label{rmk:apriori}\rm
For general hyperk\"ahler manifolds Lemma~\ref{le:Ler}
gives an estimate of the form
$$
\sL e \ge -c(1+e^2)
$$
for the energy density of solutions of~\eqref{eq:critH} 
and~\eqref{eq:floerH}. In dimensions $n=3,4$ the exponent 
$2$ is larger than the critical exponent $(n+2)/n$ 
in Theorem~\ref{thm:heinz}. For the critical points 
$f:M\to X$ of $\sA_H$ this means that the energy
$$
\sE(f)=\frac12\int_M\Abs{df}^2\,\dvol_M
$$ 
does not control the sup norm of $\Abs{df}$ even if we assume 
that there is no energy concentration near points. 
This is related to noncompactness phenomena that can 
be easily observed in examples. 
Namely, composing a 
holomorphic sphere in $X$ (for one of the complex 
structures $J_\lambda=\lambda_1I+\lambda_2J+\lambda_3K$)
with a suitable Hopf fibration gives rise to a solution
of $\dd(f)=0$.  Now the bubbling 
phenomenon for holomorphic spheres leads to sequences
$f^\nu:S^3\to X$ of solutions of~\eqref{eq:critH} 
where the derivative $df^\nu$ blows up along a Hopf
circle, while the energy remains bounded.
\end{remark}

\begin{lemma}\label{le:Les}
There is a constant $C>0$, depending only on the 
vector fields $v_i$, the metric on $X$, and the Hamiltonian
perturbation $H$, such that every solution 
$u:\R\times M\to X$ of~\eqref{eq:floerH} 
satisfies the estimate
\begin{equation}\label{eq:Les}
\sL \Abs{\p_su}^2 
\ge - C\left(1+\Abs{du}^2\right)\Abs{\p_su}^2.
\end{equation}
\end{lemma}

\begin{proof}
Abbreviate $v_0:=\p_s$ and $w_1:=[v_2,v_3]$, 
$w_2:=[v_3,v_1]$, $w_3:=[v_1,v_2]$ as in the 
proof of Lemma~\ref{le:Ler}.  
Define the 
functions $e_0,r_0:\R\times M\to\R$ by 
$$
e_0 := \frac12\Abs{\p_su}^2,\qquad
r_0 := \sum_{i=1}^3\inner{R(\p_su,\p_{v_i}u)\p_{v_i}u}{\p_su}
$$
Then 
\begin{equation}\label{eq:Le0}
\sL e_0 = \sum_{i=0}^3\Abs{\Nabla{v_i}\p_su}^2
+ \inner{\sL^\nabla\p_su}{\p_su}.
\end{equation}
As in the proof of Lemma~\ref{le:Ler} we have
\begin{equation}\label{eq:Lsu}
\sL^\nabla\p_su
= \sum_i R(\p_{v_i}u,\p_su)\p_{v_i}u 
+ h_0(u) + \xi_0(u),
\end{equation}
where the sum is over $i=1,2,3$ and
\begin{equation*}
\begin{split}
h_0(u) &:= 
\Bigl(
\Nabla{s}\Nabla{s}
- I\Nabla{s}\Nabla{v_1}
- J\Nabla{s}\Nabla{v_2} 
- K\Nabla{s}\Nabla{v_3}
\Bigr)\nabla H(u), \\
\xi_0(u) &:=
- I\Nabla{s}\p_{w_1}u 
- J\Nabla{s}\p_{w_2}u
- K\Nabla{s}\p_{w_3}u.
\end{split}
\end{equation*} 
Now 
$$
\Abs{h_0(u)} + \Abs{\xi_0(u)} 
\le c\left(
1+\Abs{du} + 
\sqrt{\sum_i\Abs{\Nabla{v_i}\p_su}^2}
\right)\Abs{\p_su}
$$
and hence it follows from~\eqref{eq:Le0} 
and~\eqref{eq:Lsu} that
\begin{eqnarray*}
\sL e_0 + r_0
&=& 
\sum_i\Abs{\Nabla{v_i}\p_su}^2
+ \inner{h_0(u)+\xi_0(u)}{\p_su} \\
&\ge &
\sum_i\Abs{\Nabla{v_i}\p_su}^2
- 2c\left(
1+\Abs{du} + 
\sqrt{\sum_i\Abs{\Nabla{v_i}\p_su}^2}
\right)e_0 \\
&\ge &
\frac12\sum_i\Abs{\Nabla{v_i}\p_su}^2
- 2c\Bigl(1 + \Abs{du} + ce_0\Bigr)e_0.
\end{eqnarray*} 
Since $e_0\le\Abs{du}^2$ and $r_0\le c\Abs{du}^2e_0$ 
this proves~\eqref{eq:Les}.
\end{proof}

\subsection*{Compactness for critical points} 

\begin{theorem}\label{thm:crit-cpct}
Let $M$ be a Cartan hypercontact $3$-manifold and $X$
be a compact flat hyperk\"ahler manifold. Let ${H:X\times M\to\R}$ 
be any smooth function.  Then the set of solutions 
of~\eqref{eq:critH} is compact in the $\Cinf$ topology.
\end{theorem}

\begin{lemma}\label{le:apriori}
Let $M$ be a Cartan hypercontact $3$-manifold
and $X$ be a compact flat hyperk\"ahler manifold.
Then there is a constant $c>0$ such that 
$$
\sA(f) \le c \int_M \Abs{\dd(f)}^2\dvol_M
$$
for every $f\in\cF$. In particular, every solution
of~\eqref{eq:ijk} is constant.
\end{lemma}

\begin{proof}
Throughout we abbreviate
$$
\Norm{f} := \sqrt{\int_M\Abs{f}^2\dvol_M},\qquad
\Norm{df} := \sqrt{\int_M\Abs{df}^2\dvol_M}.
$$
The Poincar\'e inequality asserts that there is a constant $C>0$
such that every smooth function $f:M\to\H^n$ satisfies
\begin{equation}\label{eq:poincare}
\int_Mf\,\dvol_M=0\qquad\implies\qquad
\Norm{f} \le C\Norm{df}.
\end{equation}
Since $\alpha$ is a Cartan structure equation~\eqref{eq:DDf} 
takes the form
\begin{equation}\label{eq:ddDD}
\dcD\dcD f = d^*df - \kappa\dcD f
\end{equation}
for $f:M\to\H^n$.  Here we write $\dd(f)=\dcD f$ 
because $X=\H^n$ is equipped with the standard flat metric
and $f\mapsto\dd(f)$ is a linear operator. Taking the 
inner product with $f$ we obtain 
\begin{eqnarray*}
\Norm{df}^2
&=& 
\int_M\inner{f}{\dcD\dcD f+\kappa\dcD f}\,\dvol_M \\
&\le & 
\Norm{\dcD f}^2 + \kappa\Norm{f}\Norm{\dcD f} \\
&\le & 
\Norm{\dcD f}^2 + \kappa C\Norm{df}\Norm{\dcD f} \\
&\le & 
\Norm{\dcD f}^2 + \frac12\Norm{df}^2 
+ \frac{\kappa^2 C^2}{2} \Norm{\dcD f}^2
\end{eqnarray*}
whenever $f$ has mean value zero.
By Lemma~\ref{le:acten}, this implies
$$
\sA(f) = \frac12\Bigl(\Norm{df}^2 - \Norm{\dcD f}^2\Bigr)
\le
\left(1+\kappa^2 C^2\right)\int_M\Abs{\dcD f}^2\dvol_M
$$
for every smooth map $f:M\to\H^n$. (We can drop the mean value zero 
condition by adding a constant to $f$.) Now the theorem of 
Geiges--Gonzalo~\cite{GG} shows that $M$ is a quotient of the $3$-sphere
by a finite subgroup of $\SU(2)$.  
If $M=S^3$ every smooth map $f:M\to X$ factors through 
a map to the universal cover $\H^n$ of~$X$ and the  
assertion follows. The general case follows from the 
special case for the induced map on the universal cover of~$M$.
\end{proof}

\begin{proof}[Proof of Theorem~\ref{thm:crit-cpct}]
By Lemma~\ref{le:apriori} the critical points of $\sA_H$ satisfy a uniform 
action bound. The action bound and the energy identity of Lemma~\ref{le:acten}
give a uniform $L^1$ bound on the functions $e_\nu:=\Abs{df_\nu}^2$.  
Since the exponent $\frac{3}{2}$ in the estimate~\eqref{eq:Ler} 
of Lemma~\ref{le:Ler} is less than the critical exponent $\frac{5}{3}$ 
we obtain from the Heinz trick (Theorem~\ref{thm:heinz})
a uniform $L^\infty$ bound on the sequence~$e_\nu$.  
Hence the result follows from Theorem~\ref{thm:cpct}.
\end{proof}

\begin{remark}\label{rmk:critcpct-T3}\rm
If $M$ is the $3$-torus then the assertion of 
Lemma~\ref{le:apriori} continues to hold for the 
contractible maps $f:M\to X$. In the noncontractible case
we may have nonconstant solutions of~\eqref{eq:critH}
and the estimate of Lemma~\ref{le:apriori} only holds
with an additional constant on the right.
\end{remark}

\begin{remark}\label{rmk:critcpct}\rm
Let $X$ be a K3 surface.  Then compactness fails 
for the critical points of $\sA_H$ even in the case $H=0$
and for sequences with bounded energy 
(see Remark~\ref{rmk:apriori}).
\end{remark}

\subsection*{Compactness and exponential decay for Floer trajectories} 

\begin{lemma}\label{le:bounded}
Let $M$ be a Cartan hypercontact $3$-manifold 
and $X$ be a compact hyperk\"ahler manifold. 
Let $H:X\times M\to\R$ be any smooth function
and ${u:\R\times M\to X}$ be a solution
of~\eqref{eq:floerH}. Then the following holds.

\smallskip\noindent{\bf (i)}
For every $s\in\R$ we have
\begin{equation}\label{eq:dudsu}
\frac12\int_M\Abs{du}^2
\le \sA(u(s,\cdot))
+ \Vol(M)\sup_{X\times M}\Abs{\nabla H}^2
+ \frac32\int_M\Abs{\p_su}^2.
\end{equation}

\smallskip\noindent{\bf (ii)}
If $u$ has finite energy
$$
\sE_H(u) = \int_{-\infty}^\infty\int_M
\Abs{\p_su}^2\,\kappa\,\dvol_M\,ds
<\infty
$$ 
and $\sup\Abs{du}<\infty$ then all the derivatives of $u$ are 
bounded on $\R\times M$ and $\p_su$ converges to zero 
in the $\Cinf$ topology as $s$ tends to $\pm\infty$. 

\smallskip\noindent{\bf (iii)}
If $X$ is flat then
$$
\sE_H(u)<\infty\qquad\implies\qquad \sup\Abs{du}<\infty.
$$
\end{lemma}

\begin{proof}
We prove~(i). By Lemma~\ref{le:acten}, every solution 
$u$ of~\eqref{eq:floerH} satisfies 
$$
\sE(u(s,\cdot)) 
= \sA(u(s,\cdot)) + \frac12\int_M\Abs{\nabla H(u)-\p_su}^2
$$
and hence
\begin{equation*}
\begin{split}
\frac12\int_M\Abs{du}^2
&=
\sE(u(s,\cdot)) + \frac12\int_M\Abs{\p_su}^2 \\
&\le 
\sA(u(s,\cdot))
+ \Vol(M)\sup_{X\times M}\Abs{\nabla H}^2
+ \frac32\int_M\Abs{\p_su}^2.
\end{split}
\end{equation*}
Here we have used the fact that the hypercontact structure 
on $M$ is a Cartan structure. This proves~(i).

We prove~(ii).  
Since $u$ satisfies~\eqref{eq:DDfloer} and $\Abs{du}$
is bounded the standard elliptic bootstrapping arguments 
as in the proof of Theorem~\ref{thm:reg} give uniform bounds 
on the higher derivatives of $u$.  Since $\Abs{du}$ is bounded 
it follows from Lemma~\ref{le:Les} that the function 
$\Abs{\p_su}^2$ satisfies an estimate of the form 
$$
\sL\Abs{\p_su}^2 \ge -C\Abs{\p_su}^2.
$$
This in turn implies that $u$ satisfies the mean
value inequality
$$
\Abs{\p_su(s_0,y)}^2\le 
c\int_{s_0-1}^{s_0+1}\int_M\Abs{\p_su}^2\,\dvol_M\,ds
$$
for a suitable constant $c>0$ (see Theorem~\ref{thm:heinz}
with $A=0$ and $\mu=\alpha=1$). Using the finite energy condition 
again we find that $\p_su$ converges to zero uniformly as $\Abs{s}$
tends to infinity. Convergence of the higher derivatives
of $\p_su$ follows from an elliptic bootstrapping argument 
using equation~\eqref{eq:Lsu}. This proves~(ii).

\bigbreak

We prove~(iii). Assume $X$ is flat. 
Then it follows from Lemma~\ref{le:Ler} that there are positive 
constants $A$ and $B$ such that
$$
\sL\Abs{du}^2 \ge -A - B\Abs{du}^3
$$
for every solution $u:\R\times M\to X$ of~\eqref{eq:floerH}.
Hence, by Theorem~\ref{thm:heinz}, there are positive constants 
$\hbar$ and $c$ such that every solution of~\eqref{eq:floerH} 
satisfies
\begin{equation}\label{eq:mean}
B^2\int_{B_r(z)}\Abs{du}^2 <\hbar
\quad\implies\quad
\Abs{du(z)}^2 \le c\left(
Ar^2 + \frac{1}{r^4}\int_{B_r(z)}\Abs{du}^2
\right)
\end{equation}
for $z\in\R\times M$ and $0<r\le1$. 
Now suppose $u:\R\times M\to X$ is a solution 
of~\eqref{eq:floerH} with finite energy
$\sE_H(u)<\infty$. Then the formula
$$
\int_{s_0}^{s_1}\int_M\Abs{\p_su}^2\,\kappa\,\dvol_M\,ds
= \sA_H(u(s_0,\cdot)) - \sA_H(u(s_1,\cdot))
$$
shows that there is a constant $C>0$ such that
$
\sA_H(u(s,\cdot))  \le C
$
for all~$s$. Explicitly we can choose
$C:=\sA(u(0,\cdot))+\sE_H(u)$.  Combining this 
with~\eqref{eq:dudsu} we obtain an inequality
\begin{equation}\label{eq:dudsu1}
\int_M\Abs{du}^2 \le c + 3\int_M\Abs{\p_su}^2
\end{equation}
for every $s\in\R$, where $c:= 2C + 2\Vol(M)\sup\Abs{\nabla H}$.
Next we choose $T>0$ so large that
$$
\int_T^\infty\int_M\Abs{\p_su}^2\,\dvol_M
< \frac{\hbar}{4B^2}.
$$
Then, for $z_0=(s_0,y_0)\in[T+1,\infty)\times M$  
and $r<\frac{\hbar}{8cB^2}$, we have 
\begin{eqnarray*}
\int_{B_r(z_0)}\Abs{du}^2
&\le& 
\int_{s_0-r}^{s_0+r}\int_M\Abs{du}^2\,\dvol_M\,ds \\
&\le & 
\int_{s_0-r}^{s_0+r}\left(
c+3\int_M\Abs{\p_su}^2\,\dvol_M\
\right)\,ds \\
&\le & 
2cr + 3\int_T^\infty\int_M\Abs{\p_su}^2\,\dvol_M\,ds \\
&\le &
2cr + \frac{3\hbar}{4B^2} 
< \frac{\hbar}{B^2}. 
\end{eqnarray*}
Here the second inequality follows from~\eqref{eq:dudsu1}
and the third from the fact that $s_0-r>T$.
The same estimate holds for $s_0\le -T-1$. 
Hence it follows from~\eqref{eq:mean} that $\Abs{du}$ 
is bounded. This proves the lemma.
\end{proof}

\begin{remark}\label{rmk:bounded}\rm
It is an open question if part~(iii) of Lemma~\ref{le:bounded} 
continues to hold without the hypothesis that $X$ is flat.  
\end{remark}

\begin{lemma}\label{le:T4bounded}
Let $M$ be a Cartan hypercontact $3$-manifold and 
$X$ be a compact flat hyperk\"ahler manifold.
Let $H:X\times M\to\R$ be any smooth function.  
Then there is a constant $c>0$ such that 
$$
-c\le\sA_H(u(s,\cdot))\le c
$$
for every finite energy solution $u:\R\times M\to X$
of~\eqref{eq:floerH} and every $s\in\R$. 
\end{lemma}

\begin{proof}
By Theorem~\ref{thm:crit-cpct},
there is a constant $c>0$ such that
$$
-c\le \sA_H(f)\le c
$$ 
for every critical point of $\sA_H$. Now let $u:\R\times M\to X$ 
be a finite energy solution of~\eqref{eq:floerH} and 
choose a sequence of real numbers $s^\nu\to-\infty$.
Passing to a subsequence we may assume that $u(s^\nu+\cdot,\cdot)$
converges, uniformly with all derivatives, to a solution
of~\eqref{eq:floerH} on the domain $[-1,1]\times M$. 
By~(i), this solution is a critical point of $\sA_H$.  Hence 
$$
\lim_{\nu\to\infty}\sA_H(u(s^\nu,\cdot))\le c.
$$
Since the action is nonincreasing along negative gradient flow 
lines this shows that $\sA(u(s,\cdot))\le c$ for all $s\in\R$.
The lower bound is obtained by the same argument for a 
sequence $s^\nu\to+\infty$.  
This proves the lemma. 
\end{proof}

\begin{theorem}[\bf Exponential decay]\label{thm:asymptotic}
Let $M$ be a Cartan hypercontact $3$-manifold and 
$X$ be a compact hyperk\"ahler manifold. 
Let $H:X\times M\to\R$ be a smooth function 
such that every solution of~\eqref{eq:critH} 
is nondegenerate. Let ${u:\R\times M\to X}$ be a solution 
of~\eqref{eq:floerH}. Then the following are equivalent. 

\smallskip\noindent{\bf (a)}
The energy $\sE_H(u)$ is finite and $\Abs{du}$ is bounded.

\smallskip\noindent{\bf (b)}
There are solutions $f^\pm:M\to X$ of equation~\eqref{eq:critH}
such that 
\begin{equation}\label{eq:limits}
\lim_{s\to\pm\infty} u(s,y) = f^\pm(y),\qquad
\lim_{s\to\pm\infty} \sA_H(u(s,\cdot)) = \sA_H(f^\pm),
\end{equation}
and $\lim_{s\to\pm\infty} \p_su(s,y) = 0$,
Moreover, the convergence is uniform in $y$
and~$\Abs{du}$ is bounded.

\smallskip\noindent{\bf (c)}
There are positive constants 
$\rho$ and $c_1,c_2,c_3,\dots$ such that
\begin{equation}\label{eq:exp}
\Norm{\p_su}_{C^\ell((\R\setminus[-T,T])\times M)}
\le c_\ell e^{-\rho T}
\end{equation}
for every $T>0$ and every integer $\ell\ge 0$.
Moreover, $\Abs{du}$ is bounded.
\end{theorem}

\begin{proof}
That~(c) implies~(a) is obvious. 
We prove that~(a) implies~(b). 
By Lemma~\ref{le:bounded} it follows from~(a)
that $\Abs{\p_su}$ converges to zero uniformly as $\Abs{s}$
tends to infinity and that $du$ is uniformly 
bounded with all its derivatives. Hence every 
sequence $s_\nu\to\pm\infty$ has a subsequence, still 
denoted by $s_\nu$, such that $u(s_\nu,\cdot)$
converges in the $\Cinf$ topology to a solution 
of~\eqref{eq:critH}.  Now it follows from the 
nondegeneracy of the critical points of $\sA_H$
that they are isolated.  Hence the limit is
independent of the sequence $s_\nu$. 
This proves~(b).

We prove that~(b) implies~(c). 
Consider the function $\phi:\R\to[0,\infty)$
defined by 
$$
\phi(s):= \frac12\int_M\kappa\Abs{\p_su}^2\,\dvol_M.
$$
By assumption, this function converges to zero as $s$
tends to $\pm\infty$. Moreover, its second derivative 
is given by
$$
\phi''(s) = \int_M\kappa\,\Abs{\Nabla{s}\p_su}^2 
+ \int_M\kappa\,\inner{\Nabla{s}\Nabla{s}\p_su}{\p_su}
$$
Denote by 
$$
\dcD_H := I\Nabla{v_1} + J\Nabla{v_2} + K\Nabla{v_3} - \nabla\nabla H(u)
$$
the covariant Hessian as in~\eqref{eq:hessH}.
Since the vector fields $v_i$ are independent of $s$
we have
$$
\Nabla{s}\p_su = -\Nabla{s}\dd_H(u) = -\dcD_H\p_su = \dcD_H\dd_H(u)..
$$
Differentiating this equation covariantly with respect to $s$
we obtain
$$
\Nabla{s}\Nabla{s}\p_su 
= \dcD_H\Nabla{s}\dd_H(u)+ [\Nabla{s},\dcD_H]\dd_H(u)
= \dcD_H\dcD_H\p_su - [\Nabla{s},\dcD_H]\p_su.
$$
Since $\dcD_H$ is self-adjoint with respect 
to the $L^2$ inner product with weight $\kappa$ 
this gives
$$
\phi''(s) = \int_M\kappa\,\Abs{\Nabla{s}\p_su}^2 
+ \int_M\kappa\,\Abs{\dcD_H\p_su}^2 
- \int_M\kappa\,\inner{[\Nabla{s},\dcD_H]\p_su}{\p_su}.
$$
Since $\Abs{du}$ is bounded we have an inequality
$$
\int_M\kappa\,\inner{[\Nabla{s},\dcD_H]\p_su}{\p_su}
\le c\Norm{\p_su}_{L^\infty(M)}\int_M\Abs{\p_su}^2.
$$
Moreover, by Lemma~\ref{le:bounded}, the bound on $\Abs{du}$ 
guarantees that $u(s,\cdot)$ converges in the $\Cinf$ topology 
to $f^\pm$ as $s$ tends to $\pm\infty$.  Since $f^\pm$ are 
nondegenerate critical points of $\sA_H$ we deduce that 
there is a constant $\rho>0$ such that, 
for $\Abs{s}$ sufficiently large, we have
$$
\int_M\kappa\,\Abs{\dcD_H\p_su}^2
\ge 2\rho^2 \int_M\Abs{\p_su}^2.
$$
Choosing $\Abs{s}$ so large that 
$
c\Norm{\p_su}_{L^\infty(M)}<\rho^2
$
we then obtain 
$$
\phi''(s) \ge \rho^2\phi(s).
$$
Hence
$$
\frac{d}{ds} e^{-\rho s}(\phi'(s)+\rho\phi(s)) 
= e^{-\rho s}(\phi''(s) - \rho^2\phi(s)) \ge 0.
$$
Since $\phi(s)\to0$ as $s\to\infty$ we must have 
$$
\rho\phi(s)+\phi'(s)\le0
$$ 
for all sufficiently large $s$ and hence 
$e^{\rho s}\phi(s)$ is nonincreasing.
This proves the exponential decay for $\phi$. 
To establish exponential decay for the higher derivatives 
one can use an elliptic bootstrapping argument 
based on equation~\eqref{eq:Lsu} to show that
the $L^\infty$ norm of $\p_su$ controls the higher derivatives.  
This proves the theorem.
\end{proof}

\begin{remark}\label{rmk:asymptotic}\rm
If $X$ is flat then the condition $\sup\Abs{du}<\infty$ in~(a--c)
in Theorem~\ref{thm:asymptotic} can be dropped.  
This follows from Lemma~\ref{le:bounded}~(iii) and 
the fact that each of the conditions~\eqref{eq:limits} and~\eqref{eq:exp}
guarantees finite energy.  Similarly, the next theorem 
continues to hold for general compact hyperk\"ahler manifolds 
if we impose the additional condition 
$\sup_\nu\sup_{\R\times M}\Abs{du^\nu}<\infty$.
\end{remark}

\begin{theorem}[\bf Compactness]\label{thm:compact}
Let $M$ be a Cartan hypercontact $3$-man\-i\-fold 
and $X$ be a compact flat hyperk\"ahler manifold.
Let $H:X\times M\to\R$ 
be a smooth function such that every solution $f$ of $\dd_H(f)=0$ is nondegenerate. 
Let $f^\pm$ be two distinct critical points of $\sA_H$
and $u^\nu$ be a sequence in $\sM(f^-,f^+;H)$.  
Then there is a subsequence (still denoted by~$u^\nu$), 
a catenation 
$$
u_1\in\sM(f_0,f_1;H),u_2\in\sM(f_1,f_2;H),\dots,u_N\in\sM(f_{N-1},f_N;H) 
$$
of Floer trajectories, and there are sequences
$
s^\nu_1<s^\nu_2<\cdots<s^\nu_N
$
such that
$$
f_0=f^-,\qquad f_N=f^+,\qquad \sA_H(f_{j-1}) > \sA_H(f_j), 
$$
and, for $j=1,\dots,N$, the shifted sequence 
$u^\nu(s^\nu_j+\cdot,\cdot)$ converges to $u_j$ uniformly with
all derivatives on every compact subset of $\R\times M$. 
\end{theorem}

\begin{proof}
By Lemma~\ref{le:bounded} the functions $u^\nu$ satisfy~\eqref{eq:mean} 
for suitable constants $A,B,c,\hbar$. This implies the following.

\medskip\noindent{\bf Energy quantization~I.}
{\it Let $x_0\in\R\times M$ and suppose that there 
is a sequence $x^\nu\to x_0$ such that $\Abs{du^\nu(x^\nu)}$ 
diverges to infinity.  Then
$$
\liminf_{\nu\to\infty}\int_{B_\eps(x_0)}
\Abs{du^\nu}^2 \ge \frac{\hbar}{B^2}
$$
for every $\eps>0$.}

\smallbreak

\medskip\noindent
The proof uses the Wehrheim trick. 
Suppose, by contradiction, that there is a
constant $\eps>0$ and a sequence $\nu_i\to\infty$ 
such that $B^2\int_{B_\eps(x_0)}\Abs{du^{\nu_i}}^2 < \hbar$
for every $i$.  Then we can use~\eqref{eq:mean}
with $x\in B_{\eps/2}(x_0)$ and $r=\eps/2$ to obtain
$$
\Abs{du^{\nu_i}(x)}^2 
\le c\left(\frac{A\eps^2}{4} 
+ \frac{16}{\eps^4}\int_{B_\eps(x_0)}\Abs{du^{\nu_i}}^2\right)
\le \frac{Ac\eps^2}{4} + \frac{16c\hbar}{B^2\eps^4}
$$
for all $x\in B_{\eps/2}(x_0)$ and $\nu\ge\nu_0$.
With $x=x^{\nu_i}$ it follows that the sequence 
$\Abs{du^{\nu_i}(x^{\nu_i})}$ is bounded,
a contradiction.  

\medskip\noindent{\bf Energy quantization~II.}
{\it Let $x_0=(s_0,y_0)\in\R\times M$ and suppose that there 
is a sequence $x^\nu=(s^\nu,y^\nu)\to (s_0,x_0)$ such that 
$\Abs{du^\nu(x^\nu)}$ diverges to infinity.  
Then
$$
\liminf_{\nu\to\infty}\int_{s_0-\eps}^{s_0+\eps}\int_M
\Abs{\p_su^\nu}^2 \ge \frac{\hbar}{3B^2}
$$
for every $\eps>0$.}

\smallbreak

\medskip\noindent
By Lemma~\ref{le:bounded}~(i) we have 
\begin{equation}\label{eq:Dudsu}
\int_M\Abs{\p_su^\nu}^2 \ge \frac13\int_M \Abs{du^\nu}^2 - c
\end{equation}
for some constant $c>0$ independent of $\nu$ and $s$.
The assertion follows by integrating this inequality from 
$s_0-\eps$ to $s_0+\eps$ and taking the limit $\eps\to 0$. 

\medskip\noindent
With this understood it follows that, after passing to a subsequence, 
we obtain divergence of the energy density at most near finitely 
many points.  On the complement of these finitely many points,
a further subsequence converges to a solution $u^\infty$ of $\partial_su^\infty + \dd_H(u^\infty)=0$ in the $\Cinf$ topology, 
by Theorem~\ref{thm:cpct}.  Now it follows from the 
inequality~\eqref{eq:Dudsu} that the $L^2$ norm of 
$du^\infty$ is finite on every compact subset of 
$\R\times M$ and in particular in a neighborhood of each
bubbling point.  Hence we can use the removable singularity 
theorem~\ref{thm:remsing} to deduce that the limit solution 
can be extended into the finitely many missing points.
The upshot is that, by standard arguments, we obtain a convergent 
subsequence as in the statement of the theorem, except that
$u^\nu(s^\nu_j+\cdot,\cdot)$ need only converge to $u_j$ in the 
complement of finitely many points.  If these bubbling 
points do exist we have 
$$
\sA_H(f_{j-1})-\sA_H(f_j) = \sE_H(u_j) \le 
\lim_{T\to\infty}\int_{s^\nu_j-T}^{s^\nu_j+T}\int_M\Abs{\p_su^\nu}^2
-\frac{\hbar}{B^2}
$$
for some $j$. However, this would imply that the sum of the 
energies $\sE_H(u_j)$ is strictly smaller than 
$\sE_H(u^\nu)=\sA_H(f^-)-\sA_H(f^+)$ which is clearly impossible.  
Thus bubbling cannot occur and the sequence
$\Abs{du^\nu}$ must remain uniformly bounded.
This proves the theorem.
\end{proof}

\begin{remark}\label{rmk:cpct}\rm
A key issue in developing the Floer theory of the action functional $\sA_H$ 
for general (compact) hyperk\"ahler manifolds is to extend 
Theorems~\ref{thm:crit-cpct} and~\ref{thm:compact} to the nonflat case.  
One then has to address the codimension-2 bubbling phenomenon for
finite energy sequences of solutions $f$ of $\dd_H(f)=0$ and $u$ of $\partial_su + \dd_H(u)=0$.
\end{remark}

    
\section{Moduli spaces and transversality} \label{sec:trans}  

\subsection*{Transversality for critical points}

Let $\sH:=\Cinf(X\times M)$ and, for $H\in\sH$, denote by 
$$
\sC(H) := \left\{f:M\to X\,|\,
f\mbox{ satisfies } \dd_H(f)=0 \right\}
$$
the set of critical points of $\sA_H$. Recall that 
a critical point $f\in\sC(H)$ is called nondegenerate 
if the Hessian
$$
\dcD_{f,H} := I\Nabla{v_1}+J\Nabla{v_2}+K\Nabla{v_3}-\nabla\nabla H(f)
$$
is bijective as an operator from $T_f\sF=\Om^0(M,f^*TX)$ 
to itself (respectively as an operator from 
$W^{k+1,p}(M,f^*TX)$ to $W^{k,p}(M,f^*TX)$).
Denote by
$$
\sH^{\morse}
:= \left\{H\in\sH\,|\,
\mbox{every critical point }f\in\sC(H)\mbox{ is nondegenerate}
\right\}
$$
of all $H\in\sH$ such that $\sA_H:\sF\to\R$ is a Morse function.

\begin{theorem}\label{thm:morse}
For every compact $3$-manifold $M$ with a positive hypercontact
structure and every hyperk\"ahler manifold $X$ the set 
$\sH^\morse$ is of the second category in $\sH$.
\end{theorem}

\begin{proof}
Fix an integer $\ell\ge2$ and abbreviate
$\sH^\ell := C^\ell(X\times M)$. Then the regularity 
argument in the proof of Theorem~\ref{thm:reg} shows that $f$ with $\dd_H(f)=0$ is of class $W^{\ell,p}$ for 
any $p<\infty$. Fix a constant $p>3$ and denote by 
$$
\sC^\ell := \left\{(f,H)\in W^{1,p}(M,X)\times\sH^\ell\,|\,
f\mbox{ satisfies } \dd_H(f)=0 \right\}
$$
the universal moduli space of critical points.
We prove that $\sC^\ell$ is a $C^{\ell-1}$ Banach manifold.
It is the zero set of a $C^{\ell-1}$ section of 
a Banach space bundle 
$$
\sE\to W^{1,p}(M,X)\times\sH^\ell
$$
with fibers $\sE_{f,H}:=L^p(M,f^*Tx)$. 
The section is given by 
$$
(f,H)\mapsto\dd_H(f)
$$ 
and we must prove that it is transverse to the zero section.  
Equivalently, the operator
\begin{equation}\label{eq:Dhxi}
W^{1,p}(M,f^*TX)\times\sH^\ell\to L^p(M,f^*TX),\
(\xi,h)\mapsto\dcD_{f,H}\xi - \nabla h(f)
\end{equation}
is surjective for every $H\in\sH^\ell$ and every
$f\in\sC(H)$. 

\smallbreak

Let $1/p+1/q=1$ and choose an element 
$\eta\in L^q(M,f^*TX)$ that annihilates the image 
of~\eqref{eq:Dhxi} in the sense that
$$
\int_M\inner{\eta}{\dcD_{f,H}\xi - \nabla h(f)}\,\kappa\,\dvol_M = 0
$$
for all $h\in\sH_\ell$ and $\xi\in W^{1,p}(M,f^*TX)$. 
Then, by elliptic regularity, we have 
$\eta\in W^{\ell,p}(M,f^*TX)$ and 
$$
\dcD_{f,H}\eta = 0,\qquad
\int_M\inner{\eta}{\nabla h(f)}\,\kappa\,\dvol_M = 0\qquad
\forall\;h\in\sH^\ell.
$$
In particular $\eta$ is continuous.
If $\eta\not\equiv0$ then it is easy to find a smooth function
$h:X\times M\to\R$ such that $\inner{\eta}{\nabla h(f)}\ge0$
everywhere on $M$ and $\inner{\eta}{\nabla h(f)}>0$ somewhere.
Namely, choose a point $y_0\in M$ with $\eta(y_0)\ne 0$
and a function $h_0:X\to\R$ such that 
$$
h_0(f(y_0))=0,\qquad \nabla h_0(f(y_0))=\eta(y_0).
$$ 
Then there is a neighborhood $U_0\subset M$ of $y_0$ 
such that 
$$
\inner{\eta(y)}{\nabla h_0(f(y))}>0
$$ 
for all $y\in U_0$.
Now choose a smooth cutoff function $\beta:M\to[0,1]$ 
with support in $U_0$ such that $\beta(y_0)=1$. 
Then the function $h(y,x):=\beta(y)h_0(x)$ has the 
required properties.  Thus we have proved that the 
operator~\eqref{eq:Dhxi} is always surjective and hence
$\sC^\ell$ is a $C^{\ell-1}$ Banach manifold as claimed.

Now the obvious projection 
$$
\pi^\ell:\sC^\ell\to\sH^\ell
$$
is a $C^{\ell-1}$ Fredholm map of index zero.
Since $\ell\ge 2$, it follows from the Sard--Smale theorem
that the set $\sH^{\morse,\ell}\subset\sH^\ell$ 
of regular values of $\pi^\ell$ is dense in $\sH^\ell$.
Now the result follows by the usual Taubes trick as explained,
for example, in~\cite[Chapter~3]{MS}.  Namely, 
for a constant $c>0$ we may introduce the set
$\sH^{\morse,\ell}_c$ of all $H\in\sH^\ell$ such that
the critical points $f\in\sC(H)$ with $\sup\Abs{df}\le c$
are nondegenerate. By Theorem~\ref{thm:crit-cpct}, 
this set is open in $\sH^\ell$ and
$
\sH^{\morse,\ell}=\bigcap_{c>0}\sH^{\morse,\ell}_c.
$
We then obtain with $\ell=\infty$ that each corresponding
set $\sH^\morse_c$ is open and dense in $\sH$ and so 
$$
\sH^\morse=\bigcap_{c>0}\sH^\morse_c
$$
is a countable intersection of open and dense sets in $\sH$.  
This proves the theorem.
\end{proof}

\subsection*{Fredholm theory}

The study of the spaces of solutions of~\eqref{eq:floerH}
is based on the linearized operators
$
\sD_{u,H}:W^{1,p}(\R\times M,u^*TX)\to L^p(\R\times M,u^*TX)
$
defined by
$$
\sD_{u,H} := \Nabla{s} + I\Nabla{v_1} 
+ J\Nabla{v_2} + K\Nabla{v_3} - \nabla\nabla H(u).
$$
It follows from the familiar arguments in Floer homology
that $\sD_{u,H}$ is a Fredholm operator whenever $f^\pm$ 
are nondegenerate critical points of $\sA_H$ 
and $u$ satisfies the exponential decay conditions 
of Theorem~\ref{thm:asymptotic}. It is also a standard result
that the Fredholm index of $\sD_{u,H}$ is given by the spectral 
flow of Atiyah--Patodi--Singer~\cite{APS}. 
More precisely, given a contractible critical point $f\in\sC(H)$
choose a smooth path $[0,1]\to\sF:t\mapsto f_t$ such that
$$
f_0\equiv\mathrm{constant},\qquad f_1=f
$$
and choose $\eps>0$ such that the negative eigenvalues 
of $\dcD_{f_0}$ are all less than~$-\eps$.  Now define
the integer $\mu_H(f)$ by the formula
\begin{equation}\label{eq:muH}
\mu_H(f) := -\mathrm{specflow}\left(
\left\{\dcD_{f_t,tH}+\eps(1-t)\one\right\}_{0\le t\le 1}
\right).
\end{equation}
It follows from equation~\eqref{eq:index} (with $\Sigma=S^1\times M$)
that this integer is independent of homotopy $t\mapsto f_t$ whenever 
$X$ is flat. If $f:M\to X$ is not contractible then the definition 
of the index $\mu_H(f)$ depends on the choice 
of a fixed reference map~$f_0$. 

\begin{proposition}\label{prop:mu}
Assume $H\in\sH^\morse$ and $f^\pm\in\sC(H)$.

\smallskip\noindent{\bf (i)} 
For every smooth map $u:\R\times M\to X$
satisfying~\eqref{eq:limitH} the operator 
$$
\sD_{u,H}:W^{1,p}(\R\times M,u^*TM)\to L^p(\R\times M,u^*TM)
$$
is Fredholm and its Fredholm index is 
$$
\mathrm{index}(\sD_{u,H}) = \mu_H(f^-)-\mu_H(f^+).
$$

\smallskip\noindent{\bf (ii)}
If $H:X\to\R$ is a Morse function with sufficiently small
$C^2$ norm and $f(y)\equiv x_0$ is a critical point of $H$ 
then $\mu_H(f)=\dim\,X-\mathrm{ind}_H(x_0)$ is equal to the 
Morse index of $x_0$ as a critical point of $-H$ 
(i.e.\ the number of positive eigenvalues 
of the Hessian of $H$ at $x_0$).
\end{proposition}

\begin{proof}
The Fredholm property in~(i) follows from standard arguments 
in Floer theory as in~\cite{DF,F0} in the instanton setting 
and in~\cite{F3,SF} in the symplectic setting.
The index identity is a well known result about the 
correspondence between the spectral flow and the 
Fredholm index (see~\cite{APS,RS2}). The second
assertion follows immediately from the 
definition of $\mu_H$.
\end{proof}

\subsection*{Transversality for Floer trajectories}

For $f^\pm\in\sC(H)$ we denote by $\sM(f^-,f^+;H)$
the moduli space of all solutions $u:\R\times M\to X$ 
of~\eqref{eq:floerH} and~\eqref{eq:limitH} for which 
$\Abs{du}$ is bounded.  To prove that these spaces are 
smooth manifolds we must show that the linearized operator
$\sD_{u,H}$ is surjective for every solution $u$ of 
equation~\eqref{eq:floerH} and~\eqref{eq:limitH}. Let
$$
\sH^\reg\subset\sH
$$
denote the set of all Hamiltonian perturbations $H\in\sH$
such that $\dcD_{f,H}$ is bijective for every critical
point $f\in\sC(H)$ of $\sA_H$ and $\sD_{u,H}$ is surjective 
for every ${u\in\sM(f^-,f^+;H)}$ and all $f^\pm\in\sC(H)$. 

\begin{theorem}\label{thm:trans}
For every compact $3$-manifold $M$ with a positive hypercontact 
structure and every hyperk\"ahler manifold $X$ the set $\sH^\reg$ 
is of the second category in $\sH$.  If $H\in\sH^\reg$ then 
the moduli space $\sM(f^-,f^+;H)$ is a smooth manifold of dimension
$$
\dim\,\sM(f^-,f^+;H) = \mu_H(f^-)-\mu_H(f^+)
$$
for every pair $f^\pm\in\sC(H)$. 
\end{theorem}

To prove this result we follow essentially the discussion in~\cite{FHS}.  
The first step is a unique continuation result.

\begin{proposition}\label{prop:ucon}
Let $\ell\ge 3$, $H\in\sH^\ell$, and 
${u_0,u_1:\R\times M\to X}$ be two $C^{\ell-1}$ 
solution of~\eqref{eq:floerH}.  If $u_0$ and $u_1$ 
agree to infinite order at a point $(s_0,y_0)\in\R\times M$ 
then they agree everywhere.
\end{proposition}

\begin{proof}
In local coordinates $x^1,\dots,x^m$ on $X$ and $y^1,y^2,y^3$
on $M$ both functions satisfy equation~\eqref{eq:floerH-local}.
For the difference 
$$
\hat u^k:=(u_1-u_0)^k
$$
in local coordinates this gives an estimate
$$
\Abs{\sL\hat u^k}
\le c\sum_{j=1}^m\left(
\Abs{\hat u^j}+\Abs{\frac{\p\hat u^j}{\p s}}
+\sum_{\nu=1}^3\Abs{\frac{\p \hat u^j}{\p y^\nu}}
\right),\qquad k=1,\dots,m.
$$
This is precisely the hypothesis of Aronszajn's theorem~\cite{A}.
Hence, if $\hat u$ vanishes to infinite order at a point 
it must vanish identically in a neighborhood of that point.  
This implies that the set of all points $(s,y)$ where 
$u_0$ and $u_1$ agree to infinite order is open and closed.
This proves the proposition.
\end{proof}

\begin{proposition}\label{prop:codim2}
Let $H\in\sH$ and $u:\R\times M\to X$ be a smooth map.
Let $\xi\in\Om^0(\R\times M,u^*TX)$ be a vector field along $u$
such that
$$
\sD_{u,H}\xi
= \Nabla{s}\xi + I\Nabla{v_1}\xi 
+ J\Nabla{v_2}\xi + K\Nabla{v_3}\xi - \Nabla{\xi}\nabla H(u)
=0.
$$  
If $\xi\not\equiv0$ then the set
$$
\sZ := \left\{(s,y)\in\R\times M\,|\,\xi(s,y)=0\right\}
$$
can be covered by countably many codimension $2$ 
submanifolds of $\R\times M$. In particular, the set 
$(\R\times M)\setminus\sZ$ is open, connected, 
and dense in $\R\times M$. 
\end{proposition}

\begin{proof}  
The proof has three steps.

\medskip\noindent{\bf Step~1.}
{\it If $\xi$ vanishes to infinite order 
at a point $(s_0,y_0)\in\R\times M$
then $\xi$ vanishes identically.}

\medskip\noindent
We use the identity
\begin{eqnarray*}
\sD^*\sD\xi + \sL^\nabla\xi 
&=& 
- \, I\Bigl(R(\p_su,\p_{v_1}u)\xi 
- R(\p_{v_2}u,\p_{v_3}u)\xi + \Nabla{[v_2,v_3]}\xi\Bigr)  \\
&& 
- \, J\Bigl(R(\p_su,\p_{v_2}u)\xi 
- R(\p_{v_3}u,\p_{v_1}u)\xi + \Nabla{[v_3,v_1]}\xi\Bigr)  \\
&& 
- \, K\Bigl(R(\p_su,\p_{v_3}u)\xi 
- R(\p_{v_1}u,\p_{v_2}u)\xi + \Nabla{[v_1,v_2]}\xi\Bigr).
\end{eqnarray*}
If $\sD\xi=\Nabla{\xi}\nabla H(u)$ we obtain an inequality 
of the form
$$
\Abs{\sL^\nabla\xi} \le c\left(
\Abs{\xi} + \Abs{\Nabla{s}\xi}
+ \sum_{j=1}^3 \Abs{\Nabla{v_j}\xi}
\right).
$$ 
In local coordinates the leading term of $\sL^\nabla$
has diagonal form. Hence the assertion of Step~1 follows from Aronszajn's 
theorem~\cite{A}.

\medskip\noindent{\bf Step~2.}
{\it Let $\sZ_k\subset\sZ$ denote the set where 
$\xi$ and its derivatives vanish up to order $k$.  
Then, for every $z_0=(s_0,y_0)\in\sZ_k\setminus\sZ_{k+1}$,
there is an open neighborhood $U_0\subset\R\times M$ 
and a codimension $2$ submanifold 
$V\subset\R\times M$ such that}
$$
\left(\sZ_k\setminus\sZ_{k+1}\right)\cap U_0\subset V
$$

\medskip\noindent
Fix an element $z_0\in\sZ_k\setminus\sZ_{k+1}$. 
For $\nu=(\nu_0,\nu_1,\nu_2,\nu_3)\in\N^4$ denote
$$
\nabla^\nu\xi := \Nabla{v_0}\cdots\Nabla{v_0}
\Nabla{v_1}\cdots\Nabla{v_1}
\Nabla{v_2}\cdots\Nabla{v_2}
\Nabla{v_3}\cdots\Nabla{v_3}\xi,
$$
where $v_0:=\p_s$ and each term $\Nabla{v_i}$ ocurs $\nu_i$ times.  
Since all derivatives of $\xi$ vanish up to order $k$ at the point $z_0$  
we have 
$$
\Nabla{v_i}\nabla^\nu\xi(z_0)=\nabla^\nu\Nabla{v_i}\xi(z_0).
$$
for $\Abs{\nu}:=\nu_0+\nu_1+\nu_2+\nu_3=k$ and $i=0,1,2,3$.  
Since $z_0\notin\sZ_{k+1}$ there is a multi index $\nu\in\N^4$ 
with $\Abs{\nu}=k$ and an $i\in\{0,1,2,3\}$ such that 
$\Nabla{v_i}\nabla^\nu\xi(z_0)\ne 0$.  Consider the vector field
$$
\eta := \nabla^\nu\xi
$$
along $u$.  Again using the fact that all derivatives of $\xi$ up
to order $k$ vanish at $z_0$ we obtain 
$$
\Nabla{v_0}\eta(z_0) + I\Nabla{v_1}\eta(z_0) 
+ J\Nabla{v_2}\eta(z_0) + K\Nabla{v_3}\eta(z_0) = 0.
$$
Since one of the vectors $\Nabla{v_i}\eta(z_0)$ is nonzero it follows
that the four vectors $\Nabla{v_i}\eta(z_0)$ cannot all be linearly 
independent.  Hence, in local coordinates $x^1,\dots,x^m$ on $X$
there exist indices $i,j\in\{1,\dots,m\}$ such that the differentials
of the functions $\eta^i,\eta^j$ (on an open neighborhood of $z_0$ 
in $\R\times M$) are linearly independent.
Hence, by the implicit function theorem, there is a neighborhood
$U_0\subset\R\times M$ of $z_0$ such that the set
$$
V:=\left\{x\in U_0\,|\,\eta^i(x)=\eta^j(x)=0\right\}
$$
is a codimension $2$ submanifold of $\R\times M$.  
Since 
$$
\left(\sZ_k\setminus\sZ_{k+1}\right)\cap U_0\subset V
$$
this proves Step~2. 

\medskip\noindent{\bf Step~3.}
{\it We prove the proposition.}

\medskip\noindent
By Step~1 we have
$$
\sZ = \bigcup_{k=0}^\infty \left(\sZ_k\setminus\sZ_{k+1}\right).
$$
By Step~2 each of the sets $\sZ_k\setminus\sZ_{k+1}$ can be covered by 
finitely many submanifolds of codimension $2$.  This proves 
the proposition.
\end{proof}

Let $H\in\sH^\ell$ and $u:\R\times M\to X$ be a $C^{\ell-1}$ 
solution of~\eqref{eq:floerH} and~\eqref{eq:limitH}. 
Call a point $(s,y)\in\R\times M$ {\bf regular} if
$$
\p_su(s,y)\ne 0,\qquad u(s,y)\ne f^\pm(y),\qquad
u(s,y)\notin u(\R\setminus\{s\},y).
$$
Let $\sR(u)\subset\R\times M$ denote the set of regular
points of $u$.

\begin{proposition}\label{prop:regular}
Fix an integer $\ell\ge 4$.
Let $H\in\sH^\ell$ and $u:\R\times M\to X$ be a $C^{\ell-1}$ 
solution of~\eqref{eq:floerH} and~\eqref{eq:limitH} 
with $f^-\ne f^+$. Then the set $\sR(u)$ of regular points 
of $u$ is open and dense in $\R\times M$. 
\end{proposition}

\begin{proof}
That the set $\sR(u)$ is open follows by the same argument 
as in the proof of~\cite[Theorem~4.3]{FHS}. 
We prove in four steps that $\sR(u)$ is dense.

\medskip\noindent{\bf Step~1.}
{\it The set
$$
\sR_0(u) := \left\{(s,y)\in\R\times M\,|\,
\p_su(s,y)\ne 0\right\}
$$
is open and dense in $\R\times M$.}

\medskip\noindent
The vector field $\p_su$ is in the kernel of the linearized operator
$\sD_{u,H}$ and is a vector field of class $C^{\ell-2}$ and hence of 
class $C^2$.  Now Step~1 in the proof of Proposition~\ref{prop:codim2}
continues to hold for $C^2$ vector fields and hence the set
$\sR_0(u)$ is dense in $\R\times M$.  That it is open is obvious.
This proves Step~1. 

\smallbreak

\medskip\noindent{\bf Step~2.}
{\it The set
$$
\sR_1(u) := \left\{(s,y)\in\sR_0(u)\,|\,
u(s,y)\ne f^\pm(y)\right\}
$$
is open and dense in $\R\times M$.}

\medskip\noindent
That the set is open is obvious. We prove it is dense. 
By Step~1 it suffices to prove that every point $(s,y)\in\sR_0(u)$
can be approximated by a sequence in $\sR_1(u)$. Because $\p_su(s,y)\ne 0$,
every sequence $(s_\nu,y)$ with $s_\nu\to s$ and $s_\nu\ne s$ 
belongs to the set $\sR_1(u)$ for $\nu$ sufficiently large.
This proves Step~2. 

\medskip\noindent{\bf Step~3.}
{\it The set
$$
\sR_2(u) := \left\{(s,y)\in\sR_1(u)\,|\,
u(s,y)\notin u(\R\times\{y\}\setminus\sR_0(u))\right\}
$$
is open and dense in $\R\times M$.}

\medskip\noindent
We prove that the set is open. Suppose, by contradiction, 
that there is an element $(s_0,y_0)\in\sR_2(u)$ and a sequence 
$(s_\nu,y_\nu)\in\sR_1(u)\setminus\sR_2(u)$ converging to $(s_0,y_0)$.
Since $(s_\nu,y_\nu)\notin\sR_2(u)$ there is an $s_\nu'\in\R$ such that 
$$
\p_su(s_\nu',y_\nu)=0,\qquad 
u(s_\nu',y_\nu)=u(s_\nu,y_\nu).
$$
The sequence $s_\nu'$ must be bounded; for if $s_\nu'\to\pm\infty$
then $u(s_\nu',y_\nu)$ converges to $f^\pm(y_0)$ and this 
implies $u(s_0,y_0)=f^\pm(y_0)$, a contradiction.
Thus, passing to a subsequence, we may assumne that
$s_\nu'$ converges to a point $s_0'\in\R$. It then follows 
that $u(s_0,y_0)=u(s_0',y_0)$ and $\p_su(s_0',y_0)=0$, 
contradicting the fact that $(s_0,y_0)\in\sR_2(u)$. 

We prove that the set $\sR_2(u)$ is dense in $\R\times M$. 
It suffices to prove that every element $(s_0,y_0)\in\sR_1(u)$
can be approximated by a sequence in $\sR_2(u)$. 
If this is not the case for some element $(s_0,y_0)\in\sR_1(u)$
then there is an $\eps>0$ such that the following holds:
$$
\Abs{s-s_0}<\eps\;\implies\;
\exists s'\in\R\mbox{ such that }
u(s,y_0)=u(s',y_0),\;\;\p_su(s',y_0)=0.
$$
However this contradicts Sard's theorem.  Namely for $\eps$
small the curve 
$$
\Gamma:=\{u(s,y_0)\,|\,\Abs{s-s_0}<\eps\}
$$
is a one dimensional submanifold of $X$ and we can choose 
a projection $\pi:U\to\Gamma$ on a suitable tubular neighborhood.
Consider the open set $S:=\left\{s\in\R\,|\,u(s,y_0)\in U\right\}$.
The assertion would then mean that every element of $\Gamma$
is a singular value of the map ${S\to\Gamma:s\mapsto \pi(u(s,y_0))}$.  
By Sard's theorem, this is impossible whenever $u$ is $C^1$.
This proves Step~3.

\medskip\noindent{\bf Step~4.}
{\it The set $\sR(u)$ is open and dense in $\R\times M$.}

\medskip\noindent
We have already observed that the set is open. We prove it is dense.
By Step~3, it suffices to prove that every element of $\sR_2(u)$ can be 
approximated by a sequence in $\sR(u)$. Suppose, by contradiction,
that this is not the case for some element $(s_0,y_0)\in\sR_2(u)$.
Then there is an open neighborhood $U\subset M$ of $y_0$
and two positive real number $\eps,T$ such that the follwing holds.
We abbreviate $I:=(s_0-\eps,s_0+\eps)$.
\begin{description}
\item[(a)]
$I\times U\subset\sR_2(u)\setminus\sR(u)$.
\item[(b)]
$u(s,y)\notin u(I\times U)$ for $\Abs{s}\ge T$ and $y\in U$.
\item[(c)]
The map $\bar I\to X:s\mapsto u(s,y)$ is an embedding
for every $y\in U$. 
\end{description}
Since $I\times U\subset\sR_1(u)$, the condition 
$I\times U\cap\sR(u)=\emptyset$ means that for every 
$(s,y)\in I\times U$ there is an $s'\in\R\setminus\{s\}$
such that $u(s',y)=u(s,y)$. Since $(s,y)\in\sR_2(u)$ 
we must have $\p_su(s',y)\ne0$ and, by~(b), we have 
${\Abs{s'}\le T}$.  Hence there can only be finitely many
such points $s'$.  For $s=s_0$ let 
$s_1<\dots<s_N$ be the points in $[-T,T]\setminus\{s_0\}$
with
$$
u(s_0,y_0) = u(s_1,y_0) = \cdots = u(s_N,y_0).
$$
Choose $r>0$ so small that the map 
$[s_j-r,s_j+r]\to X:s\mapsto u(s,y_0)$
is an embedding for every $j$.  Shrinking $U$ if necessary, 
we may assume that this continues to hold for every $y\in U$.

Next we claim that there is a $\delta>0$ and a 
compact neighborhood $V\subset U$ of $y_0$ such that
$$
y\in V\quad\implies\quad
u([s_0-\delta,s_0+\delta],y)
\subset \bigcup_{j=1}^Nu([s_j-r,s_j+r],y).
$$
If this were not the case, we could find sequences 
$(s_\nu,y_\nu)\to(s_0,y_0)$ and $s_\nu'\in\R\setminus\{s_\nu\}$
such that $u(s_\nu,y_\nu)=u(s_\nu',y_\nu)$ and 
$\Abs{s_\nu'-s_j}>r$ for all $j$ and $\nu$.
By taking the limit $s_\nu'\to s'$ we would then
obtain another element $s'\notin\{s_0,\dots,s_N\}$
with $u(s,y_0)=u(s_0,y_0)$, a contradiction.

Now define the set 
$$
\Sigma_j:=\left\{(s,y)\in[s_0-\delta,s_0+\delta]\times V\,|\,
u(s,y)\in u([s_j-r,s_j+r],y)\right\}
$$
for $j=1,\dots,N$. These sets are closed and their union is 
the entire set $[s_0-\delta,s_0+\delta]\times V$.
Hence, by Baire's category theorem, at least one of the
sets $\Sigma_j$ must have nonempty interior. 
Assume without loss of generality that $\Sigma_1$ has 
nonempty interior and that $(s_0,y_0)\in\INT(\Sigma_1)$.
Choose a neighborhood $W\subset V$ of $y_0$ and 
a constant $\rho>0$ such that 
$$
(s_0-\rho,s_0+\rho)\times W\subset\Sigma_1.
$$
Then for every pair $(s,y)\in(s_0-\rho,s_0+\rho)\times W$
there is a unique element $s'=:\sigma(s,y)\in[s_1-r,s_1+r]$
such that 
$$
u(s,y) = u(\sigma(s,y),y).
$$
This map $\sigma$ is evidently $C^{\ell-1}$ and satisfies 
$\sigma(s_0,y_0)=s_1$.  Moreover, 
\begin{equation*}
\begin{split}
0 
&=
\p_su(s,y) -\nabla H(u(s,y)) 
+ I\p_{v_1}u(s,y) + J\p_{v_2}u(s,y) 
+ K\p_{v_3}u(s,y)  \\
&= 
(\p_s\sigma)\p_su(\sigma,y) - \nabla H(u(\sigma,y))
+ I\Bigl(\p_{v_1}u(\sigma,y) + (\p_{v_1}\sigma)\p_su(\sigma,y)\Bigr) \\
&\quad
+ J\Bigl(\p_{v_2}u(\sigma,y) + (\p_{v_2}\sigma)\p_su(\sigma,y)\Bigr)
+ K\Bigl(\p_{v_3}u(\sigma,y) + (\p_{v_3}\sigma)\p_su(\sigma,y)\Bigr) \\
&=
\Bigl((\p_s\sigma-1)\one
+ \p_{v_1}\sigma I
+ \p_{v_2}\sigma J
+ \p_{v_3}\sigma K
\Bigr)\p_su(\sigma,y).
\end{split}
\end{equation*}
Since $\p_su\ne 0$ the four vectors $\p_su$, $I\p_su$, 
$J\p_su$, $K\p_su$ are linearly independent and thus we 
obtain $\p_{v_i}\sigma\equiv0$ for $i=1,2,3$ and 
$\p_s\sigma\equiv1$.  This means that 
$$
\sigma(s,y) = s+s_1-s_0.
$$
In other words, the solution $(s,y)\mapsto u(s+s_1-s_0,y)$
of~\eqref{eq:floerH} agrees with $u$ on an open set.
By Proposition~\ref{prop:ucon}, this implies $u(s,y)=u(s+s_1-s_0,y)$ 
for all $s$ and $y$. Hence $f^+=f^-$, a contradiction. 
This proves the proposition. 
\end{proof}

\begin{proof}[Proof of Theorem~\ref{thm:trans}]
Fix a constant $p>4$.  There is a Banach manifold 
$\sB=\sB(f^-,f^+)$ of all continuous maps $u:\R\times M\to X$ 
that are locally of class $W^{1,p}$ and, near infinity, can 
be written as 
$$
u(s,y)=\exp_{f^\pm(y)}(\xi^\pm(s,y))
$$
with 
$
\xi^+\in W^{1,p}([T,\infty)\times M,(f^+)^*TX)
$
and similarly for $\xi^-$. Fix an element $H_0\in\sH^\morse$.
Following Floer~\cite{F3} we choose a separable
Banach space ${\sH_0\subset\sH}$ of smooth functions 
${h:X\times M\to\R}$ satisfying the following axioms.

\medskip\noindent{\bf (I)}
If $f\in\sC(H_0)$ and $h\in\sH_0$ then $h$ vanishes to 
infinite order at the point $(f(y),y)$ for every $y\in M$. 

\medskip\noindent{\bf (II)}
Let $(x,y)\in X\times M$ such that 
$y\ne f(x)$ for every $f\in\sC(H_0)$.
Let $A:T_xX\to\R$ be a linear map. 
Then there are smooth functions ${h:X\to\R}$, 
${\alpha_x:X\to[0,1]}$, and ${\beta_x:M\to[0,1]}$ 
such that the following holds.

\medskip\noindent{\bf (a)}
$h(x)=0$ and $dh(x)=A$. Moreover, $\alpha$ and $\beta$ 
are supported in the balls of radius~$1$ about $x$ and $y$, 
respectively, and $\alpha(x)=\beta(y)=1$.

\medskip\noindent{\bf (b)}
For $\delta,\eps>0$ define $\alpha_x^\delta:X\to[0,1]$ 
and $\beta_y^\eps:M\to[0,1]$ by 
$$
\alpha_x^\delta(\exp_x(\xi))
:= \alpha_x(\exp_x(\delta^{-1}\xi)),\qquad
\beta_y^\eps(\exp_y(\eta))
:= \beta_y(\exp_y(\eps^{-1}\eta)).
$$
Then the function $h^{\delta,\eps}:M\times X\to\R$
given by 
$$
h^{\delta,\eps}(x',y'):=\alpha_x^\delta(x')\beta_y^\eps(y')h(x')
$$
belongs to $\sH_0$ for $\delta,\eps$ positive and sufficiently small.

\bigbreak

\medskip\noindent
To define the space $\sH_0$ we choose a smooth cutoff function 
$\rho:[0,\infty)\to[0,1]$ such that $\rho(r)=1$ for $r$ 
sufficiently small and $\rho(r)=0$ for $r\ge r_0$, where
$r_0$ is smaller than the injectivity radii of $X$ and $M$.
For $x\in X$ and $y\in M$ define 
$$
\alpha_x(\exp_x(\xi)) := \rho(\Abs{\xi}),\qquad
\beta_y(\exp_y(\eta)) := \rho(\Abs{\eta}).
$$
Then define $\sH_0$ to be the set of all smooth functions
$h:X\times M\to\R$ that vanish to infinite order along 
the graph of any element $f\in\sC(H_0)$ and such that
$$
\Norm{h}_c:= \sum_{\ell=0}^\infty c_\ell^{-1}\Norm{h}_{C^\ell}
<\infty,\qquad
c_\ell := 2^{2^\ell}\left(
\sup_x\Norm{\alpha_x}_{C^\ell}
+ \sup_y\Norm{\beta_y}_{C^\ell}
\right).
$$
This space satisfies~(I) and~(II).

\medskip\noindent
Consider the universal moduli space 
$$
\sM_0(f^-,f^+) := \left\{
(u,H_0+h)\in\sB\times\sH\,|\,h\in\sH_0,\,
u\in\sM(f^-,f^+;H)\right\}.
$$
This space is the zero set of a smooth section of the Banach space bundle 
$$
\sE\to\sB\times(H_0+\sH_0)
$$
with fibers $\sE_{u,H}=L^p(\R\times M,u^*TX)$.  
The section is $(u,H)\mapsto\p_su+\dd_H(u)$ and the claim below asserts 
that it is transverse to the zero section. Hence $\sM_0(f^-,f^+)$ is a 
smooth Banach manifold.  Now the obvious projection 
$$
\pi_0:\sM_0(f^-,f^+)\to H_0+\sH_0
$$
is a Fredholm map. Hence, by the Sard--Smale theorem, the set 
of regular values of $\pi_0$ is of the second category
in the sense of Baire in $H_0+\sH_0$.  Thus 
the set $\sH^\reg$ is dense in $\sH$. Now we may 
introduce sets $\sH^\reg_c\supset\sH^\reg$ for $c>0$,
as in the proof of Theorem~\ref{thm:morse},
where the requirement of transversality is restricted to a compact
set of Floer trajectories.  These sets are all open and, 
by what we have just proved, they are also dense in $\sH$.  
It then follows that $\sH^\reg$ is the intersection of countably 
many open and dense sets $\sH^\reg_c$ for $c=1,2,3,\dots$
and hence is of the second category in the sense of Baire.

\medskip\noindent{\bf Claim.}
{\it The operator 
$$
W^{1,p}(\R\times M,u^*TX)\times\sH_0
\to L^p(\R\times M,u^*TX),\ 
(\xi,h)\mapsto\sD_{u,H}-\nabla h(u)
$$
is surjective for every $H\in\sH$ 
and every $u\in\sM(f^-,f^+;H)$. }

\medskip\noindent
Let $1/p+1/q=1$ and suppose $\eta\in L^q(\R\times M,u^*TM)$ 
annihilates the image of the operator in the sense that
$$
\int_{-\infty}^\infty\int_M
\inner{\eta}{\sD_{u,H}\xi - \nabla h(u)}\,\kappa\,\dvol_M\,ds = 0
$$
for all $h\in\sH_0$ and $\xi\in W^{1,p}(\R\times M,u^*TX)$. 
Then $\eta$ is smooth and 
\begin{equation}\label{eq:trans}
\sD_{u,H}^*\eta = 0,\qquad
\int_{\R\times M}
\inner{\eta}{\nabla h(u)}\,\kappa\,\dvol_M\,ds = 0
\end{equation}
for all $h\in\sH_0$. We prove in three steps 
that $\eta$ vanishes identically.

\medskip\noindent{\bf Step~1.}
{\it For every $s\in\R$ we have
$\int_M\inner{\eta}{\p_su}\kappa\,\dvol_M=0$.}

\medskip\noindent
Since $\sD_{u,H}\p_su=0$ and $\sD_{u,H}^*\eta=0$ we have
\begin{eqnarray*}
\frac{d}{ds}
\int_M\inner{\eta}{\p_su}\,\kappa\,\dvol_M
&=&
\int_M\Bigl(
\inner{\eta}{\Nabla{s}\p_su}
+ \int_M\inner{\Nabla{s}\eta}{\p_su}
\Bigr)\,\kappa\,\dvol_M \\
&=&
\int_M\Bigl(
\inner{\eta}{\sD_{u,H}\p_su} 
- \int_M\inner{\sD_{u,H}^*\eta}{\p_su}
\Bigr)\,\kappa\,\dvol_M  \\
&=&
0.
\end{eqnarray*}
Here we have used the formulas $\sD_{u,H}=\Nabla{s}+\dcD_{u,H}$,
$\sD_{u,H}^*=-\Nabla{s}+\dcD_{u,H}$, and the fact that
$\dcD_{u,H}$ is self-adjoint. Since $\eta\in L^q$ and 
$\p_su\in L^p$, their inner product over $\R\times M$ 
is finite and this proves Step~1.

\medskip\noindent{\bf Step~2.}
{\it $\eta(s,y)$ and $\p_su(s,y)$ are linearly 
dependent for all $(s,y)\in\R\times M$.}

\medskip\noindent
Suppose otherwise that $\p_su(s_0,y_0)$ and $\eta(s_0,y_0)$ 
are linearly independent for some $(s_0,y_0)\in\R\times M$. 
By Proposition~\ref{prop:regular} we may assume
${(s_0,y_0)\in\sR(u)}$. Choose a compact interval $I\subset\R$ 
containing $s_0$ in its interior such that $I\times\{y_0\}\subset\sR(u)$ 
and $I\to X:s\mapsto u(s,y_0)$ is an embedding.
Then there are open neighborhoods 
$U\subset X$ of $u(s_0,y_0)$ and $V\subset M$ of $y_0$
such that
\begin{description}
\item[($*$)]
if $y\in V$ and $s\in\R$ such that $u(s,y)\in U$
then $s\in I$.
\end{description}
Otherwise there are sequences $s_\nu\in\R\setminus I$ 
and $y_\nu\to y_0$ such that $u(s_\nu,y_\nu)$ converges to $u(s_0,y_0)$.
If $s_\nu$ is unbounded then $u(s_0,y_0)\in\{f^-(y_0),f^+(y_0)\}$,
which is impossible because $(s_0,y_0)\in\sR(u)$.
Thus the sequence $s_\nu$ is bounded and hence 
has a limit point $s\in\R\setminus\INT(I)$ with $u(s,y_0)=u(s_0,y_0)$.  
Since $s\ne s_0$ and $(s_0,y_0)\in\sR(u)$ this is a contradiction. 

Since $\p_su(s_0,y_0)$ and $\eta(s_0,y_0)$ are linearly
independent, hypothesis~(II) on the space $\sH_0$ asserts 
that there is a smooth function $h_0:X\to\R$ 
and smooth cutoff functions $\alpha:X\to[0,1]$ and 
$\beta:M\to[0,1]$, centered at $x_0:=u(s_0,y_0)$ and $y_0$,
respectively, such that
\begin{equation*}
h_0(u(s_0,y_0))=0,\;
\left.\frac{\p}{\p s}\right|_{s=s_0}h_0(u(s,y_0)) = 0,\;
dh_0(u(s_0,y_0))\eta(s_0,y_0)=1,
\end{equation*}
and such that the function $h^{\delta,\eps}$ defined by
$$
h^\delta(x,y) := \alpha^\delta(x)h_0(x),\qquad
h^{\delta,\eps}(x,y):=\beta^\eps(y)h^\delta(x,y),
$$
is an element of $\sH_0$ for $\delta,\eps$ sufficiently small.
If $\delta$ and $\eps$ are so small that
$B_\delta(u(s_0,y_0))\subset U$ and $B_\eps(y_0)\subset V$
then the function $(s,y)\mapsto h^{\delta,\eps}(u(s,y),y)$ 
is supported in $I\times V$. Namely, if 
$h^{\delta,\eps}(u(s,y),y)\ne 0$ then $u(s,y)\in U$
and $y\in V$ and hence $s\in I$, by~($*$).  

Next we prove that
\begin{equation}\label{eq:dheta}
\int_\R dh^\delta_{y_0}(u(s,y_0))\eta(s,y_0)\,ds >0
\end{equation}
for $\delta>0$ sufficiently small. To see this we observe 
that there is a constant $c>0$, independent of $\delta$, 
such that the following holds. First, 
$$
\Abs{s-s_0}\le\frac{\delta}{c}\quad\implies\quad
\alpha^\delta(u(s,y_0))dh_{y_0}(u(s,y_0))\eta(s,y_0)
\ge \frac12,
$$
because 
$$
\alpha(u(s_0,y_0))=dh_{y_0}(u(s_0,y_0))\eta(s_0,y_0)=1
$$
and hence the condition ${\Abs{s-s_0}\le\delta/c}$ with $c$ sufficiently 
large guarantees that ${\alpha^\delta(u(s,y_0))\ge 3/4}$ and
$dh_{y_0}(u(s,y_0))\eta(s,y_0)\ge 2/3$. Second,
$$
\Abs{h_{y_0}(u(s,y_0))d\alpha^\delta(u(s,y_0))\eta(s,y_0)}
\le \frac{c\Abs{s-s_0}^2}{\delta}\le c^3\delta,
$$
because the function $s\mapsto h_{y_0}(u(s,y_0))$
vanishes to first order at $s=s_0$ and the first 
derivative of $\alpha^\delta$ is bounded by a 
constant times $1/\delta$. The last inequality
follows from the fact that $d\alpha^\delta(u(s,y_0))=0$
for $\Abs{s-s_0}\ge c\delta$. Both estimates taken 
together show that
$$
\int_\R dh^\delta_{y_0}(u(s,y_0))\eta(s,y_0)\,ds 
\ge \int_{s_0-c\delta}^{s_0+c\delta}
\left(\frac12-c^3\delta\right)\,ds
= 2c\delta\left(\frac12-c^3\delta\right).
$$
Thus~\eqref{eq:dheta} holds for $\delta<1/2c^3$. 

Now choose $\eps$ so small that 
$$
\int_\R dh^\delta_{y}(u(s,y))\eta(s,y)\,ds >0
$$
for every $y\in M$ with $d(y_0,y)<\eps$. Then 
the integral in~\eqref{eq:trans} is positve for the function 
$h(x,y)=h^{\delta,\eps}(x,y)=\beta^\eps(y)h^\delta(x,y)$.
This proves Step~2. 

\medskip\noindent{\bf Step~3.}
{\it $\eta$ vanishes identically.}

\medskip\noindent
Assume, by contradiction, that $\eta\not\equiv0$.  Then, by 
Proposition~\ref{prop:codim2}, the set
$$
\sU:=\left\{(s,y)\in\R\times M\,|\,\p_su(s,y)\ne 0,\,\eta(s,y)\ne 0\right\}
$$
is nonempty, open, and connected. By Step~2, there is a
continuous function $\lambda:\sU\to\R\setminus\{0\}$ such that 
$\eta(s,y) = \lambda(s,y)\p_su(s,y)$ for all $(s,y)\in\sU$.
Since $\sU$ is connected, by Proposition~\ref{prop:codim2}, 
the function $\lambda$ cannot change sign.  
Suppose $\lambda>0$ on $\sU$. 
(Otherwise replace $\eta$ by $-\eta$.) 
Then 
$$
\inner{\eta}{\p_su} = \lambda\Abs{\p_su}^2 > 0
$$
on $\sU$ and $\inner{\eta}{\p_su}=0$ on $\R\times M\setminus\sU$.
This contradicts Step~1 and proves Step~3, the claim, and the 
first assertion of the theorem.  The second assertion follows 
from Proposition~\ref{prop:mu} and the infinite dimensional
implicit function theorem.
\end{proof}

The above proof follows essentially the argument 
in~\cite[Theorem~5.1]{FHS}.  There are, however, a few subtle
but important differences.  In the present setting we cannot remove 
the Hamiltonian term $\nabla H$ from the equation by a change of 
coordinates. Second, in symplectic Floer theory the complement
$\sZ:=(\R\times M)\setminus\sU$ of the set $\sU$ in Step~3 is discrete.  
This is replaced in the present context by the codimenion
$2$ property of Proposition~\ref{prop:codim2}. 
In~\cite{FHS} the proof argues that $\p_s\lambda\equiv0$
and, because $\sZ$ is discrete, that $\lambda$ can 
therefore be defined globally on $\R\times M$ 
(and not just on $\sU$). The condition $\p_s\lambda\equiv0$ 
can also be obtained in the present case by the same 
argument, but we do not need it to obtain the contradiction.

The idea for the proof of the codimension $2$ result was 
communicated to the third author, several years ago, 
by Kim Froyshov (in the context of Seiberg--Witten theory).
This requires smooth perturbations and therefore 
we cannot work with the $C^\ell$ argument
as in the proof of Theorem~\ref{thm:morse} but must instead 
use Floer's Banach spaces of smooth functions. 
As a result the construction of the function $h$ 
in Step~2 above is somewhat less explicit than
in the proof of~\cite[Theorem~5.1]{FHS}.


\section{Floer homology} \label{sec:floer}  

We assume throughout that $M$ is a compact Cartan hypercontact 
$3$-man\-i\-fold and $X$ is a compact flat hyperk\"ahler manifold.  
For $H\in\sH^\reg$ we introduce the 
chain complex 
$$
\CF_k(M,X;H) := \bigoplus_{\stackrel{f\in\sC(H)}{ \mu_H(f)=k}}\Z_2 \langle f\rangle.
$$
This group is finitely generated by Theorem~\ref{thm:crit-cpct}.  
It is graded by the index function in equation~\eqref{eq:muH}
$$
\mu_H:\sC(H)\to\Z.
$$ 
Since $H\in\sH^\reg$, Theorem~\ref{thm:trans} 
asserts that the moduli space $\sM(f^-,f^+;H)$ is a smooth manifolds
of dimension $\mu_H(f^-)-\mu_H(f^+)$ for every pair $f^\pm\in\sC(H)$.
The real numbers act on these spaces by time shift and
it follows from Theorem~\ref{thm:compact} that
$$
\mu_H(f^-)-\mu_H(f^+)=1\qquad\implies\qquad
\#\sM(f^-,f^+;H)/\R < \infty.
$$
Thus we can use the numbers 
$$
n_2(f^-,f^+) :=  \#\sM(f^-,f^+;H)/\R\;\;(\mathrm{modulo}\;2)
$$
to define a boundary operator 
$\p^H:\CF_k(M,X;H)\to\CF_{k-1}(M,X;H)$ by 
$$
\p^H \langle f^-\rangle
:= \sum_{\stackrel{f\in\sC(H)}{ \mu_H(f^+)=k-1} }
n_2(f^-,f^+)\langle f^+\rangle
$$
for $f^-\in\sC(H)$ with $\mu_H(f^-)=k$.

\begin{theorem}\label{thm:floer}
For every $H\in\sH^\reg$ we have $\p^H\circ\p^H=0$.
\end{theorem}

To prove this one just needs to observe that the standard 
Floer gluing argument~\cite{DF,MS,SF} carries over verbatim 
to the present setting. The Floer homology groups of $(M,X;H)$ 
are now defined by
$$
\HF_k(M,X;H) := \frac{\ker\,\p^H:\CF_k(M,X;H)\to\CF_{k-1}(M,X;H)}
{\im\,\p^H:\CF_{k+1}(M,X;H)\to\CF_k(M,X;H)}.
$$
It follows again from the familiar arguments in symplectic Floer
theory that these Floer homology groups are independent of the 
choice of the Hamiltonian perturbation $H\in\sH^\reg$. 
Here one can follow verbatim the discussion in~\cite{F5,SZ},
using the solutions of~\eqref{eq:floerH} with $H$ depending 
on $s$, to prove the following theorem. 

\begin{theorem}\label{thm:iso}
There is a natural family of isomorphisms
$$
\Phi^{\beta\alpha}:\HF_*(M,X;H^\alpha)\to\HF_*(M,X;H^\beta),
$$
one for each pair $H^\alpha,H^\beta\in\sH^\reg$, such that
$$
\Phi^{\gamma\beta}\circ\Phi^{\beta\alpha}=\Phi^{\gamma\alpha},\qquad
\Phi^{\alpha\alpha} = \id.
$$
\end{theorem}

\begin{theorem}\label{thm:HF=H}
Let $X$ be a compact flat hyperk\"ahler manifold.
Then there is a natural family of isomorphisms
$$
\Phi^\alpha:H_*(X;\Z_2)\to\HF_*(M,X;H^\alpha),
$$
one for every $H^\alpha\in\sH^\reg$, such that
$$
\Phi^\beta = \Phi^{\beta\alpha}\circ\Phi^\alpha.
$$
\end{theorem}

The proof of Theorem~\ref{thm:HF=H} is based on the following result 
which asserts that the Floer chain complex agrees with the Morse 
complex for a special class of perturbations. 

\begin{theorem}\label{thm:morse=floer}
Let $M$ be a compact Cartan hypercontact $3$-manifold 
and $X$ be a compact flat hyperk\"ahler manifold.
Let ${H:X\to\R}$ be a~Morse function whose gradient flow
is Morse--Smale. Then there is a constant $\eps_0>0$ 
such that the following holds for~${0<\eps\le\eps_0}$.  
If $x^\pm$ are critical points of $H$ with index difference
$\mathrm{ind}_H(x^+)-\mathrm{ind}_H(x^-)\le1$ 
and $u:\R\times M\to X$ is a finite energy solution 
of the Floer equation
\begin{equation}\label{eq:floerHeps}
\p_su+\eps^{-1}\dd(u) = \nabla H(u),\qquad
\lim_{s\to\pm\infty}u(s,y) = x^\pm,
\end{equation}
then $\mathrm{ind}_H(x^+)-\mathrm{ind}_H(x^-)=1$, the function
$u(s,y)$ is independent of $y\in M$, and the operator 
$\sD_{u,\eps}:=\Nabla{s}+\eps^{-1}\dcD-\nabla\nabla H(u)$ 
is surjective.
\end{theorem}

\begin{remark}\label{rmk:morse=floer}\rm
Equation~\eqref{eq:floerHeps} is equivalent, via 
rescaling, to the equation
\begin{equation}\label{eq:floerHeps1}
\p_s\tu+\dd(\tu) = \eps\nabla H(\tu),\qquad
\lim_{s\to\pm\infty}\tu(s,y) = x^\pm,
\end{equation}
for the function
$
\tu(s,y):=u(\eps s,y).
$ 
Since the limit points $x^\pm$ are constant (as functions of $y$)
the energy is
$$
\frac{1}{\eps}\sE_{\eps H}(\tu) = \sE_H(u) 
= \int_{-\infty}^\infty\int_M\Abs{\p_su}^2\kappa
= \kappa\Vol(M)\Bigl(H(x^+)-H(x^-)\Bigr).
$$
The solutions of~\eqref{eq:floerHeps1}, and hence also those
of~\eqref{eq:floerHeps}, determine the boundary operator on
$\CF(M,X;\eps H)$. Moreover, $\sD_{u,\eps}$ is surjective 
if and only if the operator $\sD_{\tu,\eps H}$ 
in Proposition~\ref{prop:mu} is surjective.
\end{remark}

The proof of Theorem~\ref{thm:morse=floer} needs some preparations.

\bigbreak

\begin{lemma}\label{le:ddu}
Let $M$ be a compact hypercontact $3$-manifold and $X$
be a flat hyperk\"ahler manifold. If $H:M\times X\to\R$ is any 
smooth function and $u:\R\times M\to X$ is a
solution of~\eqref{eq:floerH} then 
\begin{equation}\label{eq:ddu}
\int_{s_0}^{s_1}\int_M
\left(\Abs{\Nabla{s}\p_su}^2 + \Abs{\dcD\p_su}^2
\right)\,\kappa
\le \left(C+\frac{4}{r^2}\right)
\int_{s_0-r}^{s_1+r}\int_M\Abs{\p_su}^2\kappa
\end{equation}
for all $s_0<s_1$ and $r>0$, where 
$$
C:=2\Norm{H}_{C^3}\Norm{\p_su}_{L^\infty}+2\Norm{H}_{C^2}^2.
$$
\end{lemma}

\begin{proof}
For $s\in\R$ define 
$$
\phi(s) := \frac12\int_M\Abs{\p_su}^2\,\kappa,\quad
\psi(s) := \frac12\int_M\left(
\Abs{\Nabla{s}\p_su}^2 + \Abs{\dcD\p_su}^2
\right)\,\kappa.
$$
Then
\begin{equation*}
\begin{split}
\phi''(s) 
&= 
\int_M\Abs{\Nabla{s}\p_su}^2\kappa
+ \int_M\inner{\Nabla{s}\Nabla{s}\p_su}{\p_su}\kappa \\
&=
\int_M\Abs{\Nabla{s}\p_su}^2\kappa
+ \int_M\inner{\Nabla{s}\Nabla{s}\nabla H(u)}{\p_su}\kappa 
- \int_M\inner{\Nabla{s}\p_su}{\dcD\p_su}\kappa.
\end{split}
\end{equation*}
Here we have used the fact that $\dcD$ commutes 
with $\Nabla{s}$, because $X$ is flat, and that 
$\dcD$ is self-adjoint with respect to the $L^2$ inner
product with weight~$\kappa$. Since 
$\Nabla{s}\p_su=\Nabla{s}\nabla H(u)-\dcD\p_su$
we obtain
$$
\phi''(s) = 2\psi(s) 
+ \int_M\inner{\Nabla{s}\Nabla{s}\nabla H(u)}{\p_su}\kappa 
- \int_M\inner{\Nabla{s}\nabla H(u)}{\dcD\p_su}\kappa.
$$
Using the inequalities 
$
\Abs{\Nabla{s}\Nabla{s}\nabla H(u)}
\le \Norm{H}_{C^3}\Abs{\p_su}^2
+ \Norm{H}_{C^2}\Abs{\Nabla{s}\p_su} 
$
and
$\Abs{\Nabla{s}\nabla H(u)}\le\Norm{H}_{C^2}\Abs{\p_su}$
we obtain
\begin{equation*}
\begin{split}
\phi''(s)
&\ge 
2\psi(s) 
- \int_M\Abs{\Nabla{s}\Nabla{s}\nabla H(u)}\Abs{\p_su}\kappa 
- \int_M\Abs{\Nabla{s}\nabla H(u)}\Abs{\dcD\p_su}\kappa \\
&\ge
2\psi(s) - \Norm{H}_{C^3}\int_M\Abs{\p_su}^3\kappa 
- \Norm{H}_{C^2}\int_M\Bigl(
\Abs{\Nabla{s}\p_su}+\Abs{\dcD\p_su}
\Bigr)\Abs{\p_su}\kappa \\
&\ge
\psi(s) - \left(\Norm{H}_{C^3}\Norm{\p_su}_{L^\infty}
+ \Norm{H}_{C^2}^2\right)\int_M\Abs{\p_su}^2\kappa \\
&=
\psi(s) - C\phi(s).
\end{split}
\end{equation*}
Now let $r,R>0$. Then, for $0\le s\le r$, we have
\begin{equation*}
\begin{split}
\int_{s_0}^{s_1}\psi-C\int_{s_0-r}^{s_1+r}\phi
&\le \int_{s_0-s}^{s_1+s}(\psi-C\phi)
\le \int_{s_0-s}^{s_1+s}\phi'' \\
&= \phi'(s_1+s)-\phi'(s_0-s)  \\
&= \frac{d}{ds}\Bigl(\phi(s_1+s)+\phi(s_0-s)\Bigr).
\end{split}
\end{equation*}
Integrating this inequality from $0$ to $t$ we obtain
$$
\frac{r}{2}\left(\int_{s_0}^{s_1}\psi-C\int_{s_0-r}^{s_1+r}\phi\right)
\le \phi(s_1+t)+\phi(s_0-t)
$$
for $\frac{r}{2}\le t\le r$. Integrating this inequality
again from $\frac{r}{2}$ to $r$ gives
$$
\int_{s_0}^{s_1}\psi
\le \left(C+\frac{4}{r^2}\right)\int_{s_0-r}^{s_1+r}\phi.
$$
This proves the lemma.
\end{proof}

\begin{lemma}\label{le:bounded-eps}
Let $M$, $X$, and $H$ be as in Theorem~\ref{thm:morse=floer}.
Then there are positive constants $\eps_0$ and $C$ 
such that every solution $u$ of~\eqref{eq:floerHeps}
with ${0<\eps\le\eps_0}$ satisfies 
$$
\sup_{\R\times M}\Abs{\p_su}\le C,\qquad
\sup_{\R\times M}\Abs{\p_{v_i}u}\le C\eps
$$ 
for $i=1,2,3$.
\end{lemma}

\begin{proof}
It is convenient to work with the solutions 
$$
\tu(s,y)=u(\eps s,y)
$$ 
of equation~\eqref{eq:floerHeps1}.
The function $s\mapsto\sA_{\eps H}(\tu(s,\cdot))$ is nonincreasing 
along $\tu$ and converges to $-\eps\kappa\Vol(M)H(x^-)$ 
as $s\to-\infty$.  Hence
\begin{equation}\label{eq:A}
\sA(\tu(s,\cdot))
= \sA_{\eps H}(\tu(s,\cdot)) + \int_M\eps H(\tu(s,\cdot))\kappa  
\le \eps\kappa\Vol(M)\Norm{H},
\end{equation}
where 
$$
\Norm{H}:=\max H-\min H.
$$ 
The energy of $u$ can be estimated by 
\begin{equation}\label{eq:E}
\sE_{\eps H}(\tu) 
= \int_{-\infty}^\infty\int_M\Abs{\p_s\tu}^2\,\kappa\,\dvol_M 
\le \eps\kappa\Vol(M)\Norm{H}.
\end{equation}
By equation~\eqref{eq:dudsu} in Lemma~\ref{le:bounded}, 
we have
$$
\frac12\int_M\Abs{d\tu}^2
\le\sA(\tu(s,\cdot)) + \eps^2\Vol(M)\sup_{\R\times M}\Abs{\nabla H}^2
+ \frac32\int_M\Abs{\p_s\tu}^2
$$
for every $s\in\R$.  Integrating this inequality from $s_0-1$ to $s_0+1$,
and using~\eqref{eq:A} and~\eqref{eq:E} we obtain
$$
\int_{s_0-1}^{s_0+1}\int_M\Abs{d\tu}^2\le c\eps,\quad 
c:=(3+4\kappa)\Vol(M)\Norm{H}+4\Vol(M)\sup_{\R\times M}\Abs{\nabla H}^2..
$$
Hence, by Lemma~\ref{le:Ler} and Theorem~\ref{thm:heinz}, 
there are positive constants $c'$ and $\eps_0$ such that
$\sup\Abs{d\tu}\le c'$ for every solution of~\eqref{eq:floerHeps1}
with $0<\eps\le\eps_0$. 

To improve this estimate we observe that the constant in 
Lemma~\ref{le:ddu} with $H$ replaced by $\eps H$ is 
$$
C=2\eps^2\Norm{H}_{C^2}^2 + 2\eps\Norm{H}_{C^3}\Norm{\p_s\tu}_{L^\infty}
\le c_1\eps,
$$
where $c_1$ depends only on $H$ and the bound on $\Abs{d\tu}$ 
established above.  Hence it follows from Lemma~\ref{le:ddu} 
with $r=\infty$ that
$$
\int_{-\infty}^\infty\int_M
\left(\Abs{\Nabla{s}\p_s\tu}^2 + \Abs{\dcD\p_s\tu}^2\right)\kappa
\le c_1\eps\sE_{\eps H}(\tu) \le c_2\eps^2
$$
Here we have used the fact that the energy of $d\tu$ is bounded by a 
constant times~$\eps$.  Since $\int_M\p_s\tu = \eps\int_M\nabla H(u)$ 
we obtain from~\eqref{eq:c0} with $\xi=\p_s\tu$ that
$$
\int_M\Abs{\p_s\tu}^2\le c_0\left(\int_M\Abs{\dcD\p_s\tu}^2
+ \Norm{H}_{C^1}^2\eps^2\right).
$$
Integrating this inequality from $s_0-1$ to $s_0+1$
gives
$$
\int_{s_0-1}^{s_0+1}\int_M\Abs{\p_s\tu}^2
\le c_0\int_{-\infty}^\infty\int_M\Abs{\dcD\p_s\tu}^2
+ 2c_0\Norm{H}_{C^1}^2\eps^2
\le c_3\eps^2.
$$
Now it follows from Lemma~\ref{le:Les} and Theorem~\ref{thm:heinz} 
that every solution of~\eqref{eq:floerHeps1} with $0<\eps\le\eps_0$
satisfies the pointwise inequality $\Abs{\p_s\tu}\le c_4\eps$
for a suitable constant $c_4>0$.  Using the equation we obtain
$\Abs{\dd(\tu)}\le c_5\eps$. Using again the fact that $\dd=\dcD$ 
(on functions with values in $\H^n$) is an elliptic operator 
whose kernel consists of the constant functions we obtain 
$\int_M\Abs{d\tu}^2\le c_6\eps^2$ for every~$s$.
Integrating this inequality from $s_0-1$ to $s_0+1$, and using
Lemma~\ref{le:Ler} and Theorem~\ref{thm:heinz}, we conclude that
every solution of~\eqref{eq:floerHeps1} with $0<\eps\le\eps_0$
satisfies the pointwise inequality $\Abs{d\tu}^2\le c_7\eps^2$ 
for a suitable constant $c_7>0$. This proves the lemma.
\end{proof}

\begin{lemma}\label{le:criteps}
Let $M$, $X$, and $H$ be as in Theorem~\ref{thm:morse=floer}.
Then there are positive constants $\eps_0$, $\delta$, and $c$ 
such that the following holds.  If $f:M\to X$ is a smooth function 
such that 
$$
\sup_M\Abs{\eps^{-1}\dd(f)-\nabla H(f)}<\delta
$$
then 
\begin{equation}\label{eq:hesseps}
\int_M\Abs{\xi}^2 \le c\int_M\Abs{\eps^{-1}\dcD\xi-\Nabla{\xi}\nabla H(f)}^2
\end{equation}
for every $\xi\in\Om^0(M,f^*TX)$.  
\end{lemma}

\begin{proof}
Suppose, by contradiction, that there are sequences $\eps_\nu\to0$
and $f_\nu:M\to X$ such that the sequence
$$
\eta_\nu := \eps_\nu^{-1}\dd(f_\nu)-\nabla H(f_\nu)
$$
converges uniformly to zero and~\eqref{eq:hesseps} does not hold for $f_\nu$.
It is convenient to choose lifts of the maps with values in the universal
cover $\H^n$ of $X$.  These lifts will still be denoted by $f_\nu:M\to\H^n$
and we introduce the sequence of mean values
$$
\bar f_\nu := \frac{1}{\Vol(M)}\int_Mf_\nu.
$$
Assume without loss of generality that the sequence 
$\bar f_\nu\in\H^n$ is bounded and hence, passing to 
a subsequence if necessary, that it converges.  
By elliptic regularity for the operator $\dcD$ 
whose kernel consists of the constant functions 
(Lemma~\ref{le:apriori}), there is a constant $c_0>0$ such that 
\begin{equation}\label{eq:c0}
\int_M\xi=0\quad\implies\quad
\int_M\Abs{\xi}^2\le c_0\int_M\Abs{\dcD\xi}^2,\quad
\sup_M\Abs{\xi}\le c_0\sup_M\Abs{\dcD\xi}
\end{equation}
for every smooth map $\xi:M\to\H^n$. To prove the second inequality
in~\eqref{eq:c0} one can use the Sobolev estimate
$\Norm{\xi}_{L^\infty}\le c\Norm{\xi}_{W^{1,p}}$ for $p>3$
and then $L^p$ regularity for $\dcD$. Applying this inequality to
the sequence $f_\nu-\bar f_\nu$ we obtain 
$$
\sup_M\Abs{f_\nu-\bar f_\nu} 
\le c_0\sup_M\Abs{\dd(f_\nu)}
= c_0\eps_\nu\sup_M\Abs{\nabla H(f_\nu)+\eta_\nu}
\to 0
$$
and so $f_\nu$ converges uniformly to the same limit as $\bar f_\nu$.
Since 
$$
\lim_{\nu\to\infty}\nabla H(\bar f_\nu)
= \lim_{\nu\to\infty}\frac{1}{\Vol(M)}\int_M\nabla H(f_\nu) 
= \lim_{\nu\to\infty}\frac{1}{\Vol(M)}\int_M\eta_\nu=0,
$$
this limit is a critical point of $H$.   Hence there is a constant $c_1>0$
such that, for $\nu$ sufficiently large and $\bar\xi\in\H^n$, we have
\begin{equation}\label{eq:c1}
\Abs{\bar\xi}\le c_1\Abs{\Nabla{\bar\xi}\nabla H(\bar f_\nu)}.
\end{equation}
Now let $\xi:M\to\H^n$ be a smooth map (thought of as a vector field
along~$f_\nu$) and denote
$$
\bar\xi := \frac{1}{\Vol(M)}\int_M\xi.
$$
Then $\dcD\xi$ has mean value zero and hence is $L^2$
orthogonal to $\Nabla{\bar\xi}\nabla H(\bar f_\nu)$. 
This implies
\begin{equation*}
\begin{split}
&\eps_\nu^{-2}\Norm{\dcD\xi}^2 + \Abs{\Nabla{\bar\xi}\nabla H(\bar f_\nu)}^2 
= \Norm{\eps_\nu^{-1}\dcD\xi - \Nabla{\bar\xi}\nabla H(\bar f_\nu)}^2 \\
&\le 
3 \Norm{\eps_\nu^{-1}\dcD\xi - \Nabla{\xi}\nabla H(f_\nu)}^2 
+ 3\Norm{\Nabla{\xi-\bar\xi}\nabla H(f_\nu)}^2 \\
&\quad
+ 3\Norm{\Nabla{\bar\xi}\nabla H(f_\nu)
- \Nabla{\bar\xi}\nabla H(\bar f_\nu)}^2 \\
&\le 
3 \Norm{\eps_\nu^{-1}\dcD\xi - \Nabla{\xi}\nabla H(f_\nu)}^2 
+ c\Norm{\xi-\bar\xi}^2 
+ c\Abs{\bar\xi}^2\Norm{f_\nu-\bar f_\nu}^2 \\
&\le 
3 \Norm{\eps_\nu^{-1}\dcD\xi - \Nabla{\xi}\nabla H(f_\nu)}^2 
+ cc_0\Norm{\dcD\xi}^2 
+ cc_1\Abs{\Nabla{\bar\xi}\nabla H(\bar f_\nu)}^2\Norm{f_\nu-\bar f_\nu}^2 \\
\end{split}
\end{equation*}
Here all norms are $L^2$ norms on $M$, the constant $c$ depends only on $H$,
and the last inequality follows from~\eqref{eq:c0} and~\eqref{eq:c1}. 
For $\nu$ sufficiently large the last two terms one the right are together 
at most one quarter of the terms on the left. For these values of $\nu$ we have
$$
\eps_\nu^{-2}\Norm{\dcD\xi}^2 + \Norm{\Nabla{\bar\xi}\nabla H(\bar f_\nu)}^2 
\le 4\Norm{\eps_\nu^{-1}\dcD\xi - \Nabla{\xi}\nabla H(f_\nu)}^2.
$$
Hence if follows from~\eqref{eq:c0} and~\eqref{eq:c1} that $f_\nu$ 
satisfies~\eqref{eq:hesseps} for $\nu$ sufficiently large, in 
contradiction to our assumption.  This proves the lemma. 
\end{proof}

\begin{lemma}\label{le:decay-eps}
Let $M$, $X$, and $H$ be as in Theorem~\ref{thm:morse=floer}.
Then there are positive constants $\eps_0$, $\delta$, 
$\rho$, and $c$ such that the following holds. 
If $T>0$ and $u:\R\times M\to X$ is a solution 
of~\eqref{eq:floerHeps} with $0<\eps\le\eps_0$ 
such that
$$
\int_{-T}^T\int_M\Abs{\p_su}^2 
<\delta
$$
then 
$$
\sup_{y\in M}\Abs{\p_su(s,y)}^2 \le 
ce^{-\rho(T-\Abs{s})}\int_{-T}^T\int_M\Abs{\p_su}^2
$$
for $\Abs{s}\le T-2$.
\end{lemma}

\begin{proof}
The functions 
\begin{equation*}
\begin{split}
\phi(s)&:=\frac12\int_M\Abs{\p_su}^2,\\
\psi(s)&:= \int_M\Abs{\Nabla{s}\p_su}^2
+ \int_M\Abs{\eps^{-1}\dcD\p_su-\Nabla{\p_su}\nabla H(u)}^2
\end{split}
\end{equation*}
satisfy
\begin{equation*}
\begin{split}
\phi''(s) 
&= \psi(s) + \int_M\inner{\nabla^2\nabla H(\p_su,\p_su)}{\p_su} \\
&\ge \psi(s) - 2\Norm{H}_{C^3}\Norm{\p_su}_{L^\infty(M)}\phi(s).
\end{split}
\end{equation*}
Hence, by Lemma~\ref{le:bounded-eps}, there is 
a constant $B>0$ such that 
$
\phi''\ge -B\phi.
$
Now apply Theorem~\ref{thm:heinz} to the function
$\phi$ to obtain 
$$
\Abs{s}\le T-1\qquad\implies\qquad
\phi(s) \le c_1\int_{s-1}^{s+1}\phi \le c_1\int_{-T}^T\phi\le c_1\delta.
$$
Careful inspection of the proof of Theorem~\ref{thm:heinz} 
for $n=1$ shows that the constant can be chosen as $c_1=8(\sqrt{B}+1)$. 
This shows that the rescaled function $\tu(s,y):=u(\eps s,y)$ 
satisfies the inequality
$$
\Abs{s}\le \eps^{-1}(T-1)\qquad\implies\qquad
\int_M\Abs{\p_s\tu}^2 \le c_1\eps^2\delta.
$$
Now integrate this inequality from $s_0-1$ to $s_0+1$.
Using Lemma~\ref{le:Les} (together with the uniform $C^1$ 
bound of Lemma~\ref{le:bounded-eps}) and Theorem~\ref{thm:heinz}
we then obtain the pointwise inequality 
$\Abs{\p_s\tu(s,y)}^2\le c_2\eps^2\delta$
for ${\Abs{s}\le \eps^{-1}(T-1)-1}$.  For the function $u$ this gives 
$$
\Abs{s}\le T-2\quad\implies\quad 
\sup_M\Abs{\eps^{-1}\dcD\p_su-\nabla H(u)}^2
= \sup_M\Abs{\p_su}^2\le c_2\delta.
$$
If $\delta$ is chosen sufficiently small we obtain from 
Lemma~\ref{le:criteps} with $\xi=\p_su$ that 
$$
\int_M\Abs{\p_su}^2 \le 
c_3\int_M\Abs{\eps^{-1}\dcD\p_su-\Nabla{\p_su}\nabla H(u)}^2
$$
for $\Abs{s}\le T-2$. Thus $\phi(s)\le c_3\psi(s)$ 
and, putting things together, we have
\begin{equation*}
\begin{split}
\phi''(s) 
&\ge \psi(s) - 2\Norm{H}_{C^3}\Norm{\p_su}_{L^\infty(M)}\phi(s) \\
&\ge \left(\frac{1}{c_3}-2\Norm{H}_{C^3}\sqrt{c_2\delta}\right)\phi(s)
\end{split}
\end{equation*}
for $\Abs{s}\le T-2$.  With $\delta$ sufficiently small this gives
$\phi''(s)\ge\rho^2\phi(s)$ and hence the function 
$s\mapsto e^{-\rho s}(\phi'(s)+\rho\phi(s))$ is nondecreasing.  
If $\phi'(s_0)\ge0$ we then obtain 
$
e^{-\rho s_0}\rho\phi(s_0)\le e^{-\rho s_0}(\phi'(s_0)+\rho\phi(s_0))
\le e^{-\rho s}(\phi'(s)+\rho\phi(s))
$
for $s_0\le s\le T-2$.  Thus 
$
\rho e^{\rho(s-s_0)}\phi(s_0)
\le \phi'(s)+\rho\phi(s)
$ 
and integrating this inequality gives 
$$
e^{\rho(T-s_0-2)}\phi(s_0)\le 
\phi(T-2) + \rho\int_{s_0}^{T-2}\phi \le (c_1+\rho)\int_{-T}^{T}\phi.
$$
If $\phi'(s_0)\le 0$ we obtain a similar inequality by reversing time.
Thus we have proved that 
$
e^{\rho(T-\Abs{s})}\phi(s)\le c_4\int_{-T}^{T}\phi
$ 
for $\Abs{s}\le T-2$, where $c_4:=e^{2\rho}(c_1+\rho)$.   
The pointwise estimate for $\Abs{\p_su}^2$ follows by 
the same argument as above from 
Lemma~\ref{le:Les} and Theorem~\ref{thm:heinz} via rescaling.  
This proves the lemma. 
\end{proof}

\begin{proof}[Proof of Theorem~\ref{thm:morse=floer}]
The proof has four steps.  It is modelled on the 
adiabatic limit argument in~\cite{DS}.

\medskip\noindent{\bf Step~1.}
{\it There exists a constant $\eps_0>0$ with the 
following significance. If~${0<\eps\le\eps_0}$, $x^\pm$ 
are critical points of $H$ with 
$\mathrm{ind}_H(x^+)-\mathrm{ind}_H(x^-)=1$, and $u_0:\R\to X$ 
is a gradient trajectory from $x^-$ to $x^+$
for $\eps H$, i.e.\
\begin{equation}\label{eq:morse}
\frac{d}{ds} u_0(s) = \nabla H(u_0(s)),\qquad
\lim_{s\to\pm\infty}u_0(s) = x^\pm,
\end{equation}
then the function $\R\times M\to X:(s,y)\mapsto u_0(s)$ 
is a regular solution of~\eqref{eq:floerHeps},
i.e.\ the operator $\sD_{u_0,\eps}$ is surjective.}

\medskip\noindent
Let $\xi\in W^{1,p}(\R\times M,u_0^*TX)$ and define 
$\bar\xi\in W^{1,p}(\R,u_0^*TX)$ by
$$
\bar\xi(s) := \frac{1}{\Vol(M)}\int_M\xi(s,y)\,\dvol_M(y)
$$
for $s\in\R$. Then
$$
\sD_{u_0,\eps}\xi = \p_s\bar\xi + \Nabla{\bar\xi}\nabla H(u_0)
+ \sD_{u_0,\eps}(\xi-\bar\xi).
$$
Denote by $W^{1,p}_0(\R\times M,u_0^*TX)\subset W^{1,p}(\R\times M,u_0^*TX)$
the subspace of all functions $\xi$ such that $\xi(s,\cdot)$ has mean value
zero on $M$ for every $s$ and simliarly for 
$L^p_0(\R\times M,u_0^*TX)\subset L^p(\R\times M,u_0^*TX)$.
Then the operator 
$$
\sD_{u_0,\eps}:W^{1,p}_0(\R\times M,u_0^*TX)
\to L^p_0(\R\times M,u_0^*TX)
$$
is equivalent to the operator 
$$
\sD_{\tu_0\eps H}= \Nabla{s} + \dcD - \eps\nabla\nabla H(\tu_0)
$$
associated to the rescaled function $\tu_0(s):=u_0(\eps s)$. 
This operator is bijective for $\eps=0$ and hence also for 
$\eps>0$ sufficiently small.  Hence Step~1 follows from the 
above decomposition of the operator $\sD_{u_0,\eps}$ 
(and the fact that there are only finitely
many index one gradient trajectories up to time shift).

\medskip\noindent{\bf Step~2.}
{\it There is a constants $\eps_0>0$ with the 
following significance.  
If $x^\pm$ are critical points of $H$ such that 
$\mathrm{ind}_H(x^+)-\mathrm{ind}_H(x^-)=1$, and $u:\R\times M\to X$ 
and $u_0:\R\to X$ are solutions of~\eqref{eq:floerHeps} 
and~\eqref{eq:morse}, respectively, such that
$$
0<\eps\le\eps_0,\qquad 
\sup_{\R\times M} d(u,u_0)<\delta
$$
then there is an $s_0\in\R$ such that 
$u(s+s_0,y)=u_0(s)$ for all $s$ and $y$.} 

\medskip\noindent
We wish to find a real number $s_0$ close to zero such that
\begin{equation}\label{eq:perp}
\frac{1}{\Vol(M)}\int_Mu(s_0,\cdot) - u_0(0) \perp\nabla H(u_0(0)).
\end{equation}
To prove that $s_0$ exists we consider the function 
$$
\phi(s) := \frac{1}{\Vol(M)}\int_M
\inner{u(s,\cdot)-u_0(0)}{\nabla H(u_0(0))}\,\dvol_M.
$$
It satisfies 
$$
\Abs{\phi(0)} \le \delta \mu,\qquad
\mu:=\Abs{\nabla H(u_0(0))}>0.
$$
Choose a constant $\rho>0$ so small that 
$$
\Abs{x-u_0(0)}\le\rho\qquad\implies\qquad
\inner{\nabla H(x)}{\nabla H(u_0(0))}>\frac{\mu^2}{2}.
$$ 
Let $C$ be the constant of Lemma~\ref{le:bounded-eps} so that 
$\sup_{\R\times M}\Abs{\p_su}\le C$. Then we have
$
\Abs{u(s,y)-u_0(0)}\le \Abs{u(s,y)-u(0,y)}+\delta
\le C\Abs{s}+\delta,
$
and hence
$$
C\Abs{s}+\delta\le\rho\qquad\implies\qquad
\Abs{u(s,y)-u_0(0)}\le\rho.
$$
Combining the last two inequalities we have,
for $C\Abs{s}+\delta<\rho$, that
\begin{equation*}
\begin{split}
\dot\phi(s) 
&= \frac{1}{\Vol(M)}\int_M
\inner{\p_su(s,\cdot)}{\nabla H(u_0(0))} \\
&= \frac{1}{\Vol(M)}\int_M
\inner{\nabla H(u(s,\cdot))}{\nabla H(u_0(0))} \\
&\ge \frac{\mu^2}{2}.
\end{split}
\end{equation*}
To obtain a zero of $\phi$ we need this inequality on an 
interval of length $T$ (on either side of zero)
where $\frac{1}{2}\mu^2T\ge\delta\mu$, or equivalently 
$T\ge \frac{2\delta}{\mu}$. On the other hand, the interval 
at our disposal has length at most $(\rho-\delta)/C$.  
Thus we must impose the condition
${(\rho-\delta)/C>2\delta/\mu}$, 
or equivalently 
$$
\delta\left(1+\frac{2C}{\mu}\right) < \rho.
$$
Under this assumption there is a real number $s_0$ with
$\Abs{s_0}\le2\delta/\mu$ such that~\eqref{eq:perp} holds.  
We can still control the distance of $u(s+s_0,y)$ and $u_0(s)$ 
by a fixed multiple of $\delta$.  We assume from now 
on that  $\Abs{u(s+s_0,y)-u_0(s)}\le c\delta$ for all $s$ and $y$ 
and that~\eqref{eq:perp} holds. 

Consider the functions
\begin{equation*}
\begin{split}
\xi(s,y) &:= u(s+s_0,y)-u_0(s),\\
\eta(s,y)
&:= \nabla H(u(s+s_0,y))-\nabla H(u_0(s))
-\Nabla{\xi(s,y)}\nabla H(u_0(s)).
\end{split}
\end{equation*}
Then $\Abs{\eta(s,y)}\le\Norm{H}_{C^3}\Abs{\xi(s,y)}^2$ and
\begin{equation}\label{eq:xieta}
\p_s\xi + \eps^{-1}\dcD\xi-\Nabla{\xi}\nabla H(u_0) = \eta.
\end{equation}
Hence the functions 
\begin{equation*}
\bar\xi(s):= \frac{1}{\Vol(M)}\int_M\xi(s,\cdot),\qquad
\bar\eta(s):= \frac{1}{\Vol(M)}\int_M\eta(s,\cdot)
\end{equation*}
satisfy
$$
\p_s\bar\xi - \Nabla{\bar\xi}\nabla H(u_0)
= \bar\eta,\qquad \inner{\bar\xi(0)}{\nabla H(u_0(0))} = 0.
$$
Since the gradient flow of $H$ is Morse--Smale 
and $u_0$ is an index-$1$ gradient trajectory of $H$ 
the kernel of the operator 
$$
D_{u_0}:=\p_s+\nabla\nabla H(u_0):W^{1,p}(\R,\H^n)\to L^p(\R,\H^n)
$$
is $1$-dimensional and is spanned by $\p_s\tu_0$.
Since 
$
\p_su_0(0)=\nabla H(u_0(0))
$ 
the restriction of $D_{u_0}$ to the codimention-$1$ 
subspace of all $\zeta\in W^{1,p}(\R,\H^n)$ that satisfy 
$\inner{\zeta(0)}{\nabla H(u_0(0))} = 0$ 
is a Banach space isomorphism.  This implies that there 
is a constant $c_0>0$, depending only on $u_0$, such that 
$$
\inner{\zeta(0)}{\nabla H(u_0(0))} = 0
\qquad\implies\qquad
\norm{\zeta}_{W^{1,p}}\le 
c_0\norm{\p_s\zeta -\Nabla{\zeta}\nabla H(u_0)}_{L^p}.
$$
Applying this to the elements $\zeta=\bar\xi$ 
we have $D_{u_0}\bar\xi=\bar\eta$ and hence
\begin{equation*}
\norm{\bar\xi}_{L^p}
\le c_0\norm{\bar\eta}_{L^p}
\le \frac{c_0}{\Vol(M)^{1/p}}\Norm{\eta}_{L^p}
\le \frac{c_0c\Norm{H}_{C^3}}{\Vol(M)^{1/p}}\delta\Norm{\xi}_{L^p}.
\end{equation*}
Here we have used the inequality 
$\Abs{\eta}\le \Norm{H}_{C^3}\Abs{\xi}^2 \le c\Norm{H}_{C^3}\delta\Abs{\xi}$.
Now it follows from~\eqref{eq:xieta} and the discussion in
the proof of Step~1 for the rescaled operator $\sD_{\tu_0,\eps H}$
that, for a suitable constant (still denoted by $c_0$) 
and $\eps>0$ sufficiently small, we have
\begin{eqnarray*}
\norm{\xi-\bar\xi}_{L^p}
&\le& 
c_0\eps \norm{\eta-\bar\eta}_{L^p} \\
&\le& 
c_0\eps\left(1+\frac{1}{\Vol(M)^{1/p}}\right) \norm{\eta}_{L^p} \\
&\le& 
\frac{c_0c\Norm{H}_{C^3}}{\Vol(M)^{1/p}}\left(\eps+\eps\Vol(M)^{1/p}\right) 
\delta\norm{\xi}_{L^p}.
\end{eqnarray*}
If $\delta(1+\eps+\eps\Vol(M)^{1/p})<\Vol(M)^{1/p}/c_0c\Norm{H}_{C^3}$
then $\xi$ must vanish and this proves Step~2. 

\smallbreak

\medskip\noindent{\bf Step~3.}
{\it There are positive constant $\eps_0$ and $c$
such that the following holds. If $x^\pm$ are
critical points of $H$ and $u:\R\times M\to\H^n$
is a lift of a solution of~\eqref{eq:floerHeps}
(with $0<\eps\le\eps_0$) to the universal cover 
$\H^n$ of $X$, then the function
$$
\bar u(s) := \frac{1}{\Vol(M)}\int_M
u(s,\cdot)\,\dvol_M
$$
satisfies the inequalities
$$
\int_{-\infty}^\infty \Abs{\p_s\bar u(s)}^2\,ds 
\le H(x^+) - H(x^-) + c\eps^2,\qquad
\int_{-\infty}^\infty 
\Abs{\Nabla{s}\p_s\bar u(s)}^2\,ds \le c,
$$
and $\Abs{\p_s\bar u(s) - \nabla H(\bar u(s))}\le c\eps$
for every $s\in\R$.}

\medskip\noindent
First note that
$$
\p_s\bar u(s) - \eps\nabla H(\bar u(s))
= \frac{1}{\Vol(M)}\int_M\left(
\nabla H(u(s,\cdot))-\nabla H(\bar u(s))
\right)\,\dvol_M
$$
and hence
\begin{equation*}
\begin{split}
\Abs{\p_s\bar u(s) - \nabla H(\bar u(s))}^2
&\le
\frac{\Norm{H}_{C^2}}{\Vol(M)}
\int_M\Abs{u(s,\cdot)-\bar u(s)}^2\,\dvol_M \\
&\le
\frac{c_1\Norm{H}_{C^2}\eps^2}{\Vol(M)}
\int_M\Abs{\dd(u)}^2\,\dvol_M \\
&\le 
c_2\eps^4.
\end{split}
\end{equation*}
Here the second inequality follows from~\eqref{eq:c0} and the
last from Lemma~\ref{le:bounded-eps}.  
Second, the function $\bar u$ satisfies 
\begin{equation*}
\int_{-\infty}^\infty
\Abs{\Nabla{s}\p_s\bar u(s)}^2 \,ds
\le
\frac{1}{\Vol(M)}
\int_{-\infty}^\infty\int_M
\Abs{\Nabla{s}\p_su}^2\,\dvol_M\,ds 
\le 
c_3
\end{equation*}
Here we have used Lemma~\ref{le:bounded-eps} and 
Lemma~\ref{le:ddu} for the rescaled function 
$\tu(s,y):=u(\eps s,y)$ with $C$ equal to a constant times
$\eps^2$. Third, we have
\begin{equation*}
\begin{split}
\int_{-\infty}^\infty
\Abs{\p_s\bar u(s)}^2 \,ds
&=
\frac{1}{\Vol(M)}\int_{-\infty}^\infty\int_M
\Bigl(\Abs{\p_su}^2 - \Abs{\p_su-\p_s\bar u}^2\Bigr)  \\ 
&\le
H(x^+)-H(x^-)
+ \frac{c_0}{\Vol(M)}\int_{-\infty}^\infty\int_M
\Abs{\dcD\p_su}^2   \\
&\le H(x^+) - H(x^-) + c_4\eps^2.
\end{split}
\end{equation*}
Here we have used~\eqref{eq:c0} and Lemma~\ref{le:ddu},
again for the rescaled function $\tu(s,y):=u(\eps s,y)$.
This proves Step~3.

\medskip\noindent{\bf Step~4.}
{\it We prove the theorem.}

\medskip\noindent
Let $x^\pm$ be a pair of critical points of $H$ of index difference
less than or equal to~$1$. Suppose, by contradiction, that there 
is a sequence of solutions $u_\nu:\R\times M\to X$
of~\eqref{eq:floerHeps} associated to a sequence $\eps_\nu\to0$ 
such that $u_\nu(s,y)$ is not independent of $y$. 
Replace each $u_\nu$ by a lift to the universal 
cover $\H^n$ of $X$ (still denoted by $u_\nu$)
with the same limit point $\lim_{s\to-\infty}u_\nu(s,y)$. 

First it follows from Lemma~\ref{le:decay-eps} that the functions 
$s\mapsto \p_su_\nu(s,y)$ satisfy a uniform $L^1$ bound.  
Namely, if $\delta$ is the constant of 
Lemma~\ref{le:decay-eps} and ${N>\Vol(M)(H(x^+)-H(x^-))/\delta}$ 
is an integer then, for each $\nu$, the real axis can 
be divided into $N$ intervals such that the energy of $u_\nu$ 
on each of these intervals is less than $\kappa\delta$ and hence, 
by Lemma~\ref{le:decay-eps}, $\p_su_\nu$ satisfies uniform
exponential estimates on all these intervals.
This shows that the images of the functions $u_\nu$ 
are contained in a fixed compact subset of $\H^n$. 

Now consider the associated 
functions
$$
\bar u_\nu(s) := \frac{1}{\Vol(M)}\int_Mu(s,\cdot)\,\dvol_M.
$$
Normalize the sequence such that
$H(\bar u_\nu(0)) = 2^{-1}(H(x^+)+H(x^-))$.
The $W^{2,2}$-bound of Step~3 guarantees the existence 
of a subsequence (still denoted by~$\bar u_\nu$) that converges 
in the $C^1$-norm on every compact subset of $\R$ to a 
gradient trajectory $\bar u_\infty$ of $H$.  The energy bound
of Step~3 shows that the limit sequence has energy at 
most $H(x^+)-H(x^-)$. We claim that $\bar u_\infty$ connects 
$x^-$ to $x^+$.  Otherwise, the standard compactness argument 
would give a subsequence converging to a catenation of at least 
two gradient trajectories running from $x^-$ to $x^+$,
contradicting the Morse--Smale property of the gradient flow.
Now it follows from Step~3 that
$$
\int_{-\infty}^\infty\Abs{\p_s\bar u_\infty}^2
= H(x^+)-H(x^-) = \lim_{\nu\to\infty}
\int_{-\infty}^\infty\Abs{\p_s\bar u_\nu}^2.
$$
This implies that $\bar u_\nu(s_\nu)$ must converge 
to $x^\pm$ for every sequence $s_\nu\to\pm\infty$.
Hence $\bar u_\nu$ converges uniformly to $\bar u_\infty$ 
on all of $\R$. Now it follows from the Sobolev
inequality and the elliptic estimate for the operator
$\dcD$  that
\begin{equation*}
\Norm{u_\nu(s,\cdot)-\bar u_\nu(s)}_{L^\infty(M)}
\le c_1\Norm{\dcD u_\nu(s,\cdot)}_{L^p(M)} 
\le c_2\eps_\nu
\end{equation*}
for $p>3$.  Here the last inequality follows from 
Lemma~\ref{le:bounded-eps}. Hence 
$$
\lim_{\nu\to\infty}\sup_{s,y}
\Abs{u_\nu(s,y)-\bar u_\infty(\eps_\nu s)} = 0.
$$
By Step~2 this implies that, for $\nu$ sufficiently large, 
$u_\nu(s,y)$ agrees with $\bar u_\infty(s)$
up to a time shift and hence is independent of $y$.
This contradicts our assumption and proves the theorem.
\end{proof}

\begin{proof}[Proof of Theorem~\ref{thm:HF=H}]
Let $H:X\to\R$ be as in Theorem~\ref{thm:morse=floer}. 
Then, if $\eps>0$ is sufficiently small, each Floer trajectory 
for $\eps H$ of index $1$ is a Morse gradient line 
and there are no nontrivial Floer trajectories
with index less than $1$. Thus, for $H'\in\sH^\reg$ 
sufficiently $C^2$ close to $\eps H$, the Floer chain complex 
$(\CF(M,X;H'),\p^{H'})$ coincides with the Morse complex of $\eps H$.
Hence the Floer homology group $\HF(M,X;H')$ is naturally isomorphic 
to the Morse homology of $(X,\eps H)$.  This gives rise to an isomorphism
$H_*(X;\Z_2)\to\HF(M,X;H')$ and composition with the
isomorphisms $\HF(M,X;H')\to\HF(M,X;H^\alpha)$ of Theorem~\ref{thm:iso} 
gives a family of isomorphisms satisfying the requirements 
of Theorem~\ref{thm:HF=H}. 
\end{proof}

\begin{proof}[Proof of Theorem~\ref{thm:CZ}]
Assume $X$ is a compact flat hyperk\"ahler manifold
and let $H\in\sH^\morse$. Then, by Theorem~\ref{thm:crit-cpct}, 
the number of critical points of $\sA_H$ remains unchanged under 
any perturbation of $H$ that is sufficiently small
in the $C^2$ norm. Hence, by Theorem~\ref{thm:trans},
we may assume without loss of generality that $H\in\sH^\reg$.
By Theorem~\ref{thm:HF=H}, we then have
$$
\#\sC(H) = \dim \CF_*(M,X;H)
\ge \dim\HF_*(M,X;H) = \dim H_*(X;\Z_2).
$$
This proves the theorem.
\end{proof}

\begin{remark}\label{rmk:S3H4}\rm
An alternative proof of Theorem~\ref{thm:HF=H} can be given along the 
lines of~\cite{PSS}, avoiding the adiabatic limit argument 
of Theorem~\ref{thm:morse=floer}.  This would involve 
Morse--Bott exponential decay for finite energy solutions 
of~\eqref{eq:floerH} with $H=0$ on a half cylinder $[0,\infty)\times M$
respectively $(-\infty,0]\times M$.  Since $X$ is flat, such solutions 
converge to a point in $X$ as $s\to\pm\infty$, and one can then study 
solutions where this limit point lies on a gradient trajectory
of a Morse function on $X$, as in~\cite{PSS}, to obtain the 
desired isomorphism from Morse to Floer homology, respectively its inverse.

If $M:=S^3$ with the standard hypercontact structure, 
the Morse--Bott exponential decay as $s\to+\infty$ 
can be reduced to the removable singularity 
theorem~\ref{thm:remsing}: If $u:\R\times S^3\to X$ 
is a solution of~\eqref{eq:floerH} with $H=0$
and $w:\H\setminus\{0\}\to X$ is given by $w(e^{-s}y) := u(s,y)$
then
\begin{equation}\label{eq:w}
\p_0w-I\p_1w-J\p_2w-K\p_3w = 0.
\end{equation}
Moreover, the energy of $w$ on a ball of radius $r=e^{-s_0}$ is given by
$$
r^2\int_{B_r}\Abs{dw}^2 
= \sA(u(s_0,\cdot)) 
= 2\int_{s_0}^\infty\int_{S^3}\Abs{\p_su}^2.
$$
(Here we use $\kappa=2$ for $M=S^3$.)  We emphasize that no such 
argument is available for the limit $s\to-\infty$.  This reflects
a fundamental asymmetry in equation~\eqref{eq:floerH} related
to the noncommutativity of the quaternions.
\end{remark}

 
\appendix 

\section{Hypercontact manifolds} \label{app:hyco}

Let $M$ be an oriented $3$-manifold.  
Three contact structures $\xi_1,\xi_2,\xi_3$ on $M$
are said to form a {\bf hypercontact structure} if there exists a 
$1$-form ${\alpha=(\alpha_1,\alpha_2,\alpha_3)\in\Om^1(M,\R^3)}$
such that $\alpha_i\wedge d\alpha_i>0$, $\xi_i=\ker\alpha_i$, and
\begin{equation}\label{eq:hc}
\alpha_i\wedge d\alpha_i=\alpha_j\wedge d\alpha_j=:\sigma,\qquad
\alpha_i\wedge d\alpha_j+\alpha_j\wedge d\alpha_i=0
\end{equation}
for $i\ne j$.  The $1$-form $\alpha$ is determined by the contact 
structures $\xi_i$ up to multiplication by a positive function 
on $M$.  We shall sometimes abuse notation and refer to the $1$-form
$\alpha\in\Om^1(M,\R^3)$ as the hypercontact structure.
Associated to $\alpha$ is a family of contact forms
$$
\alpha_\lambda := \inner{\lambda}{\alpha} = 
\lambda_1\alpha_1+\lambda_2\alpha_2+ \lambda_3\alpha_3
$$
parametrized by the standard $2$-sphere $S^2\subset\R^3$.
In this formulation equations~\eqref{eq:hc} hold if and only the volume 
form $\alpha_\lambda\wedge d\alpha_\lambda$ is independent of $\lambda$.
Hypercontact structures were introduced and studied 
by Geiges--Gonzalo~\cite{GG,GG1}.  They used the term 
{\it taut contact sphere}  for the map 
$\lambda\mapsto\alpha_\lambda$.  
The term {\it hypercontact structure}
was used with a different meaning in~\cite{GTH}.

\begin{lemma}\label{le:reeb}
Let $\alpha\in\Om^1(M,\R^3)$ be a hypercontact structure.
Then the associated Reeb vector fields $v_1,v_2,v_3\in\Vect(M)$ 
are everywhere linearly independent.
\end{lemma}

\begin{proof}
Since $\alpha_\lambda\wedge d\alpha_\lambda=\Abs{\lambda}^2\sigma$
for $\lambda\in\R^3$ the $2$-forms $d\alpha_1,d\alpha_2,d\alpha_3$
are everywhere linearly independent. Since 
$d\alpha_i=\iota(v_i)\sigma$ this shows that $v_1,v_2,v_3$ 
are everywhere linearly independent.
\end{proof}

\begin{remark}\label{rmk:metric}\rm
If the $1$-forms $\alpha_1,\alpha_2,\alpha_3$ form a hypercontact
structure then, by Lemma~\ref{le:reeb}, the Reeb vector fields 
$v_1,v_2,v_3$ form a global framing of the tangent bundle. 
Call the hypercontact structure {\bf positive} if this
framing is compatible with the orientation. This can be 
achieved by reversing the sign of all three $1$-forms, 
if necessary. In the positive case the function
\begin{equation}\label{eq:kappa}
\kappa := d\alpha_1(v_2,v_3)
= d\alpha_2(v_3,v_1) = d\alpha_3(v_1,v_2)
\end{equation}
on $M$ is positive. Moreover, it is convenient to choose a Riemannian
metric on $M$ in which the $v_i$ form an orthonormal basis.  
The associated volume form is given by
$$
\dvol_M=\frac{\alpha_i\wedge d\alpha_i}{\kappa},\qquad i=1,2,3.
$$ 
\end{remark}

\begin{remark}\label{rmk:reeb}\rm
Let $\alpha_1,\alpha_2,\alpha_3$ be a hypercontact structure 
with Reeb vector fields $v_1,v_2,v_3$ and, for $\lambda\in S^2$,
denote 
$
v_\lambda:=\lambda_1v_1+\lambda_2v_2+\lambda_3v_3.  
$
Then $v_\lambda$ is the Reeb vector field of $\alpha_\lambda$. 
To see this note that 
\begin{equation}\label{eq:adaij}
\alpha_i(v_j)+\alpha_j(v_i)=0,\qquad
d\alpha_i(v_j,\cdot) + d\alpha_j(v_i,\cdot) = 0
\end{equation}
for $i\ne j$, by~\eqref{eq:hc} and Lemma~\ref{le:reeb}.
Hence $\alpha_\lambda(v_\lambda)=1$ and 
$d\alpha_\lambda(v_\lambda,\cdot)=0$. 
\end{remark}

\begin{lemma}\label{le:dual}
Let $\alpha$ be a hypercontact structure on $M$
with Reeb vector fields $v_1,v_2,v_3$.  Let
$\kappa:M\to\R$ be defined by~\eqref{eq:kappa} 
and $\mu:M\to\R^3$ by 
$$
\mu_1 := \alpha_2(v_3),\qquad
\mu_2 := \alpha_3(v_1),\qquad
\mu_3 := \alpha_1(v_2).
$$
Let $e_1,e_2,e_3$ denote the standard basis of $\R^3$.
Then the following holds.

\smallskip\noindent{\bf (i)}
The Lie brackets of the Reeb vector fields satisfy
\begin{equation}\label{eq:lie}
[v_2,v_3]=\kappa v_1,\qquad
[v_3,v_1]=\kappa v_2,\qquad 
[v_1,v_2]=\kappa v_3
\end{equation}
if and only if 
\begin{equation}\label{eq:murei}
d\mu(v_i) = \kappa e_i\wedge\mu,\qquad i=1,2,3.
\end{equation}

\smallskip\noindent{\bf (ii)}
If~\eqref{eq:lie} and~\eqref{eq:murei} hold then 
$\kappa$ is constant.  Conversely, if
$\kappa$ and $\mu$ are constant then $\mu\equiv 0$.

\smallskip\noindent{\bf (iii)}
The function $\mu$ vanishes if and only if 
$\alpha_i\wedge d\alpha_j=0$ for $i\ne j$, 
or equivalently $d\alpha_i=\kappa *\alpha_i$ for $i=1,2,3$. 
Here $*$ denotes the Hodge $*$-operator.
\end{lemma}

\begin{definition}\label{def:cartan}
A positive hypercontact structure $\alpha$ with $\mu\equiv0$ 
is called a {\bf Cartan structure}.
\end{definition}

\begin{corollary}\label{cor:cartan}
If $\alpha$ is a Cartan structure then $\kappa$ is constant, 
the $\alpha_i$ form the dual basis of the $v_i$,
the $v_i$ satisfy~\eqref{eq:lie}, 
$\alpha_i\wedge d\alpha_j=0$ for $i\ne j$,
and $d^*\alpha_i=0$ for $i=1,2,3$.
\end{corollary}

\begin{proof}[Proof of Lemma~\ref{le:dual}]
We introduce the $1$-form $\rho\in\Om^1(M,\R^3)$
and the vector fields $w_1,w_2,w_3\in\Vect(M)$ by
$$
\rho:=\frac{1}{\kappa}
\left(\begin{array}{c}
d\alpha_2(v_3,\cdot) \\
d\alpha_3(v_1,\cdot) \\
d\alpha_1(v_2,\cdot)
\end{array}\right),\qquad
\begin{array}{l}
w_1 := [v_2,v_3],\\
w_2 :=[v_3,v_1],\\
w_3 :=[v_1,v_2].
\end{array}
$$
Then $\rho$ satisfies
\begin{equation}\label{eq:hc1}
\rho_i(v_j)=\delta_{ij},\qquad 
\alpha(\xi) = \rho(\xi)+\rho(\xi)\wedge\mu.
\end{equation}
We also introduce the matrices
$$
A:=\left(\begin{array}{ccc}
\alpha_1(w_1) & \alpha_1(w_2) & \alpha_1(w_3) \\
\alpha_2(w_1) & \alpha_2(w_2) & \alpha_2(w_3) \\
\alpha_3(w_1) & \alpha_3(w_2) & \alpha_3(w_3)
\end{array}\right),\;
S := \left(\begin{array}{ccc}
\rho_1(w_1) & \rho_1(w_2) & \rho_1(w_3) \\
\rho_2(w_1) & \rho_2(w_2) & \rho_2(w_3) \\
\rho_3(w_1) & \rho_3(w_2) & \rho_3(w_3)
\end{array}\right),
$$
$$
\Phi:= \left(\begin{array}{ccc}
1 & \mu_3 & -\mu_2 \\
-\mu_3 & 1 & \mu_1 \\
\mu_2 & -\mu_1 & 1
\end{array}\right),\;
B:= \left(\begin{array}{ccc}
d\mu_1(v_1) & d\mu_1(v_2) & d\mu_1(v_3) \\
d\mu_2(v_1) & d\mu_2(v_2) & d\mu_2(v_3) \\
d\mu_3(v_1) & d\mu_3(v_2) & d\mu_3(v_3)
\end{array}\right).
$$
Then the second equation in~\eqref{eq:hc1} implies
$\Phi S=A$.  Next we observe that
$$
\alpha_i([v_j,v_k]) = d\alpha_i(v_j,v_k) 
- \cL_{v_j}\alpha_i(v_k)+ \cL_{v_k}\alpha_i(v_j).
$$
Hence
$
\alpha_i([v_j,v_k]) = \kappa + d\mu_j(v_j)+d\mu_k(v_k)
$
whenever $i,j,k$ is a cyclic permutation of $1,2,3$,
and
$
\alpha_i([v_i,v_k]) =  -\cL_{v_i}\alpha_i(v_k).
$
These identities can be summarized in the form
$
\alpha_i(w_j) + d\mu_j(v_i) 
= \left(\kappa+\sum_kd\mu_k(v_k)\right)\delta_{ij}
$
or 
\begin{equation}\label{eq:hc2}
\Phi S = A = \left(\kappa+\sum_kd\mu_k(v_k)\right)\one-B^T.
\end{equation}
Moreover, we have
\begin{eqnarray*}
0 &=& dd\alpha_1(v_1,v_2,v_3)  \\
&= & \cL_{v_1}d\alpha_1(v_2,v_3) 
+ \cL_{v_2}d\alpha_1(v_3,v_1) 
+ \cL_{v_3}d\alpha_1(v_1,v_2)  \\
&& 
-\, d\alpha_1(v_1,[v_2,v_3])
- d\alpha_1(v_2,[v_3,v_1])
- d\alpha_1(v_3,[v_1,v_2])  \\
&=& 
d\kappa(v_1) 
- d\alpha_1(v_2,[v_3,v_1])
- d\alpha_1(v_3,[v_1,v_2]).
\end{eqnarray*}
Repeating the argument for $\alpha_2$ and~$\alpha_3$ 
and using equation~\eqref{eq:adaij} we obtain
\begin{equation}\label{eq:hc3}
\begin{split}
d\kappa(v_1) &= \kappa(\rho_3(w_2)-\rho_2(w_3)),\\
d\kappa(v_2) &= \kappa(\rho_1(w_3)-\rho_3(w_1)),\\
d\kappa(v_3) &= \kappa(\rho_2(w_1)-\rho_1(w_2)),
\end{split}
\end{equation}
Hence $\kappa$ is constant if and only if the matrix
$S$ is symmetric. 

We prove~(i). Equation~\eqref{eq:lie} is equivalent 
to $S=\kappa\one$ and equation~\eqref{eq:murei}
to $B=\kappa(\Phi-\one)$.  If $S=\kappa\one$ then it follows
from~\eqref{eq:hc2} that 
$$
B^T = \kappa(\one-\Phi) + \sum_kd\mu_k(v_k)\one.
$$
Examining the diagonal entries 
we find that $d\mu_k(v_k)=0$
for $k=1,2,3$ and hence $B^T=\kappa(\one-\Phi)$. 
This in turn implies that $B^T=-B$ and thus $B=\kappa(\Phi-\one)$. 
Conversely, if $B=\kappa(\Phi-\one)$ then $B$ is skew symmetric 
and $d\mu_k(v_k)=0$ for all $k$.  So it follows from~\eqref{eq:hc2}
that $\Phi S = \kappa\one+B=\kappa\Phi$ and hence
$S=\kappa\one$.  

\smallbreak

We prove~(ii).  If~\eqref{eq:lie} holds
then $S=\kappa\one$ is symmetric and so $\kappa$ 
is constant, by~\eqref{eq:hc3}.  Conversely, if $\kappa$ 
and $\mu$ are constant then, by~\eqref{eq:hc2}, we have 
$\Phi S=\kappa\one$ and, by~\eqref{eq:hc3}, 
$S=S^T$.  Hence $\Phi$ is symmetric and so~$\mu\equiv0$.  

To prove~(iii) we observe that, 
for every cyclic permutation $i,j,k$ of $1,2,3$, we have
$\alpha_i\wedge d\alpha_j=\kappa\mu_k\dvol_M$
and ${\kappa*\alpha_i = d\alpha_i + \mu_kd\alpha_j-\mu_jd\alpha_k}$.
(Take the product with a $1$-form $\beta$
and use the identity $(\beta\wedge d\alpha_i)(v_1,v_2,v_3)=\kappa\beta(v_i)$.) 
This proves the lemma.
\end{proof}

\begin{example}\label{ex:S3}\rm
The standard hypercontact structure on the unit sphere 
$S^3\subset\R^4$ with coordinates $y=(y_0,y_1,y_2,y_3)$ 
is given by the $1$-forms
\begin{equation*}
\begin{split}
\alpha_1 &:= y_0dy_1-y_1dy_0 +y_2dy_3-y_3dy_2,\\
\alpha_2 &:= y_0dy_2-y_2dy_0 +y_3dy_1-y_1dy_3,\\ 
\alpha_3 &:= y_0dy_3-y_3dy_0 +y_1dy_2-y_2dy_1.
\end{split}
\end{equation*}
Identify $\R^4$ with the quaternions via 
$
y=y_0+\i y_1+\j y_2+\k y_3
$ 
and $\R^3$ with the imaginary 
quaternions via $\lambda=\i\lambda_1+\j\lambda_2+\k\lambda_3$. 
Then the $1$-form 
$
\alpha_\lambda :=\lambda_1\alpha_1+\lambda_2\alpha_2+ \lambda_3\alpha_3
$
and its Reeb vector field $v_\lambda$ are given by
$
\alpha_\lambda(y;\eta) = \RE(\lambda y\bar\eta)
$
and
$ 
v_\lambda(y) = \lambda y
$
for $\lambda\in S^2\subset\IM(\H)$ and $\eta\in T_yS^3$.
We emphasize that in this example $\mu\equiv0$ and $\kappa\equiv2$.

The standard hypercontact structure on $S^3$ is preserved 
by the right action of the unit quaternions via 
$\Sp(1)\times S^3\to S^3:(a,y)\mapsto ya$. 
For the left action of $\Sp(1)$ on $S^3$ we have
$
\phi_a^*\alpha_\lambda = \alpha_{a^{-1}\lambda a}
$
and
$
\phi_a^*v_\lambda  = v_{a^{-1}\lambda a},
$
where $\phi_a\in\Diff(S^3)$ is given by 
$\phi_a(y):=ay$ for $a\in\Sp(1)$ and ${y\in S^3}$. 
\end{example}

\begin{proposition}[Geiges--Gonzalo~\cite{GG,GG1}]
\label{prop:GG}
Every Cartan hypercontact $3$-manifold $(M,\alpha)$ 
is diffeomorphic to a quotient of the $3$-sphere 
(with the standard hypercontact structure up to scaling) 
by a finite subgroup of $\Sp(1)$.
\end{proposition}

\begin{proof}
By rescaling, if necessary, we may assume that $\kappa=2$. 
Then there is a unique Lie algebra homomorphism 
$
\Lie(\Sp(1))=\IM(\H)\to\Vect(M):\xi\mapsto v_\xi
$
such that $v_\i,v_\j,v_\k$ are the Reeb vector fields 
of $\alpha_1,\alpha_2,\alpha_3$, respectively. 
Since $M$ is compact and $\Sp(1)$ is simply connected,
this Lie algebra homomorphism integrates to a unique 
Lie group homomorphism 
$$
\Sp(1)\to\Diff(M):x\mapsto\phi_x.
$$
This group action of $\Sp(1)$ on $M$ is transitive,
because $M$ is connected, and it has finite isotropy 
subgroups. Fix an element $y_0\in M$ and define the map
${\psi:\Sp(1)\to M}$ by $\psi(x):=\phi_x(y_0)$.
This map induces a diffeomorphismm $\Sp(1)/\G_0\to M$, 
where $\G_0:=\left\{x\in\Sp(1)\,|\,\phi_x(y_0)=y_0\right\}$
denotes the stabilizer of $y_0$.  This diffeomorphism 
identifies the vector field $x\mapsto\i x$ on $\Sp(1)/\G_0$
with the vector field $v_\i$ on $M$ and similarly 
for $\j$ and $\k$. 
\end{proof}

 
\section{The Heinz trick for subcritical exponents}\label{app:heinz}   

Let $M$ be a smooth Riemannian $n$-manifold (not necessarily compact)
and let $\sL$ be a scalar second order elliptic operator.
We assume that $\sL$ differs from the Laplace Beltrami operator
$\Delta:=-d^*d$ by a first order operator.  We study nonnegative 
solutions $e:M\to\R$ of the differential inequality
\begin{equation}\label{eq:heinz}
\sL e \ge - A - B e^\mu
\end{equation}
where 
$$
1\le \mu \le \frac{n+2}{n}.
$$
In the critical case $\mu=(n+2)/n$ the Heinz trick gives 
a mean value inequality for nonnegative solutions 
$e:B_r(p_0)\to[0,\infty)$ of~\eqref{eq:heinz}  with 
sufficiently small $L^1$ norm (see for example~\cite{RS3,W}).   
For $\mu<(n+2)/n$ the same proof shows that the
condition on the $L^1$ norm can be dropped and one obtains
a global estimate for the sup-norm in terms of the 
$L^1$ norm of $e$. 

\begin{theorem}\label{thm:heinz}
Let $K\subset M$ be a compact set and let
$1\le\mu\le(n+2)/n$. 

\smallskip\noindent{\bf (i)}
Assume $\mu<(n+2)/n$. Then there is a constant $c>0$ 
with the following significance.
If $e:M\to\R$ is a nonnegative $C^2$ function 
satisfying~\eqref{eq:heinz} then 
\begin{equation}\label{eq:heinz1}
\sup_K e \le c\left(A + \int_Me\,\dvol_M
+ \left(B^{n/2}\int_Me\,\dvol_M\right)^\alpha\right),
\end{equation}
where $\alpha:=2/(2+n-n\mu)$.

\smallskip\noindent{\bf (ii)}
Assume $\mu=(n+2)/n$. Then there are positive constants 
$\hbar,\delta,c$ with the following significance.
If $e:M\to\R$ is a nonnegative $C^2$ function 
satisfying~\eqref{eq:heinz} then, for $x\in K$ 
and $0<r\le\delta$, we have
\begin{equation}\label{eq:hein2}
B^{n/2}\int_{B_r(x)}e < \hbar\quad\implies\quad
e(x) \le c\left(Ar^2 + \frac{1}{r^n}\int_{B_r(x)}e\,\dvol_M\right).
\end{equation}
\end{theorem}

\begin{proof}
The proof has three steps.  For $p_0\in M$ and $r>0$ we denote 
by $B_r(p_0)\subset M$ the closed ball of radius $r$ about $p_0$. 
The first step restates Theorem~9.20 in~\cite{GT}.

\medskip\noindent{\bf Step~1.}
{\it There are constants $c_1>0$ and $\delta>0$ with 
the following significance.  If $p_0\in K$ and 
$0<r\le\delta$ then every nonnegative $C^2$ 
function $e:B_r(p_0)\to\R$ satisfies}
$$
\Delta e\ge 0\quad\implies\quad 
e(p_0) \le \frac{c_1}{r^n}\int_{B_r(p_0)}e\,\dvol_M.
$$

\medskip\noindent

\medskip\noindent{\bf Step~2.}
{\it There are constants $c_2>0$ and $\delta>0$ with 
the following significance. If $p_0\in K$, $0<r\le\delta$, 
and $A\ge 0$, then every nonnegative $C^2$
function $e:B_r(p_0)\to\R$ satisfies}
$$
\sL e\ge -A\quad\implies\quad 
e(p_0) \le c_2\left(A r^2+\frac{1}{r^n}\int_{B_r(p_0)}e\,\dvol_M\right).
$$

\medskip\noindent
Let $\delta$ be smaller than the injectivity radius of $M$ and than
the constant in Step~1. Choose geodesic coordinates 
$y^1,y^2,\dots,y^n$ in $B_r(y_0)$ with $y^i(p_0)=0$.  
Then 
$$
\sL = \sum_{\mu,\nu}a^{\mu\nu}\frac{\p^2}{\p y^\mu\p y^\nu} 
+ \sum_\nu b^\nu\frac{\p}{\p y^\nu}
$$
with $a^{\mu\nu}(0)=\delta^{\mu\nu}$. Choose $\delta$ so small that
$$
\Abs{y}\le\delta\qquad\implies\qquad
\Abs{a^{\nu\nu}(y)-1} + \delta\Abs{b^\nu(y)}\le\frac{1}{n}
$$
for $\nu=1,\dots,n$. Denote by $\Delta_0=\sum_\nu(\frac{\p}{\p y^\nu})^2$
the standard Laplace operator and consider the function
$$
u(y) := \frac{A}{2}\Abs{y}^2.
$$
This function satisfies 
$
\Delta_0u=nA
$
and 
$$
(\sL u-\Delta_0u)(y)
= A\left(\sum_\nu(a^{\nu\nu}(y)-1)
+ \sum_\nu b^\nu(y)y^\nu\right)
\ge -A.
$$
Hence $\sL u\ge 2A$ and 
$$
\sL(e+u)\ge \sL e + (n-1)A \ge 0.
$$
By Step~1, this implies
$$
e(0) = e(0)+u(0)
\le \frac{c_1}{r^n}\int_{B_r}(e+u)\dvol_M.
$$
Hence the assertion follows from the fact that
$$
\int_{B_r}u\,\dvol_M \le \om_n A\int_0^r\rho^{n+1}\,d\rho
= \frac{\om_n A r^{n+2}}{n+2}.
$$
Here $\om_n$ denotes the area of the unit sphere in $\R^n$
and $\delta$ is chosen so small that $\dvol_M$ and the volume 
form of the flat metric differ by a factor at most~$2$.  
This proves Step~2.

\medskip\noindent{\bf Step~3.}
{\it There is a constant $c_3>0$ with the following 
significance. If $e:M\to\R$ is a nonnegative $C^2$ 
function satisfying
$$
\sL e\ge -A-Be^\mu
$$
for some constants $A,B\ge 0$ then}
$$
\sup_K e \le c_3\left(A + \int_Me\,\dvol_M
+ \left(
B^{n/2}\int_Me\,\dvol_M\right)^{2/(2+n-n\mu)}
\right).
$$

\medskip\noindent
Let $\delta$ be as in Step~2 and assume $c_2\delta^2\le\frac{1}{4}$.  
Fix a point $p_0\in K$. Define ${h:[0,\delta]\to\R}$ by 
$$
h(s) := \left(\frac{\delta-s}{\delta}\right)^n\max_{B_s(p_0)}e.
$$
Then 
$$
h(0)=e(p_0),\qquad h(\delta)=0. 
$$
Since $h$ is nonnegative there is an 
$s^*\in[0,\delta)$ and a $p^*\in B_{s^*}(p_0)$ such that
$$
h(s^*) = \max_{0\le s\le\delta}h(s),\qquad
c:=e(p^*) = \max_{B_{s^*}(p_0)}e.
$$
Denote 
$$
\eps := \frac{\delta-s^*}{2}.
$$
Then 
$$
\max_{B_\eps(p^*)}e \le\max_{B_{s^*+\eps}(p_0)}e
= \frac{\delta^nh(s^*+\eps)}{\left(\delta-s^*-\eps\right)^n}
\le \frac{2^n\delta^nh(s^*)}{\left(\delta-s^*\right)^n}
= 2^n\max_{B_{s^*}(p_0)}e
= 2^nc.
$$
Hence in $B_\eps(p^*)$ we have the inequality
$$
\sL e \ge - A-Be^\mu \ge -A - B(2^nc)^\mu.
$$
By Step~2 this implies 
\begin{equation}\label{eq:cr}
c = e(p^*)
\le c_2\left((A+B(2^nc)^\mu) r^2+\frac{1}{r^n}\int_Me\,\dvol_M\right)
\end{equation}
for $0\le r\le\eps$.   Now comes the crucial case distinction.

\smallbreak

\smallskip\noindent{\bf Case~1.}
If $c\le A$ then we have 
$$
e(p_0)\le c\le A
$$ 
and so the desired estimate holds with $c_3=1$.  
Thus we may assume $A\le c$. 

\smallskip\noindent{\bf Case~2.}
Assume 
$$
A\le c,\qquad c_2B2^{n\mu} c^{\mu-1}\eps^2\ge\frac{1}{4}.
$$
Then we may choose $r\le\eps<\delta$ such that 
$c_2B2^{n\mu}c^{\mu-1}r^2=\frac{1}{4}$ and obtain
$$
c_2(A+B(2^nc)^\mu) r^2
\le c_2c\delta^2+c_2B(2^nc)^\mu r^2\le \frac{c}{2}
$$
Hence,  by~\eqref{eq:cr}, we have
$$
c\le \frac{2c_2}{r^n}\int_Me\,\dvol_M
= 2c_2(4c_2B2^{n\mu})^{n/2}c^{(n\mu-n)/2}\int_Me\,\dvol_M.
$$
Since $\mu<(n+2)/n$ we have $2+n-n\mu>0$ and hence
$$
e(p_0)\le c\le c_3\left(B^{n/2}\int_Me\,\dvol_M\right)^{2/(2+n-n\mu)},
$$
with $c_3:=\left(2c_2(c_22^{n\mu+2})^{n/2}\right)^{2/(2+n-n\mu)}$.
(For the critical exponent we have $(n\mu-n)/2=1$.  In this situation
Case~2 can be excluded by the assumption of a sufficiently 
small upper bound on $B^{n/2}\int e\,\dvol_M$.)

\smallskip\noindent{\bf Case~3.}
Assume 
$$
A\le c,\qquad c_2B2^{n\mu}c^{\mu-1}\eps^2 < \frac{1}{4}.
$$ 
Then we may choose $r=\eps$ and obtain
$c_2(A+B(2^nc)^\mu)\eps^2\le \frac{c}{2}$ as before.  Hence, by~\eqref{eq:cr},
we have
$$
c\le \frac{2c_2}{\eps^n}\int_Me\,\dvol_M.
$$
Since $\delta-s^*=2\eps$ this gives
$$
e(p_0) = h(0) \le h(s^*)
= c\left(\frac{\delta-s^*}{\delta}\right)^n 
= \frac{2^nc\eps^n}{\delta^n} \le \frac{2^{n+1}c_2}{\delta^n}\int_Me\,\dvol_M.
$$
Thus in this case the estimate of Step~3 holds with $c_3=2^{n+1}c_2/\delta^n$. 
This proves the theorem.
\end{proof}

    
\section{A removable singularity theorem} \label{app:remsing}

Denote by $B\subset\R^4$ the unit ball
with coordinates $t=(t_0,t_1,t_2,t_3)$ and by 
$$
B_r := \left\{t\in\R^4\,|\,\Abs{t}\le r\right\},\qquad
S_r := \left\{t\in\R^4\,|\,\Abs{t}= r\right\},
$$ 
the ball and sphere of radius $r$.  
Let $X$ be a hyperk\"ahler manifold with complex structures $I,J,K$.
Let $w:B\to\Vect(X)$ and $\Xi=(\xi_i^j):B\to\R^{4\times 4}$ be smooth maps
such that $\Xi(0)=\one$ is the identity matrix and $\Xi(t)$ is nonsingular for 
every $t\in B$.  We examine solutions of the equation
\begin{equation}\label{eq:local}
\sum_{i=0}^3 
\bigl(\xi_0^i(t)\p_iu
+ \xi_1^i(t)I\p_iu
+ \xi_2^i(t)J\p_iu
+ \xi_3^i(t)K\p_iu\bigr) 
= \nabla w(t,u).
\end{equation}
Associated to equation~\eqref{eq:local} is the elliptic operator
$$
L := \sum_{i,j=0}^k a^{ij}\p_i\p_j + \sum_{j=0}^3b^j\p_j,\qquad
a^{ij}:=\sum_\nu\xi_\nu^i\xi_\nu^j,\qquad
b^j:=\sum_{\nu,i}(\p_i\xi_\nu^j)\xi_\nu^i.
$$

\begin{theorem}\label{thm:remsing}
Assume $X$ is a compact flat hyperk\"ahler 
manifold (possibly with boundary).  If $u:B\setminus\{0\}\to X$ 
is a solution of~\eqref{eq:local} on the punctured disc~and
$$
\int_B\Abs{du}^2 = \sum_{i=0}^3\int_B\Abs{\p_iu}^2 < \infty
$$
then $u$ extends to a smooth function from $B$ to $X$. 
\end{theorem}

\begin{remark}\label{rmk:remsing}\rm
In Theorem~\ref{thm:remsing} the condition that $X$ is flat 
cannot be omitted.  For example, let $f:S^3\to X$ 
be a nonconstant critical point of the 
hypersymplectic action functional $\sA$. 
Such critical points are described in the introduction
(compositions of rational curves with Hopf fibrations)
and they do not exist in the flat case, by Lemma~\ref{le:apriori}. 
Identify $S^3$ with the unit sphere in $\H$ and define 
$u:\H\setminus\{0\}\to X$ by 
$$
u(t) := f(\Abs{t}^{-1}t).
$$
Then $u$ satisfies the equation
$$
\p_0u - I\p_1u - J\p_2u - K\p_3u = 0.
$$
Moreover, we have
$
\Abs{du(t)}^2 = \Abs{t}^{-2}
\abs{df(\abs{t}^{-1}t)}^2
$
and hence 
$$
\int_{B_r} \Abs{du}^2 = \frac{r^2}{2}\int_{S^3}\Abs{df}^2 
= r^2\sA(f)
$$
for every $r>0$. However, the singularity of $u$ at the origin
cannot be removed. 
\end{remark}

\begin{lemma}\label{le:Le}
Assume $X$ is a compact flat hyperk\"ahler manifold.
Then there is a constant $C>0$ with the following significance.
If $u:B\setminus\{0\}\to X$ is a solution of~\eqref{eq:local} 
then the function $e=e_u:B\to[0,\infty)$ defined by 
$$
e(t) := \frac12\sum_{j=0}^3\Abs{\sum_{i=0}^3\xi_j^i(t)\p_iu(t)}^2
$$
satisfies the inequality
$$
Le \ge -C(1+e^{3/2}). 
$$
\end{lemma}

\begin{proof}
The proof uses word by word the same arguments as in 
Lemma~\ref{le:Ler} and will be omitted.
\end{proof}

The exponent $\frac{3}{2}=\frac{n+2}{n}$ in Lemma~\ref{le:Le} is the critical 
exponent of Theorem~\ref{thm:heinz} for $n=4$.  Hence  
every solution $u:B\setminus\{0\}\to X$ of~\eqref{eq:local} 
satisfies an inequality of the form 
\begin{equation}\label{eq:r2r}
\Abs{t}=r\qquad\implies\qquad
\Abs{du(t)}^2\le cr^2+\frac{c}{r^4}\int_{B_{2r}}\Abs{du}^2
\end{equation}
for $r$ sufficiently small and a suitable constant $c$.  Thus
$\Abs{t}^4\Abs{du(t)}^2$ converges to zero as $t$ tends to zero.

It is convenient to introduce the $1$-forms $\theta_1,\theta_2,\theta_3$ 
and the vector fields $v_0,v_1,v_2,v_3$ on $B$ by 
\begin{equation*}
\begin{split}
\theta_1 &:= t_0dt_1 - t_1dt_0 - t_2dt_3 + t_3dt_2,\\
\theta_2 &:= t_0dt_2 - t_2dt_0 - t_3dt_1 + t_1dt_3,\\
\theta_3 &:= t_0dt_3 - t_3dt_0 - t_1dt_2 + t_2dt_1, \\
v_0 &:= t_0\p_0 + t_1\p_1 + t_2\p_2 + t_3\p_3,\\
v_1 &:= t_0\p_1 - t_1\p_0 - t_2\p_3 + t_3\p_2,\\
v_2 &:= t_0\p_2 - t_2\p_0 - t_3\p_1 + t_1\p_3,\\
v_3 &:= t_0\p_3 - t_3\p_0 - t_1\p_2 + t_2\p_1.
\end{split}
\end{equation*}
Note that the $v_i$ are orthogonal and $\Abs{v_i(t)}=\Abs{t}$.
In particular, for $t\in S_r$ the vectors 
$r^{-1}v_1(t),r^{-1}v_2(t),r^{-1}v_3(t)$ form an 
orthonormal basis of the tangent space $T_tS_r=t^\perp$.
The energy and the hypersymplectic action of a smooth map
$u:S_r\to X$ are defined by 
$$
\sE_r(u) := \frac1{r^2}\int_{S_r}\sum_{i=1}^3\Abs{du(v_i)}^2,\qquad
\sA_r(u) := \int_{S_r}\sum_i\theta_i\wedge u^*\om_i.
$$

\begin{lemma}\label{le:EA}
The energy and hypersymplectic action satisfy the isoperimetric inequality
\begin{equation}\label{eq:isoperimetric}
\sA_r(u)\le r\sE_r(u)
\end{equation}
and the energy identities
\begin{equation}\label{eq:ArE}
\sE_r(u) + \frac{2}{r}\sA_r(u) 
= \frac{1}{r^2}\int_{S_r}\Abs{Idu(v_1)+Jdu(v_2)+Kdu(v_3)}^2\,\dvol_{S_r}
\end{equation}
for every smooth map $u:S_r\to X$ and
\begin{equation}\label{eq:energy-action}
\int_{B_r}\Abs{du}^2
= \sA_r(u) + \int_{B_r}\Abs{\p_0u+I\p_1u+J\p_2u+K\p_3u}^2
\end{equation}
for every smooth map $u:B_r\setminus\{0\}\to X$ 
satisfying $\lim_{t\to0}\Abs{t}^4\Abs{du(t)}^2=0$.
\end{lemma}

\begin{proof}
We have $\theta_i(v_j)=r^2\delta_{ij}$ 
and so the standard volume form on $S_r$ is 
$\dvol_{S_r}=r^{-3}\theta_1\wedge\theta_2\wedge\theta_3$.
Hence $\theta_i\wedge u^*\om_i=r^{-1}u^*\om_i(v_j,v_k)\dvol_{S_r}$
for every cyclic permutation $i,j,k$ of $1,2,3$.  This implies
$$
\sA_r(u) = \frac{1}{r}\int_{S_r}
\Bigl(u^*\om_1(v_2,v_3) 
+ u^*\om_2(v_3,v_1) + u^*\om_3(v_1,v_2)
\Bigr)\,\dvol_{S_r}
$$
and hence the isoperimetric inequality~\eqref{eq:isoperimetric}.
The energy identity~\eqref{eq:ArE} is an adaptation of Lemma~\ref{le:acten}
to the present notation.  To prove~\eqref{eq:energy-action} we 
assume that $u:B_r\setminus\{0\}\to X$ satisfies 
$
\lim_{t\to0}\Abs{t}^4\Abs{du(t)}^2=0.
$
Then it follows from~\eqref{eq:isoperimetric} that
$
\lim_{\rho\to 0}\sA_\rho(u) = 0.
$
Moreover, by direct computation, we have
$$
\int_{B_r\setminus B_\rho}\Bigl(
\Abs{du}^2 - \Abs{\p_0u+I\p_1u+J\p_2u+K\p_3u}^2
\Bigr)
= \sA_r(u) - \sA_\rho(u)
$$
for $0<\rho\le r$. The assertion follows by taking
the limit $\rho\to 0$.  This proves the lemma. 
\end{proof}

\begin{lemma}\label{le:BS}
Assume $X$ is compact and fix any real number $0<\mu<4$. 
Let $u:B\setminus\{0\}\to X$ be a solution of~\eqref{eq:local}
satisfying 
$
\lim_{t\to0}\Abs{t}^4\Abs{du(t)}^2=0.
$
Then there are positive constants $r_0$ and $c$ such that 
$$
0<r\le r_0\qquad\implies\qquad
\int_{B_r}\Abs{du}^2 \le cr^\mu.
$$ 
\end{lemma}

\begin{proof}
Since $\Xi(0)$ is the identity matrix, there is a constant $C>0$ 
such that every solution of~\eqref{eq:local} 
satisfies the estimate
\begin{equation}\label{eq:ineq}
\Abs{\p_0u(t)+I\p_1u(t)+J\p_2u(t)+K\p_3u(t)}^2 
\le C^2(\Abs{t}^2\Abs{du(t)}^2 + 1)
\end{equation}
Combining this with~\eqref{eq:energy-action} we obtain
\begin{eqnarray}\label{eq:Bdu}
\int_{B_r}\Abs{du}^2
&=& 
\sA_r(u) + \int_{B_r}\Abs{\p_0u+I\p_1u+J\p_2u+K\p_3u}^2 
\nonumber \\
&\le&
\sA_r(u) + C^2r^2\int_{B_r}\Abs{du}^2  +  C^2\Vol(B)r^4.
\end{eqnarray}
Since
$$
\Abs{du(v_0)+Idu(v_1)+Jdu(v_2)+Kdu(v_3)}
= r\Abs{\p_0u+I\p_1u+J\p_2u+K\p_3u}
$$
on $S_r$ and $r^2\Abs{du}^2+1\le(r\Abs{du}+1)^2$, 
it follows also from~\eqref{eq:ineq} that 
\begin{equation*}
\begin{split}
\Abs{du(v_0)}
&\ge
\Abs{Idu(v_1)+Jdu(v_2)+Kdu(v_3)}
- Cr^2\Abs{du} - Cr, \\
\Abs{du(v_0)}^2 
&\ge 
\Abs{Idu(v_1)+Jdu(v_2)+Kdu(v_3)}^2
- 6Cr^3\Abs{du}^2 - 6Cr^2\Abs{du}.
\end{split}
\end{equation*}
This implies 
\begin{eqnarray*}
\int_{S_r}\Abs{du}^2
&=& 
\frac{1}{r^2}\int_{S_r}\sum_{i=0}^3\Abs{du(v_i)}^2  
=
\sE_r(u) + \frac{1}{r^2}\int_{S_r}\Abs{du(v_0)}^2  \\
&\ge&
\sE_r(u) + \frac{1}{r^2}\int_{S_r}\Abs{Idu(v_1)+Jdu(v_2)+Kdu(v_3)}^2  \\
&& -\, 6Cr\int_{S_r}\Abs{du}^2 - 6C\int_{S_r}\Abs{du}   \\
&\ge&
2\sE_r(u) + \frac{2}{r}\sA_r(u) 
- 6C(r+\delta)\int_{S_r}\Abs{du}^2 - \frac{3C}{2\delta}\int_{S_r}1.
\end{eqnarray*}
Here we have dropped the volume form $\dvol_{S_r}$
in the notation.  The last step follows from~\eqref{eq:ArE}.
Since $\sE_r(u)\ge r^{-1}\sA_r(u)$ and the area of the
$3$-sphere is $4\Vol(B)$ this gives
$$
\bigl(1+6C(r+\delta)\bigr)\int_{S_r}\Abs{du}^2 
\ge \frac{4}{r}\sA_r(u) - \frac{6C\Vol(B)r^3}{\delta}.
$$
On the other hand, by~\eqref{eq:Bdu} we have
$$
\bigl(1-C^2r^2\bigr)\int_{B_r}\Abs{du}^2
\le \sA_r(u) +  C^2\Vol(B)r^4.
$$
Combining these two inequalities we obtain
$$
\int_{B_r}\Abs{du}^2 
\le\frac{1+6C(r+\delta)}{1-C^2r^2} \frac{r}{4}\int_{S_r}\Abs{du}^2
+ \left(\frac{C^2}{1-C^2r^2} + \frac{3C}{2\delta}\right)\Vol(B)r^4
$$
for $r<1/C$.  Choose $\delta$ so small that $(1+6C\delta)\mu<4$. 
Then, for $r$ sufficiently small and a suitable constant 
$c>0$, we have 
\begin{equation}\label{eq:duBS}
\int_{B_r}\Abs{du}^2 \le r\mu^{-1}\int_{S_r}\Abs{du}^2 + cr^4.
\end{equation}
Define the function $\phi:(0,1]\to\R$ by 
$$
\phi(r) := r^{-\mu}\int_{B_r}\Abs{du}^2 + \frac{\mu c}{4-\mu}r^{4-\mu}.
$$
Then the derivative of $\phi$ is
\begin{eqnarray*}
\frac{d}{dr}\phi(r) 
&=&
r^{-\mu}\int_{S_r}\Abs{du}^2
- \mu r^{-\mu-1}\int_{B_r}\Abs{du}^2
+ \mu cr^{3-\mu}  \\
&=&
\mu r^{-\mu-1}
\left(r\mu^{-1}\int_{S_r}\Abs{du}^2
- \int_{B_r}\Abs{du}^2
+ cr^4\right)  \ge 0.
\end{eqnarray*}
The last inequality follows from~\eqref{eq:duBS} and holds for 
$r$ sufficiently small, say for $0<r\le r_0$.  Hence
$$
\int_{B_r}\Abs{du}^2 \le \phi(r)r^\mu \le \phi(r_0)r^\mu
$$
for $0< r\le r_0$. This proves the lemma.
\end{proof}

\begin{proof}[Proof of Theorem~\ref{thm:remsing}]
Choose a real number $\mu$ such that $2<\mu<4$.  
Combining Lemma~\ref{le:BS} with the the 
inequality~\eqref{eq:r2r} we obtain
$$
\Abs{du(t)}^2 \le \frac{c}{\Abs{t}^{4-\mu}}
$$
for a suitable constant $c>0$.  For $4<p<8/(4-\mu)$ this implies
$$
\int_B\Abs{du}^p 
= \int_0^1\int_{S_r}\Abs{du}^p
\le 4\Vol(B)c^p\int_0^1r^{3-(4-\mu)p/2} \,dr
<\infty.
$$
That the integral is finite follows from the fact that $3-\frac{1}{2}(4-\mu)p>-1$.
By the Sobolev embedding theorem our function $u:B\setminus\{0\}\to X$
is H\"older continuous and extends to a $W^{1,p}$ function on $B$.
Now it follows from the standard elliptic bootstrapping techniques
that the extended function $u$ is smooth.  This proves the theorem.
\end{proof}


\end{document}